\documentclass{amsart}
\usepackage{amsthm,stmaryrd}
\usepackage{amssymb,enumerate}
\usepackage{epsfig}
\usepackage[usenames,dvipsnames]{color}
\usepackage{verbatim}
\usepackage{hyperref}
\usepackage{mathrsfs}
\usepackage[all]{xy}
\newcommand{\T}{{\mathcal T}}
\newcommand{\mcg}{\mathsf {MCG}}

\newcommand{\col}{\mathsf{Col}}
\newcommand{\Cob}{\mathcal{C}ob}

\newcommand{\GrVect}{\mathcal{G}r\mathcal{V}ect}

\newcommand{\card}{\operatorname{card}}
\newcommand{\B}{\mathcal{B}}
\newcommand{\lk}{\operatorname{lk}}
\newcommand{\ev}{\operatorname{ev}}
\newcommand{\SL}{\mathsf{SL}}

\newcommand{\tcoev}{\wt{\operatorname{coev}}}
\newcommand{\brk}[1]{{\left\langle{#1}\right\rangle}}
\newcommand{\brkk}[1]{{\left\langle\langle{#1}\right\rangle\rangle}}

\newcommand{\ve}{\varepsilon}
\newcommand{\ro}{r}

\newcommand{\Uqg}{\ensuremath{\mathcal U}}
\newcommand{\La}{{\mathscr L}}
\newcommand{\Nr}{{\mathsf N}}
\newcommand{\Zr}{{\mathsf Z}}
\newcommand{\PP}{{\mathsf P}}
\newcommand{\Pac}{\Pa^\sqcup}
\newcommand{\Paca}{\Pa^\sqcap}

\newcommand{\Hr}{H_r}
\newcommand{\qr}{{q}}
\newcommand{\coh}{\omega}
\newcommand{\vp}{\varphi}
\newcommand{\cV}{\mathscr{V}}
\newcommand{\cF}{\mathscr{F}}
\newcommand{\V}{\mathsf{V}}
\newcommand{\VV}{\mathbb{V}}
\newcommand{\VVV}{\mathbb{SV}}
\newcommand{\hS}{\wh{S}}
\newcommand{\Su}{\Sigma}
\newcommand{\Skein}{\mathsf{S}}
\newcommand{\Pa}{\mathsf{P}}

\newcommand{\e}{{\operatorname{e}}}
\newcommand{\slt}{{\mathfrak{sl}(2)}}
\newcommand{\Uq}{{U_q\slt}}
\newcommand{\UsltH}{{U_q^{H}\slt}}
\newcommand{\Ubar}{{\wb U_q^{H}\slt}}
\newcommand{\unit}{\ensuremath{\mathbb{I}}}
\newcommand{\cat}{\mathscr{C}}
\newcommand{\Rib}{{{\operatorname{Rib}}_\mathscr{C}}}

\newcommand{\Id}{\operatorname{Id}}
\newcommand{\bp}[1]{{\left(#1\right)}}
\newcommand{\bc}[1]{{\left[#1\right]}}
\newcommand{\qn}[1]{{\left\{#1\right\}}}
\newcommand{\ds}[1]{{\mathcal{#1}}}
\newcommand{\degr}{{\mathbf k}}
\newcommand{\de}{{k}}
\newcommand{\qd}{{\mathsf d}}
\newcommand{\End}{\operatorname{End}}
\newcommand{\Hom}{\operatorname{Hom}}
\newcommand{\C}{\ensuremath{\mathbb{C}} }
\newcommand{\Z}{\ensuremath{\mathbb{Z}} }
\newcommand{\R}{\ensuremath{\mathbb{R}} }
\newcommand{\N}{\ensuremath{\mathbb{N}} }
\newcommand{\Q}{\ensuremath{\mathbb{Q}} }
\newcommand{\wt}{\widetilde}
\newcommand{\wh}{\widehat}
\newcommand{\wb}{\overline}
\newcommand{\wbU}{U}
\newcommand{\ms}[1]{\mbox{\tiny$#1$}}
\newcommand{\oo}{\infty}
\newcommand{\Cp}{{\ddot\C}}
\newcommand{\symf}{\rho}
\newcommand{\dep}{\delta} 
\newcommand{\et}{{\quad\text{and}\quad}}
\newcommand{\lw}{\operatorname{lw}}
\newcommand{\hw}{\operatorname{hw}}

\newcommand{\qx}[2]{\frac{\qn{#1}}{\qn{#2}}}
\newcommand{\DT}{{\tau}}

\newcommand{\qt}{\mathsf t} 

\newcommand{\Ma}{f_k} 

\newcommand{\sjtop}[6]{\left|\begin{array}{ccc}#1 & #2 & #3 \\#4 & #5 &
      #6\end{array}\right|} 
\newcommand{\epsh}[2]
         {\begin{array}{c} \hspace{-1.3mm}
        \raisebox{-4pt}{\epsfig{figure=#1,height=#2}}
        \hspace{-1.9mm}\end{array}}

\newtheorem{theo}{Theorem}[section]
\newtheorem{lemma}[theo]{Lemma}
\newtheorem{prop}[theo]{Proposition}
\newtheorem{cor}[theo]{Corollary}
\newtheorem{conj}[theo]{Conjecture}
\newtheorem{question}[theo]{Question}

\theoremstyle{definition}
\newtheorem{defi}[theo]{Definition}
\newtheorem{rem}[theo]{Remark}
\newtheorem{example}[theo]{Example}
\newtheorem{exo}[theo]{Exercise}

\theoremstyle{remark}

\newcounter{exo} \newcounter{numexercice}
\renewcommand{\theexo}{\arabic{exo}}

\begin{document}
\title[Non semi-simple TQFTs]{Non semi-simple TQFTs, Reidemeister torsion and Kashaev's invariants}
\author[C. Blanchet]{Christian Blanchet}
\address{Univ Paris Diderot, Sorbonne Paris Cit\'e, IMJ-PRG, UMR 7586 CNRS, Sorbonne Universit\'e, UPMC Univ Paris 06, F-75013, Paris, France} 
\email{blanchet@imj-prg.fr}

\author[F. Costantino]{Francesco Costantino}
\address{Institut de Recherche Math\'ematique Avanc\'ee\\
  Rue Ren\'e Descartes 7\\
 67084 Strasbourg, France}
\email{costanti@math.unistra.fr}

\author[N. Geer]{Nathan Geer}
\address{Mathematics \& Statistics\\
  Utah State University \\
  Logan, Utah 84322, USA}
\thanks{The first author's research was supported by 
French ANR project ModGroup ANR-11-BS01-0020. The second author's research was supported by 
French ANR project ANR-08-JCJC-0114-01. Research of the third author was  partially supported by 
NSF grants  DMS-1007197 and DMS-1308196.  The  third author would like to thank 
Institut Math\'ematique de Jussieu, Paris, France for their generous hospitality during April of 2013.
  All the authors would like to thank  the Erwin Schr\"odinger Institute for Mathematical Physics in Vienna for support during a stay in the Spring of 2014, where part of this work was done.
 }\
\email{nathan.geer@gmail.com}

\author[B. Patureau-Mirand]{Bertrand Patureau-Mirand}
\address{UMR 6205, LMBA, universit\'e de Bretagne-Sud, universit\'e
  europ\'eenne de Bretagne, BP 573, 56017 Vannes, France }
\email{bertrand.patureau@univ-ubs.fr}

\begin{abstract}
  We construct and study a new family of TQFTs based on nilpotent
  highest weight representations of quantum $\slt$ at a root of unity
  indexed by generic complex numbers. This extends to cobordisms the
  non-semi-simple invariants defined in \cite{CGP} including the
  Kashaev invariant of links.  Here the modular category framework
  does not apply and we use the ``universal construction''.  Our TQFT
  provides a monoidal functor from a category of surfaces and their
  cobordisms into the category of graded finite dimensional vector
  spaces and their degree $0$-morphisms and depends on the choice of a
  root of unity of order $2r$. The functor is always symmetric
  monoidal but for even values of $r$ the braiding on GrVect has to be
  the super-symmetric one, thus our TQFT may be considered as a
  super-TQFT.  In the special case $r=2$ our construction yields a
  TQFT for a canonical normalization of Reidemeister torsion and we
  re-prove the classification of Lens spaces via the non-semi-simple
  quantum invariants defined in \cite{CGP}.  We prove that the
  representations of mapping class groups and Torelli groups resulting
  from our constructions are potentially more sensitive than those
  obtained from the standard Reshetikhin-Turaev functors; in
  particular we prove that the action of the bounding pairs generators
  of the Torelli group has always infinite order.
\end{abstract}

\maketitle
\setcounter{tocdepth}{3}

\section{Introduction}

In 1988 E. Witten \cite{Wi} introduced the notion of  Topological Quantum Field
Theory, with  path integrals of the Chern-Simons action as motivating example. This was the starting point of a fascinating interaction
between mathematics and theoretical physics,
deeply relating many domains such as knot theory, Von Neumann algebras,
 Hopf algebras, Lie algebras and quantum groups,
Chern-Simons theory, conformal field theory.

 A Topological Quantum Field Theory includes
topological invariants of $3$-dimensional manifolds.
 Following Witten's ideas, a rigorous construction of
such invariants was first obtained
from knot theory and quantum groups by Reshetikhin and Turaev  \cite{RT}. Then
  Blanchet, Habegger, Masbaum and Vogel  \cite{BHMV}  constructed the whole TQFT using Kauffman bracket skein model and the so called universal construction.
 It was further shown by Turaev  \cite{Tu} that a relevant algebraic structure
was that of a {\em modular category}. From a modular category one can derive
a Topological Quantum Field Theory and in particular,
invariants of links and $3$-manifolds.
 A modular category has to be semi-simple with finite set of isomorphism classes of simple objects. In more recent work \cite{TuHQFT}, Turaev has extended the theory and defined modular $G$-categories, which are $G$-graded, to fit with cobordism categories of surfaces and $3$-manifolds equipped with flat $G$ connection. A modular $G$-category still has to be semi-simple and has to satisfy the finiteness condition in each degree. 
Significative examples in the context of quantum groups are still to be found.
 
 Other approaches and  examples of TQFTs have been produced where the 
 source  category of surfaces and their cobordisms differs from the standard
 one by extra decoration or structure:
 see for instance \cite{AK},\cite{BB},\cite{BM},\cite{GP2} and
 \cite{TuHQFT}. 
 A purely Hopf algebra approach for non semi-simple TQFTs was given by Kerler and Lyubashenko in \cite{KL}. It would be interesting to extend this picture to include the family of TQFTs constructed here.

In this paper we build a new series of TQFTs based on nilpotent
representations of quantum $\slt$ at a root of unity of order $2r$ where
$r\geq 2$, is not
divisible by $4$. 
The corresponding category is $\C/2\Z$-graded with
semi-simplicity only in generic grading and the finiteness only up to tensor
action with non trivial $1$-dimensional modules.  The invariants of closed
$3$-manifolds based on these categories have been constructed in
\cite{CGP}. The whole theory deals with cobordisms equipped with colored graph
and compatible cohomology class over $\C/2\Z$, or equivalently $\C^*$ flat
connection, satisfying certain generic admissibility condition. The flat
connection gives a local system. The Reidemeister torsion, defined with
certain indeterminacy from the corresponding complex, was normalized by Turaev
using complex spin structures. For $r=2$ it is know \cite{jM2,Vi} that the
nilpotent quantum $\slt$ invariant of links is a framed version of the
multivariable Alexander polynomial. We show that the closed $3$-manifold
invariant for $r=2$ is a canonical normalization of Reidemeister torsion for
which we get a whole TQFT.  
For general $r$, we compute  the Verlinde formulas which give in the generic case the graded dimensions of the TQFT vector spaces. We may expect an interpretation of those  in conformal field theory.
This paper also contains several examples and applications, including the construction of new mapping class group representations which have interesting properties: the action of a Dehn twist on separating curve has infinite order.  Such behavior is in sharp contrast with the usual quantum representations where all the Dehn twists have finite order.  We were not able to find any elements in the kernel of these mapping class group representations (modulo the center).

 We now give an overview of the construction and results.

\subsection{The difficulties of applying the universal construction}
Suppose one is given an invariant $Z$ of closed oriented $3$-manifolds with the property that $Z(M\sqcup N)=Z(M)Z(N)$. 
Then one may try to build a TQFT out of $Z$ as follows. For each oriented surface $\Su$ define its (infinite dimensional) vector space $\cV(\Su)$ as the $\C$-span of all the oriented $3$-manifolds bounding $\Su$, and also $\cV'(\Su)$ as the span of all the manifolds bounding $\overline{\Su}$ (the surface with opposite orientation). One can define a pairing $\langle\cdot,\cdot,\rangle_Z:\cV'(\Su)\otimes\cV(\Su)\to \C$ by $\langle N,M\rangle_Z=Z(N\circ M)\in \C$.
Finally define $\V(\Su)=\cV(\Su)/\ker_R\langle\cdot,\cdot,\rangle_Z$ (where $\ker_R$ is the annihilator in $\cV(\Su)$ of the whole $\cV'(\Su)$, i.e. the right kernel of the pairing).  Similarly define $\V'(\Su)=\cV'(\Su)/\ker_L\langle\cdot,\cdot,\rangle_Z$.
 The assignment:
$$\{\text{surface}\} \to \{\text{vector spaces}\},  \; \; \Su
\mapsto \V(\Su),$$
extends to a functor $\V$ where cobordisms are sent to linear maps.  
 This functor is called the \emph{universal construction}.  In \cite{BHMV}, it is shown that if the invariant $Z$ behaves well under 1- and 2-surgeries then the  functor $\V$ gives finite dimensional vector spaces. If moreover one can show that generators for the vector space associated with a disjoint surface can be represented by split bordisms, then the functor is monoidal and leads to a finite dimensional TQFT. 
The main example of such a TQFT arises from the universal construction applied to the usual quantum invariant $Z$ coming from a modular category.

The goal of this paper is to use the universal construction to define new TQFTs from the re-normalized quantum invariants  $\Zr$
constructed in \cite{CGP}.  This is not a straightforward application of the universal construction because the following difficulties arise:  
\begin{enumerate}[  \; (D1)]
\item \label{I:Diff1}  Since $\Zr$ is not just an invariant of a 3-manifold but instead an invariant of an \emph{admissible} triple $(M,T,\coh)$, the associated TQFT must be define on suitably decorated surfaces and cobordisms.  
In particular, the surfaces have framed colored points and a 1-cohomology class.  The cobordisms are triples (a $3$-manifold $M$, a colored ribbon graph in $M$, a cohomology class $\coh\in H^1(M\setminus T;\C/2\Z)$).  
 \item \label{I:Diff2} The category of modules used to define  $\Zr$ is not semi-simple.
 \item \label{I:Diff3} The invariant  $\Zr$ does not behave well under 1-surgeries (i.e. connected sums).
\end{enumerate}
Overcoming these difficulties resulted in TQFTs that have several useful applications and novel qualities which we will discuss in the next subsection.  Then in Subsection \ref{SS:OvercomingDiff} we will discuss how we get over Difficulties (D\ref{I:Diff1})--(D\ref{I:Diff3}). 

\subsection{Main Results}
Difficulty (D\ref{I:Diff3}) implies that functor $\V$ obtained from applying the universal construction to $\Zr$ is not monoidal.  However, inspired by the construction of spin TQFT in \cite{BM}, we show that we can use $\V$ to obtain the following 
theorem which is a restatement of Corollary \ref{C:FunctorExists}. 
\begin{theo}\label{T:IntroMainTheorem}
For each integer $r>1$, the functor $\V$ can be upgraded to a monoidal functor $\VV:\Cob\to \GrVect$ from the category of decorated surfaces and their decorated cobordisms (see Subsection \ref{S:cobcat}) to the category of finite-dimensional $\Z$-graded $\C$-vector spaces and their degree $0$-morphisms.
\end{theo}
When $r$ is even the symmetric braiding on $\GrVect$ is non-trivial and coincides with that of super-vector spaces.  So in these cases we obtain a ``super-TQFT''. 
Furthermore, we can explicitly compute the graded dimension of $\VV(\Su)$ via a ``Verlinde formula'' (Theorem \ref{P:verlinde}).

The functors in Theorem \ref{T:IntroMainTheorem} are the TQFTs we study in this paper.  They have several interesting features.  First, they depend on a cohomology class $\coh$.  In some cases we can vary values of this cohomology class continuously.  This feature is very useful.  In particular,  it can be used to study the mapping class group representations associated to the TQFTs of Theorem \ref{T:IntroMainTheorem}.  At this point we do not know if these representations are faithful.    We will now state several results which support this possibility.  Let $\Su$ be a decorated surface of genus $g$ with 1-cohomology class $\coh$.
\begin{theo}
If $r\geq 3$, $g\geq 3$ and the value of $\coh$ on any non-zero cycle is irrational then the action of Johnson's generators of the Torelli groups has infinite order on $\VV(\Su)$. 
\end{theo}
The above theorem is a restatement of Theorem \ref{T:torelli}.  The following two theorems are restatements of Theorems \ref{T:Dehnnon-separating} and  \ref{T:mcgr=2g=1}, respectively.

\begin{theo}
 If  $\coh=0$
 then the action of a Dehn twist along a separating curve of $\Su$ has infinite order on $\VV(\Su)$. 
\end{theo}
\begin{theo}
 If $g=1$ and $r=2$ then
 the mapping class group representation corresponding to $\VV(\Su)$
is faithful modulo its
  center.
\end{theo}
The above results are to be compared with the known fact that in the standard WRT-case one gets quantum representations of the mapping class groups in which the order of the Dehn twists are finite and so which are not faithful (but they are asymptotically faithful \cite{An}, \cite{FWW}).

Another interesting feature of the TQFTs of Theorem \ref{T:IntroMainTheorem}  
is that they place some well known invariants in the larger setting of re-normalized invariant based on non-semi-simple categories.  This allows on one side to study some well-known invariants through a new lens, and on the other side to exploit known results about these invariants to investigate our TQFTs.  We will now discuss three such examples.

First, in Theorem \ref{teo:reidemeister} we prove that when $r=2$ then $\Zr(M,\emptyset,\coh)$ 
is equal to a
canonical normalization of the Reidemeister torsion of $(M,\phi_\coh)$. 
Here
$\phi_\coh\in \Hom(H_1(M,\Z),\mathbb{C}^*)\approx H^1(M,\mathbb{C}^*)\approx
H^1(M,\mathbb{C}/2\Z) $ corresponds to $\omega$. Note that previously Turaev
used a complex spin structure to fix the normalization. Our formula extracts
the contribution of the complex spin structure in Reidemeister-Turaev
torsion.  An interesting consequence is that non-semi-simple $\slt$ quantum
invariants for $r=2$ classify lens spaces (see Proposition~\ref{prop:distinguishlens}), where the usual semi-simple quantum invariants did not.
Thus, 
for $r=2$ our construction gives a TQFT for Reidemeister torsion.

A second example is a relation of the invariants $\Zr$ with the
Kashaev invariants of knots in the sphere: indeed when $M=S^3$ and $T$
is a framed knot colored by the ``Kashaev-module'' over $\UsltH$ and
$\coh=0$ (for $r$ odd) then $\Zr(M,T,0)$  equals Kashaev's invariant which is at the base of the
formulation of the well known Volume Conjecture.  (For generic $\coh$, $\Zr(M,T,\coh)$ equals the
Akutsu-Deguchi-Ohtsuki invariant \cite{ADO} of $T$, for a suitable
coloring of $T$).  As already discussed in \cite{CGP}, our invariants
allow an extension of Kashev invariants to knots and links in any
manifold and thus a  generalized Volume
Conjecture for arbitrary manifolds.
So if we apply the functor $\VV$ to a surface equipped with the zero-cohomology class 
 then we obtain a TQFT for these Kashev invariants.  

The third example is the Kauffman skein algebra of $\Su$. We show that
a suitable finite sum of modules $\V(\Su)$ is acted on by this algebra
(see Subsection~\ref{sub:kauffman}).  The existence of the Kauffman
skein algebra in our TQFTs makes it possible to compute a large class
of examples and to make connections with a great deal of previous
works.

\subsection{Overcoming Difficulties (D\ref{I:Diff1})--(D\ref{I:Diff3})}\label{SS:OvercomingDiff}  Difficulty (D\ref{I:Diff1}):
The invariant $\Zr$ is not defined for \emph{any}  triple $(M,T,\coh)$ but instead only for admissible triples (see Definition \ref{D:Admissiblethreeuple}). Basically, a triple $(M,T,\coh)$ is \emph{admissible} if $T$ contains an edge colored by a projective module of $\UsltH$ or $\coh$ is not integral. 
With this requirement on $\Zr$ we need to ensure that the composition cobordisms is always admissible. 
We solve this problem by requiring that each component of a cobordism from $\Su_-\to \Su_+$ \emph{which does not touch}  $\Su_-$ is admissible (Definition \ref{D:Cobordisms}). 
This automatically ensures that when pairing $N\in \cV(\Su)$ with $M\in  \cV'(\Su)$ via $\langle\cdot,\cdot,\rangle_{\Zr}$ we obtain a closed cobordism which is admissible.  However, the admissibility requirement  breaks the symmetry between $\Su_-$ and $\Su_+$ : a manifold representing an admissible cobordism  $\Su_-\to \Su_+$ need not
represent an admissible cobordism  $\overline{\Su}_+\to \overline{\Su}_-$.

Difficulty (D\ref{I:Diff2}):  Since the category of modules underlying the construction of the invariants $\Zr$
is not semi-simple, to determine a basis for the TQFT, we completely classify all the projective modules of the category and compute their ``modified traces''; we decided to provide all the purely algebraic aspects of our analysis in a forthcoming paper \cite{CGP3}, while here we cite in the Appendix all the facts we will need.  
Also, it should be noted that the non-semi-simplicity makes the definition of $\Zr$ much more complicated and requires new techniques, see \cite{CGP}.

Difficulty (D\ref{I:Diff3}): 
Since the invariant $\Zr$ does not behave well under $1$-surgeries, then a manifold bounding a surface contains an essential sphere may not be  equivalent in $\V(\Su)$ to the disjoint union of the two manifolds obtained by cutting along the sphere and filling back by two $3$-balls.  This is the reason why $\V$ is not monoidal, see Proposition \ref{prop:vspheres}.  We solve this problem by understanding exactly in which cases this cutting and refilling operation is possible (Proposition \ref{P:1-surg}) and in the other cases we replace it by simply cutting but not refilling.  This automatically leads to the definition of a larger, graded vector space $\VV(\Su)$ for any surface $\Su$ (Definition \ref{def:VV}).

\subsection{Further directions of investigation}
As described above, plenty of challenging questions are open by the present
    work. Let us formulate some of them here:
\begin{question}[Torelli representations]
Let $\Su$ be a closed surface with an irrational cohomology class $\coh$.
For a fixed $r\geq 2$ is the representation of the Torelli group on $\VV(\Su)$ faithful?  If not what is the kernel?
\end{question}
As stated above, when $r=2$ the invariants $\Zr$ coincide (up to some factors)
with the Reidemeister torsion of $3$-manifolds equipped with the abelian
representation of their fundamental group associated to $\coh$. So our
   construction provides a new TQFT for the torsion:
\begin{question}
When $r=2$ is there a relation between $\VV(\Su)$ and the vector spaces associated to $\Su$ coming from the TQFTs defined by Kerler in \cite{Ke} and by Frohman and Nicas in \cite{FN}.  
\end{question}
As stated above, when $\coh=0$ one gets representations of the whole mapping class group and the invariants (at least for odd $r$) are related to Kashaev invariants and so, via the Volume Conjecture to hyperbolic geometry:
\begin{question}[Kashaev's TQFTs]
Suppose that $r$ is odd and that $\coh=0$. \begin{enumerate}
\item Is the action of mapping class group $\mcg$ on $\VV(\Su)$ faithful?  If not what is its kernel?  
\item If $\phi\in \mcg(\Su)$, what is the asymptotic growth rate as $r\to \infty$ of the spectral radius of the operator induced by $\phi$ on $\VV(\Su)$?  In particular, if $\phi$ is pseudo-Anosov is this behavior related to the dilation factor or to the volume of the associated mapping torus?
\end{enumerate}
\end{question}
Related to the above question is the problem of giving an explicit base of $\VV(\Su)$ when $\coh$ is integral (we solved this problem in the case $g=1$ and when $\coh$ is non-integral in Subsection \ref{sub:genusgcase}).
The results of \cite{CGP2} should be a guide to answering the following question.

\begin{question}[Relation with the standard TQFTs]
What is the relation between the standard Reshetikhin-Turaev TQFTs at level $r$ for $\Uq$ and $\VV(\Su)$ when $\coh$ is integral? 
\end{question}

Finally the category of $\UsltH$-modules  we consider seems similar to the
category of modules over the ``triplet vertex algebra $W(p)$'' appearing in
the context of logarithmic conformal field theory \cite{FGST}. In the
standard Reshetikhin-Turaev case at level $r$ the spaces associated to $\Su$
can be identified with the conformal blocks of the Chern-Simons conformal
   field theory at the same level, so it is natural to ask the following:
\begin{question}
Is there a (logarithmic) conformal field theory whose conformal blocks correspond to our $\VV(\Su)$?
\end{question}
 
\subsection{Structure of the paper}
Each section is preceded by a summary highlighting the main points of
the section.  We strongly advise the reader to read these summaries
before getting to the details of the section.
Section~\ref{sec:invariants} recalls the definition of the invariants
$\Zr$. Section~\ref{sec:universal} defines the category of
cobordisms $\Cob$ and applies the universal construction to $\Cob$.
This section also contains properties of $\V$ which can be deduced
easily from this construction.  Section~\ref{S:PropTQFT} studies in
depth the properties of $\V$ and shows, in particular, that $\V(\Su)$
is finite dimensional, non-monoidal and carries a natural grading.  In
Section~\ref{sec:VV}, we use $\V$ to define the monoidal functor
$\VV$; then we compute the graded dimension of $\VV(\Su)$.
Section~\ref{sec:applications} is devoted to providing examples of our
construction and contains the first topological applications of it: in
particular, we show that the generators of the Torelli groups act with
infinite order when $\coh$ is irrational and relate our construction
to the Reidemeister torsion and to the Kauffman skein algebra.  The
Appendix contains some algebraic facts we need in the paper.

\section{Invariants of links and closed 3-manifolds}\label{sec:invariants}


In this section we recall the construction of the invariants $\Zr(M,T,\coh)$ of $3$-manifolds equipped with the datum of a (possibly empty) ribbon graph $T$ and a   cohomology class $\coh\in H^1(M\setminus T;\C/2\Z)$ defined in \cite{CGP}. 
(Actually the invariants we will define differ from the original ones by a re-normalization factor which suits the purposes of the universal construction.)
In Subsection \ref{SS:Quantsl2} we define the category $\cat$ of modules which we will use to color the edges of the ribbon graphs in the present paper. In particular we provide the complete list of all the simple modules of $\cat$: particular attention should be given to the modules $V_\alpha$ and $\sigma^k$. 
In Subsection \ref{ss:F'} we recall the construction of invariants of $\cat$-colored ribbon graphs in $S^3$ using the modified trace technique.  Finally in Subsection \ref{sub:invariants} we recall the definition of the invariants $\Zr$. Special attention here should be payed to the fact that the invariants are defined only for the set of \emph{admissible} $3$-uples (see Definition \ref{D:Admissiblethreeuple}): this apparently technical point has strong consequences on the structure of our TQFT.
\subsection{A quantization of $\slt$}\label{SS:Quantsl2}  Here we give a slightly
generalized version of quantum $\slt$.  Fix an integer $r\geq 2$ which is either odd or 2 times an odd.  Let $q=e^\frac{\pi\sqrt{-1}}{r}$ and $r'=r$ if $r$ is odd and $r'=r/2$ if $r$ is even.  We shall use the notation
$q^x=e^{\frac{\pi\sqrt{-1} x}{r}}$ and $\{x\}:=q^x-q^{-x}$.

Let $\UsltH$ be the
$\C(q)$-algebra given by generators $E, F, K, K^{-1}, H$ and
relations:
\begin{align*}
  HK&=KH, & HK^{-1}&=K^{-1}H, & [H,E]&=2E, & [H,F]&=-2F,\\
  KK^{-1}&=K^{-1}K=1, & KEK^{-1}&=q^2E, & KFK^{-1}&=q^{-2}F, &
  [E,F]&=\frac{K-K^{-1}}{q-q^{-1}}.
\end{align*}
In \cite{CGP}, $\UsltH$ is given an explicit  Hopf algebra structure.  Define $\Ubar$ to be the Hopf algebra $\UsltH$ modulo the relations $E^\ro=F^\ro=0$.

Let $V$ be a finite dimensional $\Ubar$-module.  An eigenvalue
$\lambda\in \C$ of the operator $H:V\to V$ is called a \emph{weight}
of $V$ and the associated eigenspace is called a \emph{weight space}.
We call $V$ a \emph{weight module} if $V$ splits as a direct sum of
weight spaces and $\qr^H=K$ as operators on $V$.  Let $\cat$ be the
tensor category of finite dimensional weight $\Ubar$-modules.  The category $\cat=\oplus_{\wb{\alpha}\in \C/2\Z} \cat_{\wb{\alpha}}$ is $\C/2\Z$-graded where $ \cat_{\wb{\alpha}}$ is the category of modules whose weights are all in the class  $\wb{\alpha}$ (mod $2\Z$).  
If $V\in \cat_{\wb{\alpha}}$ then we say $V$ is \emph{homogeneous} and has \emph{degree} $\wb{\alpha}$.  
Moreover, the category $\cat$ is a ribbon category (see \cite{GPT}).  

Every simple module of $\cat$ is isomorphic  to one of the modules in the following list:
\begin{itemize}
\item The simple modules $S_i$, for $i=0,\cdots, r-1$, with highest weight $i$ and 
  dimension $i+1$.   These modules are the ``standard'' simple modules of $\Ubar$ used to define the usual R-T 3-manifold invariant.  The degree of $S_i$ is $i\ {\rm mod} \ 2$. 
\item The one dimensional modules $\C^H_{kr}$, for $k\in
  \Z$,  where both $E$ and $F$ act by zero and $H$ acts by $kr$.  The degree of $\C^H_{kr}$ is $kr\ {\rm mod} \ 2$.
\item The simple modules $S_i\otimes  \C^H_{kr}$, for $i=0,\cdots, r-1$ and $k\in \Z\setminus \{0\}$.  
The degree of $S_i\otimes  \C^H_{kr}$ is $i+kr \ {\rm mod} \ 2$.
\item The simple highest weight module $V_\alpha$ of highest weight $\alpha+r-1$ for  $\alpha\in(\C\setminus \Z)\cup
  r\Z$.   All of these modules are projective and have dimension $r$.  The degree of $V_\alpha$ is  $\alpha+r-1 \ {\rm mod} \ 2$.  If $\alpha+\beta \notin \Z$ and $\alpha,\beta \in (\C\setminus \Z)\cup r\Z$ then $V_\alpha\otimes V_\beta\simeq\oplus_{k\in \Hr} V_{\alpha+\beta+k}$ where $\Hr=\{-(r-1),-(r-3),\ldots, r-1\}$.

\end{itemize}
Remark that there is a slight redundancy in the above notation as
$S_0=\C^H_0$, $S_{r-1}=V_{0}$ and $S_{r-1}\otimes\C^H_{kr}=V_{kr}$.
In particular, we will denote by $\sigma$ the module $\C^H_{2r'}$ and
by $\sigma^k$ the module $\C^H_{2kr'}$.  
Then $V_{\alpha+2kr'}=V_\alpha\otimes \sigma^k$.
In the appendix we describe the projective indecomposable
modules of $\cat$.
\subsection{Ribbon graphs, the ribbon functor $F$, and $F'$}\label{ss:F'}
Let $M$ be an oriented 3--manifold.  A \emph{ribbon graph} in $M$ is
an oriented compact surface in $M$ decomposed into elementary
pieces: bands, annuli, and coupons (see \cite{Tu}).  A coupon is a box with oriented bands representing a morphism in $\cat$.
A \emph{$\cat$-colored ribbon graph} in $M$ is a ribbon graph whose bands and
annuli are colored by objects of $\cat$ and whose coupons are colored
with morphisms of $\cat$ (for the most part of this paper the reader can
assume that ribbon graphs do not have coupons).  A ribbon graph can and will
be identified with a framed oriented link or graph embedded in $M$ (here the
framing can be considered a vector field on the graph which is nowhere tangent
to it).  

We now recall the category of $\cat$-colored ribbon graphs $\Rib$
(for more details see \cite{Tu} Chapter~I).  The objects of $\Rib$
are sequences of pairs $(V,\epsilon)$, where $V\in Ob(\cat)$ and
$\epsilon=\pm$ determines the orientation of the corresponding
edge. The morphisms of $\Rib$ are isotopy classes of
$\cat$-colored ribbon graphs in $\R^2 \times [0,1]$ and their formal
linear combinations with coefficients in $\C$. 

Let  $\Cp=(\C\setminus \Z)\cup r\Z$ and $\qd:\Cp\to\C^*$ be the map given by
\begin{equation}\label{E:Def_qd}
 \qd(\beta)=(-1)^{r-1}\prod_{j=1}^{r-1}
\frac{\qn{j}}{\qn{\beta+r-j}}=(-1)^{r-1}\frac{r\,\qn{\beta}}{\qn{r\beta}}.
\end{equation}
Let $F$ be the usual ribbon functor from $\Rib$ to $\cat$ (see
\cite{Tu}).  Let $L$ be a $\cat$-colored ribbon graph in $\R^3$ with
at least one edge colored by a simple module $V=V_\alpha$ for some
$\alpha \in \Cp$.  Cutting this edge, we obtain a $\cat$-colored
(1,1)-ribbon graph $T_V$ whose closure is $L$.  Since $V$ is simple
$F(T_{V})\in
\End_{\cat}(V)= \C \Id_V$.  Let $<T_{V}> \, \in \C$ be such that
$F(T_{V})= \, <T_{V}> \, \Id_V$.   Then as shown in \cite{CGP} 
the assignment 
\begin{equation}\label{E:Def_of_F'}
  F'(L) = \qd(\alpha)<T_{V}> \in \C
\end{equation} 
is independent of the choice of the edge to be cut and yields a well
defined invariant of $L$.
In the appendix, we describe the extension of $F'$ to $\cat$-colored
ribbon graph in $\R^3$ with at least one edge colored by a projective
module.

A trivalent, oriented framed graph $\Gamma$ whose edges are colored by $\Cp$ can be seen as a $\cat$-colored ribbon graph
as follows (for more details see \cite{CGP}).  To each edge of $\Gamma$ labeled by $\alpha\in \Cp$ associate the simple module $V_\alpha$ described above.   Each $\Hom(V_\alpha, V_\beta\otimes V_\gamma)$ is either $1$ or $0$-dimensional.    Choose a basis for all these homomorphism spaces as discussed in \cite{CGP}.   To each trivalent vertex of $\Gamma$ associate a coupon equipped with the appropriate chosen morphism.  This procedure leads to a well defined invariant of trivalent graphs considered up to isotopy.    So when it is convenient, instead of considering $\cat$-colored ribbon graphs we will restrict to $\Cp$-colored oriented framed trivalent graphs. 

\subsection{The invariants of closed $3$-manifolds}\label{sub:invariants}
In \cite{CGP}, the link invariant $F'$ is used to define an invariant $\Nr(M,T,\coh)$ where $M$ is a connected, closed, oriented $3$-manifold, $T\subset M$ is a (possibly empty)  suitably decorated ribbon graph and $\coh\in H^1(M\setminus T;\C/2\Z)$ is a cohomology class (where $\C/2\Z$ is considered as an abelian group), satisfying a compatibility condition. 
The first purpose of this subsection is to introduce the reader to these objects.   
Then, we define a new invariant $\Zr$ which extends $\Nr$ to disconnected manifolds and differs by a normalization factor.  
This renormalization insures that the TQFT defined in Subsection \ref{SS:TQFTfunctor} has a list of desired properties.  

 We start by
recalling some definitions given in \cite{CGP}.  Let $M$ be a compact
oriented 3--manifold, $T$ a $\cat$-colored ribbon graph in $M$ and
$\coh\in H^{1}(M\setminus T,\C/2\Z)$.  Let $L$ be an oriented framed link in
$S^3$ which represents a surgery presentation of $M$.  One defines a $\C/2\Z$-valued map $g_\coh$
on the set of edges of $L\cup T$ by
$g_\coh(e_i)=\coh(m_i)$, where $m_i$ is a meridian of $e_i$ oriented so that it bounds a disc whose algebraic intersection with $e_i$ is $1$. We shall call $g_\coh$ the
\emph{$\C/2\Z$-coloring} of $L\cup T$ induced by $\coh$.

\begin{defi}\label{def:adm} 
  Let $M$, $T$ and $\coh$ be as above.
  \begin{enumerate}
  \item We say that $(M,T,\coh)$ is {\em a compatible triple} if for each edge
    $e$ of $T$ its coloring is in $\cat_{g_\coh(m_e)}$ where $m_e$ is a
    meridian of $e$.
  \item A link $L\subset S^3$ which is a surgery presentation for a
    compatible triple $(M,T,\coh)$
    is \emph{computable} if one of the two following conditions holds:
  \begin{enumerate}
  \item $L\neq \emptyset$ and $g_\coh(L_i) \in ({\C/2\Z})\setminus
    (\Z/2\Z)$ for all components $L_i$ of $L$, or
  \item $L=\emptyset$ and there exists an edge of $T$ colored by
    a projective module of $\cat$.
  \end{enumerate}
\end{enumerate}
\end{defi}

Let $\Hr=\{1-r,3-r,\ldots,r-1\}$.
For $\alpha\in \C\setminus\Z$ we define the Kirby color
$\Omega_{{\alpha}}$ as the formal linear combination
\begin{equation}
  \label{eq:Om}
  \Omega_{{\alpha}}=\sum_{k\in \Hr}\qd(\alpha+k)V_{\alpha+k}.
\end{equation}
If $\wb{\alpha}$ is the image of $\alpha$ in $\C/2\Z$ we say
that $\Omega_{{\alpha}}$ has degree $\wb{\alpha}$.  We can ``color'' a
knot $K$ with a Kirby color $\Omega_{{\alpha}}$: let
$K({\Omega_{{\alpha}}})$ be the formal linear combination of knots
$\sum_{k\in \Hr} \qd(\alpha+k) K_{\alpha+k}$ where $K_{\alpha+k}$ is the
knot $K$ colored with $V_{\alpha+k}$.  If
$\wb{\alpha}\in\C/2\Z\setminus \Z/2\Z$, by $\Omega_{\wb{\alpha}}$, we
mean any Kirby color of degree $\wb\alpha$.
Next we recall the main theorems of \cite{CGP}.  

\begin{theo}[\cite{CGP}]\label{T:sl2CompSurgInv}
 Let  $(M,T,\coh)$ be a compatible triple where $M$ is connected.  If $L$ is a link which gives rise to a computable surgery
  presentation of $(M,T,\coh)$ then
  $$\Nr(M,T,\coh)=\dfrac{F'(L\cup T)}{\Delta_+^{p}\ \Delta_-^{s}}$$
  is a well defined topological invariant (i.e. depends only of the
orientation preserving diffeomorphism class of the triple $(M,T,\coh)$), 
  where $\Delta_+,\ \Delta_-$ are constants given in \cite{CGP} (see
  Equation \eqref{eq:delta}), $(p,s)$ is the signature of the linking
  matrix of the surgery link $L$ and for each $i$ the component $L_i$
  is colored by a Kirby color of degree $g_{\coh}(L_i)$.
\end{theo}

In this paper we will use without further mention the universal
coefficients theorem to identify $H^1(M\setminus
T;\C/2\Z)$ and $\Hom(H_1(M\setminus \Z;\Z);\C/2\Z)$.  In particular, we will say that a class $\coh\in
H^1(M\setminus T;\C/2\Z)$ is \emph{integral} if belongs to
$\Hom(H_1(M\setminus T;\Z);\Z/2\Z)$.
A cohomology class $\coh\in H^1(M\setminus T,*;\C/2\Z)$, relative
to a finite set of base points $*$,  is
\emph{integral} if its restriction to $M\setminus T$ is integral in
$H^1(M\setminus T;\C/2\Z)$.
Hence a non-integral cohomology class takes a value in
$\C/2\Z\setminus\Z/2\Z$ on the homology class of a {\em closed} curve
in $M\setminus T$.

\begin{defi}[Admissible triple]\label{D:Admissiblethreeuple}
  Let $(M,T,\coh)$ be a compatible triple where $M$ is connected. 
  Then $(M,T,\coh)$ is said to be \emph{admissible} if either $\coh$
  is non-integral or $T$ contains an edge whose color is a
  projective module in $\cat$.  In general, an admissible triple is a
  compatible triple such that each connected components is admissible.
\end{defi} 
We will now show that the invariant $\Nr$ is defined for an admissible triple $(M,T,\coh)$ where $M$ connected.  
We have  the following two cases:
\begin{description}
\item[Case 1] If $\coh$ is non-integral then Proposition 1.5 of \cite{CGP} states that $(M,T,\coh)$ has a
computable surgery presentation and we can apply Theorem \ref{T:sl2CompSurgInv}.
\item[Case 2] If $\coh$ is integral then  $(M,T,\coh)$
does not have a computable surgery presentation (unless $M=S^3$).  However, 
since $(M,T,\coh)$ is admissible then $T$ has an edge
colored by a projective module $P$.  In this case, we define a new triple $(M,T',\coh')$ where   $(T',\coh')$ is a local
modification of $(T,\coh)$ as follows.  
 For any $\alpha\in\C\setminus \Z$, then $V_\alpha$
generates the ideal of projective modules of $\cat$ and so for any projective module $P\in\cat$, there exists a module $W\in\cat$ and morphisms $g:V_\alpha\otimes W\to P$ and $f:P\to V_\alpha\otimes W$ such
that $gf=\Id_P$.  Let $\Gamma_\alpha:(P,+)\to(P,+)$ be the (1-1)-ribbon graph given by stacking the coupon colored with $g$ on top of the coupon colored with $f$.  Then $\Gamma_\alpha$ has a component colored by $V_\alpha$ and $F(\Gamma_\alpha)=\Id_P$.  
The modification of $T$ consists in
replacing a neighborhood of a part of the edge colored by $P$ with
$\Gamma_\alpha$.  There is a unique way to modify locally the
cohomology class $\coh$ to a cohomology class $\coh'$ compatible with
$T'$.  This cohomology class takes the value
$\wb\alpha\in\C/2\Z\setminus\Z/2\Z$ on the meridian of the edge of
$\Gamma_\alpha$ colored by $V_\alpha$.  Hence, $(M,T',\coh')$ has a
computable surgery presentation.  It can be shown that
$\Nr(M,T,\coh)=\Nr(M,T',\coh')$ is independent of $\Gamma_\alpha$, see
Remark 3.4 of \cite{CGP}.  Note that $(M,T,\coh)$ and
$(M,T',\coh')$ will be called skein equivalent in Section
\ref{SS:Skein}.
\end{description}

In the present paper we will use a renormalized version of $\Nr$ and
 extend it to an invariant $\Zr$ of (possibly disconnected)
admissible triples by
\begin{equation}\label{eq:Zdefi}
\Zr(M,T,\coh)=\eta^{b_0(M)}\lambda^{b_1(M)}\prod_{i=1}^n \Nr(M_i,T\cap
M_i,\coh_{|M_i})\in\C\end{equation}
where $b_0(M)=n$ is the number of connected component of $M$, $b_1(M)$ is its
first Betti number and $M_i$ is the $i$\textsuperscript{th} component of $M$.  The
scalars $\lambda$ and $\eta$ are given by 
$$\lambda=\frac{\sqrt{r'}}{r^2}\quad\text{ and }\quad\eta=\frac1{r\sqrt{r'}}.$$
Since we plan on using $\Zr$ to define a TQFT we need it to satisfy certain properties.   The normalization of $\Nr$ ensures these properties. 
In particular, the choice of $\eta$ and $\lambda$ guarantees that 
the TQFT has good behavior with respect to the $1$ and $2$-surgeries (see
Propositions \ref{P:1-surg} and \ref{P:2-surg}).  The invariant $\Zr$ also satisfies $\Zr(M\sqcup M')=\Zr(M)\Zr(M')$.  Finally, if $L$ is a 
computable presentation of a compatible triple $(M,
T,\omega)$ with $M$ connected, then we have the following useful expression:
\begin{equation}
  \label{eq:Zrconnected}
  \Zr(M,T,\omega)=\eta\lambda^{m}\dep^{-\sigma}F'(L\cup T)
\end{equation}
where the components of $L$ are colored by Kirby colors as in Theorem \ref{T:sl2CompSurgInv}, $m\in\N$ is the number of components of $L$, $\sigma\in\Z$ is the
signature of the linking matrix $\lk(L)$ and 
\begin{equation}
  \label{eq:delta}
  \dep=\lambda\Delta_+=(\lambda\Delta_-)^{-1}=q^{-\frac32}\e^{-i(s+1)\pi/4}
 \text{ where } s \text{ is in } \{1,2,3\} \text{ with } s \equiv r \text{ mod }4.
\end{equation}
Later on we will extend \eqref{eq:delta} to ``extended manifolds'' as needed to fix the framing anomaly (see formula \eqref{eq:Zrweighted}).

\renewcommand{\ds}{}
\section{The universal construction of TQFT}\label{sec:universal}
In this section we provide the general framework to which we apply the universal construction for TQFTs. 
After defining in detail what a decorated surface is in Subsection \ref{S:deco2}, we define the cobordisms in Subsection \ref{S:deco3} and finally provide the category $\Cob$ to which our functors will apply in Subsection \ref{S:cobcat}.
In Subsection \ref{SS:TQFTfunctor} we apply the universal construction and define the functors $\V$ and $\V'$.  The main point of attention here is the very last point of Definition \ref{D:Cobordisms} which is needed in order to make sure that when composing morphisms to get a closed morphism the obtained manifold is admissible according to Definition \ref{D:Admissiblethreeuple}. Because of this requirement,  the role in our theory of the spaces $\V(\Su)$ and $\V'(\Su)$ is not symmetric. 
In Subsection \ref{SS:Vdunion} we prove a list of properties of $\V$ and $\V'$ which are basically given ``for free'' by the universal construction; in particular they are ``lax monoidal'' (but not monoidal).

\subsection{Decorated surfaces}\label{S:deco2}
We define a 2+1-cobordism category of decorated manifolds 
(Turaev defined a similar category of ``extended'' cobordisms in
\cite{Tu}).
\begin{defi}[Objects]
  A \emph{decorated surface} is a $4$-uple $\Su =(\Su,\{p_i\},\coh,
  {\La})$ where:
\begin{itemize}
\item $\Su$ is a closed, oriented surface which is 
  an ordered disjoint union of connected surfaces each having a distinguished base point $*$; 
\item $\{p_i\}$ is a finite (possibly empty) set of homogeneous $\cat$-colored 
   oriented framed points in $\Su$
  distinct from the base points, i.e. each $p_i\in \Su$ is equipped
  with a sign, a non-zero tangent vector and a color which is an
  object of $ \cat_{\wb{\alpha_i}}$ for some $\wb{\alpha_i}\in \C/2\Z$;
 \item
     $\coh\in H^1(\Su\setminus \{p_1,\ldots, p_k\}, *;\C/2\Z)$ is a
    cohomology class such that $\coh( m_i)=\overline{\alpha_i}\
    \rm{mod}\ 2\mathbb{Z}$ where $m_i$ is the boundary of a regular neighborhood $D_i$ of $p_i$.  The meridian $m_i$ is
    oriented in the positive sense (as the boundary of $D_i$) if $p_i$
    is negative and in the negative sense if $p_i$ is
    positive\footnote{The choice of the orientations here is a consequence of the
      convention to orient boundaries described in Subsection
      \ref{S:deco3}.};
\item ${\La}$ is a lagrangian subspace of $H_1(\Su;\R)$.
\end{itemize}
\end{defi} 

The \emph{opposite or negative} of $(\Su,\{p_1,\ldots p_k\},\coh,{\La})$
is defined as $(\overline{\Su},\{\overline{p}_1,\ldots
\overline{p}_k\},\coh,{\La})$ where $\overline{\Sigma}$ is $\Sigma$ with the opposite orientation, $\overline{p}_i$ is $p_i$ with the opposite sign and opposite vector but same object of $\cat$.

\begin{rem} 
  There is exactly one base point on each connected
  component of $\Su$, so $H^1(\Su\setminus \{p_1,\ldots p_k\}
  ,*;\C/2\Z)=H^1(\Su\setminus \{p_1,\ldots p_k\};\C/2\Z)$.  Also all
  cohomology classes can be interpreted as equivalence class of flat
  $\C^*$-bundles trivialized at the base points (where
  $\C^*\simeq\C/2\Z$).
\end{rem}

\subsection{Decorated cobordisms}\label{S:deco3}
With one exception, we orient the boundaries of manifolds using the
``outward vector first'' convention.  The exception concerns the
$0+1$-dimensional strata: the boundary of an edge of a ribbon graph is
formed by its set of univalent (framed) vertices.  Such a vertex is
oriented by a $+$ sign if its adjacent edge is oriented outgoing from
the vertex and a $-$ sign if its adjacent edge is oriented incoming to
the vertex.  We follow here Turaev's convention (see \cite{Tu}, IV Section 1.5).

\begin{defi}[Cobordisms]\label{D:Cobordisms}
 Let $\ds{\Su}_\pm=(\Su_\pm,\{p_i^\pm\},\coh_\pm,{\La}_\pm)$ be decorated surfaces.  A \emph{decorated cobordism} from $\ds{\Su}_-$ to  $\ds{\Su}_+$ is a $5$-uple $\ds M=(M,T,f, \coh,n)$ where:
\begin{itemize}
\item $M$ is an oriented $3$-manifold with boundary $\partial M$;
\item $f: \overline{\Su_-}\sqcup \Su_+\to \partial M$
is a diffeomorphism that
  preserves the orientation;
 denote the image under $f$ of the base points of 
  $\overline{\Su_-}\sqcup \Su_+$ by $*$;  
\item $T$ is a $\cat$-colored ribbon graph in $M$ such that 
$\partial T=\{\wb{f(p_i^-)}\}\cup \{f(p_i^+)\}$
and the color of the edge of $T$ containing $f(p_i^\pm)$ equals the color of $p_i^\pm$; 
\item $\coh\in H^1(M\setminus T,*;\C/2\Z)$ is a cohomology class
  relative to the base points on $\partial M$, such that
  the restriction of $\coh$ to $(\partial M\setminus \partial T)\cap \Su_\pm$ is  $(f^{-1})^*(\coh_{\pm})$;
\item the coloring of $T$ is compatible with $\coh$, i.e. each oriented edge
  $e$ of $T$ is colored by an object in $\cat_{\coh(m_e)}$ where $m_e$ is the
  oriented meridian of $e$;
\item $n$ is an arbitrary integer
  called the {\em signature-defect} of $\ds M$;
\item each connected component of $M$ disjoint from 
$f(\overline{\Su_-})$ is   admissible\footnote{Here we extend Definition \ref{D:Admissiblethreeuple} to the case $\partial M\neq \emptyset$.}.
\end{itemize}
  Figure \ref{F:decobo} illustrates an example of decorated cobordism.  We use $f$ to identify $\partial_+\ds M=f({\Su_+})\simeq \Su_+$ and $\partial_-\ds M=f(\wb{\Su_-})\simeq\wb{\Su_-}$.   We will denote $(\Su_+,\{p^+_i\},\coh
_{|\Su_+},\La_+)$ by $\partial_+ \ds M$ and $(\Su_-,\{p^-_i\},\coh
_{|\Su_-},\La_-)$ by
$\wb{\partial_-\ds M}$.
\end{defi}

\begin{figure}[!]
  \centering 
  \(\epsh{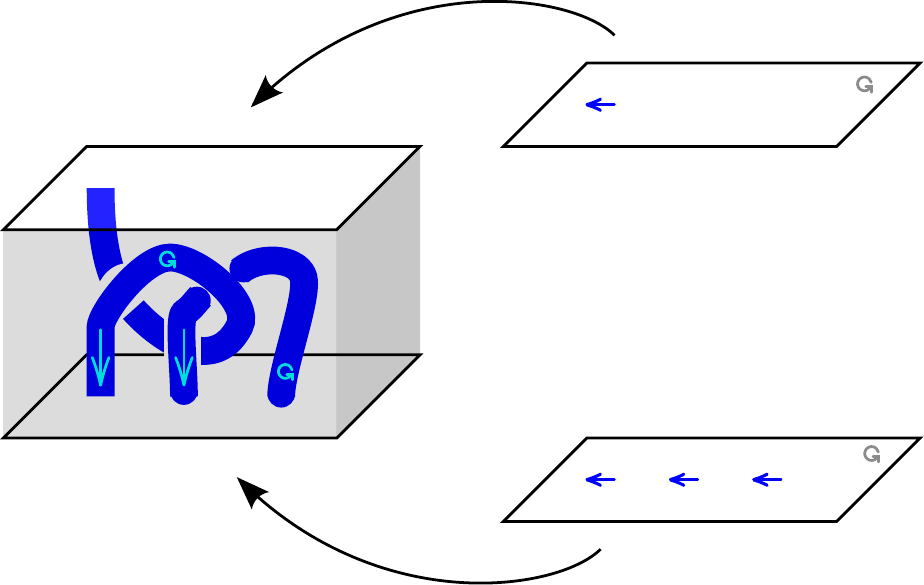}{40ex}
  \put(-285,-10){$M$}
 \put(0,80){$\Su_+$}\put(0,-70){$\Su_-$}
  \put(-180,60){$f^+$}\put(-180,-60){${f^-}\circ 
   \wb{\Id_{\Su_-}}$}
   \put(-256,-17){$U$}\put(-178,-15){${V}$}
  \put(-110,-62){\ms{(U,+)}}\put(-104,61){\ms{(U,+)}}
  \put(-80,-62){\ms{(V,+)}}\put(-50,-62){\ms{(V,-)}}
  \)
  \caption{
   Orientations of a decorated cobordism $M:\Su_-\to\Su_+$. The framed graph $T$ is in blue.  Its component are colored by the  modules $U$ and $V$.  Here   $\wb{\Id_{\Su_-}}:\Su_-\to\wb{\Su_-}$ is the identity, $f^{+}=f_{|\Su_+}$ and $f^{-}=f_{|\wb{\Su_-}}$.  The coloring,  framing and orientations of the dots of $\Su_\pm$ match those of  $T$.
   }
  \label{F:decobo}
\end{figure}
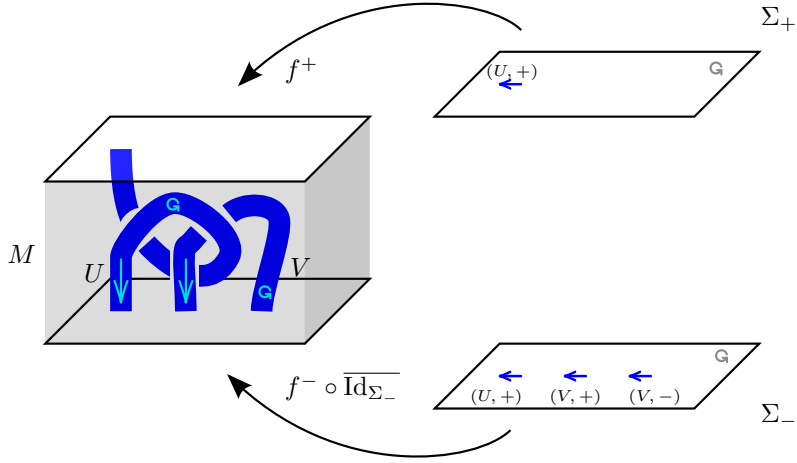
 We consider decorated cobordisms from $\ds{\Su}_-$ to $\ds{\Su}_+$  up to diffeomorphism:  a \emph{diffeomorphism}  $$g:~(M,T,f,\omega,n)\to   (M',T',f',\omega',n')$$ is  an orientation preserving diffeomorphism of the underlying manifolds $M$ and $M'$, still denoted by $g$,  
  such that $g(T)=T'$, $g\circ f=f'$, $\omega=g^*(\omega')$ and 
  $n=n'$.  Remark
  that up to diffeomorphism, $M=(M,T,f,\omega,n)$ only depends on $f$
  up to isotopy.

Notice that the last condition in Definition \ref{D:Cobordisms}
ensures that all cobordisms from $\emptyset$ to 
 $\ds{\Su}_2$ are
admissible (as in Definition \ref{D:Admissiblethreeuple}). In
particular, any closed cobordism $(M,T,\emptyset, \coh,n)$ is admissible.  Hence we 
extend the definition of $\Zr$ to closed decorated manifolds as follows:
\begin{equation}
  \label{eq:Zrweighted}
  \Zr(\ds{M},T,\emptyset, \coh,n):=\dep^{n}\Zr(\ds{M},T,\coh)
\end{equation}
where $\dep$ is given in Equation \eqref{eq:delta}.

\begin{rem}\label{rem:inducedcoh}
  Two different cohomology classes that are compatible with the same pair
  $(M,T)$ differ by an element of $H^1(M,\partial M;\C/2\Z)$.  When this group
  is zero the compatible cohomology class is unique. 
\end{rem}
\subsection{The decorated cobordism category}\label{S:cobcat}
For a symplectic vector space $(V,\symf)$ and a triple
$({\La}_-,{\La},{\La}_+)$ of Lagrangian subspaces of
$V$ the \emph{Maslov index} is defined as follows (for details
see \cite[Section IV.3.5]{Tu}).  For each vector $v\in
{\La}_-+{\La}$ choose a decomposition $v=v_-+v_0$ where
$v_-\in {\La}_-$ and $v_0\in {\La}$.  For $v,w\in
{\La}_+\cap ({\La}_-+{\La})$ set $\langle
v,w\rangle=\symf(v_0,w)$.  This assignment is a well defined symmetric
bilinear form on ${\La}_+\cap ({\La}_-+{\La})$.
The Maslov index $\mu({\La}_-,{\La},{\La}_+)\in \Z$
is the signature of the bilinear form $\langle\cdot, \cdot\rangle$.
The Maslov index is totally antisymmetric in its entries.

Given a decorated cobordism $\ds M=(M,T,f, \coh,n)$, from $\ds{\Su}_-$
to $\ds{\Su}_+$, let $f^-=f_{|\wb{\Su_-}}:\wb{\Su_-}\to M$ and
$f^+=f_{|\Su_+}:\Su_+\to M$ be the components of $f$. If $\La_-$
and $\La_+$ are the Lagrangians of $\ds{\Su}_-$ and $\ds{\Su}_+$,
respectively, then we define two other Lagrangians
$M_*(\La_-)=(f^+)_*^{-1}\bp{(f^-\circ
  \wb{\Id_{\Su_-}})_*(\La_-)}\subset H_1(\ds{\Su}_+;\R)$ and
$M^*(\La_+)=(f^-\circ
\wb{\Id_{\Su_-}})_*^{-1}\bp{(f^+)_*(\La_+)}\subset H_1(\ds{\Su}_-;\R)$
where $\wb{\Id_{\Su_-}}:\Su_-\to\wb{\Su_-}$ is the identity.  
The fact that $M_*(\La_-)\subset H_1(\Su_+,\R)$ and $M^*(\La_+)\subset
H_1(\Su_-,\R)$ are indeed Lagrangians follows from the analysis of the
Lagrangian relations induced by $f_+$ and $f_-$ (see \cite[Section
IV.3.4]{Tu}).

\begin{defi}\label{D:Compostion}
 For $j=1,2$,
  let $\ds M_j=(M_j,T_j,f_j, \coh_j,n_j)$ be decorated cobordism such
  that $\Su=(\Su,\{p_i\},\coh_\Su,\La)=\partial_+ \ds
  M_1=\overline{\partial_-\ds M_2}$.  
  Let ${\La}_-$ and ${\La}_+ $  
  be the Lagrangian subspaces of the decorated surfaces
  $\wb{\partial_-\ds M_1}$ and $\partial_+ \ds M_2$, respectively.
The composition $\ds
  M_2\circ \ds M_1$ is the decorated cobordism $(M,T,f,\coh,n)$ 
  from $\wb{\partial_-M_1}$ to $\partial_+M_2$
  defined as follows:
\begin{itemize}
\item
  $M=M_1\cup_{\Su} M_2 $ is the manifold obtained by gluing
  $\partial_+ \ds M_1$ and $\wb{\partial_-\ds M_2}$ through the map   $f_2^-\circ  (f_1^+)^{-1}$, 
\item $T$ is the $\cat$-colored ribbon graph $T_1\cup_{\{p_i\}} T_2$
  in $M$,
\item $f=f_1^-\sqcup f_2^+$,
\item $\coh$ is the 
cohomology class given in Lemma
  \ref{L:McohomEx} below,
\item $n=n_1+n_2-\mu({M_1}_*(\La_-),\La,{M_2}^*(\La_+))$ where $\mu$
  is the Maslov index of the Lagrangian subspaces of $H_1(\Su,\R)$.
\end{itemize}
\end{defi}

\begin{lemma}\label{L:McohomEx}
Let $\ds M_1, \ds M_2, $ and $\ds M=\ds M_2\circ \ds M_1$ be as in Definition \ref{D:Compostion}. 
We distinguish here the base points of $\wb{\partial_-\ds M_2}$,
$\Su$, $\partial_+M_1$ denoted respectively by $*_2$, $*_\Su$ and
$*_1$.  There is a unique a cohomology class
$\wt\coh\in H^1(M\setminus T, *_1\cup *_2\cup *_\Su;\C/2\Z)$ whose
restriction to $H^1(M_j\setminus T_j, *_j\cup *_\Su;\C/2\Z)$ is
$\coh_j$ for $j=1,2$.  
Define $\coh\in H^1(M\setminus T, *_1\cup *_2;\C/2\Z)$ by restricting  $\wt\coh$.
\end{lemma} 
\begin{proof}
  Denote the decorated surface $\partial_+ \ds
  M_1=\overline{\partial_-\ds M_2}$ by $(\Su,\{p_i\},\coh_0, {\La})$.
  Consider the relative Mayer-Vietoris exact sequence (see \cite{Hat}
  page 204) with coefficients in $\C/2\Z$:
  \begin{align*}
    \cdots \longrightarrow& H^0(\Su\setminus \{p_1,\ldots, p_k\},*_\Su)  
    \longrightarrow H^1(M\setminus T,*_1\cup *_2\cup *_\Su)\longrightarrow \\
    & H^1(M_1\setminus T_1,*_1\cup *_\Su)\oplus H^1(M_2\setminus
    T_2,*_2\cup *_\Su) \longrightarrow H^1(\Su\setminus \{p_1,\ldots,
    p_k\},*_\Su) \longrightarrow\cdots.
\end{align*}
The last map sends the pair $(\coh_1,\coh_2)$ to $0$;
indeed 
both $\coh_1$ and $\coh_2$ restrict to the same class in
$H^1(\Su\setminus \{p_i\},*_\Su)$ 
by hypothesis. So there exists $\wt\coh\in H^1(M\setminus T,*_1\cup
*_2\cup *_\Su)$ mapping onto $(\coh_1,\coh_2)$; moreover, this lift is
unique because $H^0(\Su\setminus\{p_i\},*_\Su)=0$ (since there is
exactly one base point in each component of $\Su$).
\end{proof}

\begin{rem}\label{R:McohomEx} 
In this remark we give a description of the kernel of the map $(\coh_1,\coh_2)\mapsto\coh$ given in Lemma \ref{L:McohomEx}.
 Excision gives
  canonical isomorphisms
  $$H^0(*_\Su;\C/2\Z)\cong
  H^0(*_\Su\cup*_j,*_j;\C/2\Z)\cong
  H^0(*_\Su\cup*_1\cup*_2,*_1\cup*_2;\C/2\Z).$$ With the sequences of
  the triples $(M_j\setminus T_j,*_\Su\cup*_j,*_j)$ and $(M\setminus
  T,*_\Su\cup*_1\cup*_2,*_1\cup*_2)$ it gives maps
  $\delta_j:H^0(*_\Su;\C/2\Z)\to H^1(M_j\setminus T_j,*_j\cup
  *_\Su;\C/2\Z)$ and $\delta:H^0(*_\Su;\C/2\Z)\to H^1(M\setminus
  T,*_\Su\cup*_1\cup*_2;\C/2\Z)$.  In particular, the long exact
  sequence with $\delta$ gives
  \begin{displaymath}
    \begin{array}{rcl}
      0\longrightarrow H^0(*_\Su)\stackrel{\delta}{\longrightarrow}
      H^1(M\setminus T,*_\Su\cup*_1\cup*_2)&\stackrel{}{\longrightarrow}&
      H^1(M\setminus T,*_1\cup*_2)\longrightarrow \cdots\\
      \wt\coh\quad&\mapsto&\quad\coh
    \end{array}
  \end{displaymath}
  Hence two compatible cohomology classes $\coh_1,\coh'_1$ on $M_1$
  (respectively $\coh_2,\coh'_2$ on $M_2$) lead to the same cohomology
  class $\coh$ on $M_2\circ M_1$ if and only if $\coh'_2=\coh_2+\delta_2(\vp)$
  and $\coh'_1=\coh_1+\delta_1(\vp)$ for the same $\vp\in
  H^0(*_\Su;\C/2\Z)$.
\end{rem}

As we now explain, decorated surfaces and cobordisms form a tensor
category $\Cob$.  The objects of $\Cob$ are decorated surfaces and
morphisms are 
diffeomorphism classes
of decorated cobordisms. 
 The composition is given in Definition \ref{D:Compostion}.  The tensor
product of $\Cob$ is given by the disjoint union: if
$\ds{\Su}_j=(\Su_j,\{p_i^j\},\coh_j,{\La}_j)$, for $ j=1,2$, are two
decorated surfaces then define $\ds{\Su}_1\otimes \ds \Su_2 $ as
$(\Su_1\sqcup \Su_2 ,\{p_i^1\} \sqcup \{p_i^2\},\coh_1\oplus
\coh_2,{\La}_1\oplus {\La}_2)$.
Here the ordering of the components of $\Su_1\sqcup \Su_2$ is obtained by
putting those of $\Su_1$ first then those of $\Su_2$.  In particular,
$\Su_1\sqcup \Su_2\neq\Su_2\sqcup\Su_1$ but they are isomorphic.
Similarly, if $\ds
M_j=(M_j,T_j,f_j,\coh_j,n_j), $ for $ j=1,2$, are two cobordisms then
define $\ds{M}_1\otimes \ds M_2 $ as $(M_1\sqcup M_2 ,T_1\sqcup
T_2
,f_1\sqcup f_2,\coh_1\oplus \coh_2,n_1+ n_2)$ (often, for the sake of clarity we will write $M_1\sqcup M_2$ instead of $M_1\otimes M_2$).  The composition of
cobordisms is given in Definition \ref{D:Compostion}.  It is known
that such a composition is associative: the integers decorating
3-manifolds can be interpreted as the signature of a 4-manifold
bounding them.  The correcting term in the gluing of these manifolds
reflects the defect of additivity of the signature by Wall's formula
(see \cite[Theorem IV.4.1.1]{Tu}).  The associativity
of gluing of 4-manifolds ensures that the above definition gives an
associative composition law.  For a direct algebraic proof of this
associativity see \cite{Tu} (in particular, Lemma IV.3.7).

In the following paragraphs, we detail two special classes of morphisms of particular interest.
\subsubsection{Mapping cylinders}\label{rem:mcgrepres}
Let  $\Su=(\Su,\{p_i\},\coh)$ be the data of a decorated surface without a Lagrangian.  Then given a Lagrangian $\La\subset H_1(\Su,\R)$, we denote $\Su_{\La}$ as the decorated surface $(\Su,\{p_i\},\coh, \La)$.  The cylinder $L_{\La,\La'}:=(\Su\times[0,1],\{p_i\}\times[0,1],\Id\times\{0\}\sqcup\Id\times\{1\}  ,\pi^*(\coh),0)$, where $\pi:\Su\times[0,1]\to\Su$ is the projection, defines a canonical isomorphism $\Su_{\La}\simeq\Su_{\La'}$. 

  If $f:\Su_1\to\Su_2$ is a diffeomorphism of decorated surfaces (in
  particular, $f_*(\La_1)=\La_2$),
   the mapping cylinder of $f$ is the decorated cobordism
  from $\Su_1$ to $\Su_2$ given by
  $M_f=(\Su_2\times[0,1],\{p_i^2\}\times[0,1],f\times\{0\}\sqcup
  \Id\times\{1\},\pi^*(\coh_2),0)$.
  The mapping cylinder construction is functorial: $M_f\circ M_g=M_{fg}$.
  The diffeomorphism class of $M_f$ only depends of $f$ up to isotopy;
  in particular, the mapping cylinder construction gives a
  group morphism from the Torelli group of $\bp{\Su,\{p_i\}}$ to the
  automorphisms of the decorated surface $\bp{\Su,\{p_i\},\coh,\La}$.
\subsubsection{Degree cylinders}\label{R:CylDef}
In Section \ref{S:PropTQFT}, we will use the following  special cylinders to define a degree on our TQFT.  For  a decorated
surface $\Su$ and $\vp\in H^0(\Su,\C/2\Z)$, we define
\begin{equation}
  \label{eq:IdH0}
  \Id_{\Su}^\vp=(\Su\times[0,1],\{p_i\}\times[0,1],\Id\times\{0\}
  \sqcup\Id\times\{1\},\pi^*(\coh)+\partial\vp,0):\Su\to\Su
\end{equation}
where the cohomology class $\partial \vp$ is the unique
cohomology class which is $0$ on cycles and whose value on the class
of the relative path $\{*\}\times[0, 1]$ is $\vp(*)$.  
More
precisely, let
$*_i$ be the set of base-points on $\Su\times \{i\}$.
Then for cohomology with coefficients in
$\C/2\Z$ we have:
$$H^0(\Su)=H^0(*_1)=H^0(*_0\sqcup*_1,*_0)\stackrel\partial\longrightarrow 
H^1(\Su\times[0,1]\setminus\{p_i\}\times[0,1],*_0\sqcup*_1)$$ where the map
$\partial$ is part of the long exact sequence associated to the triple
$(\Su\times [0,1],*_0\sqcup*_1,*_0)$.
Remark that as $\partial \vp$ is zero on meridians of
$\{p_i\}\times[0,1]$ 
then 
$\pi^*(\coh)+\partial\vp$ is a compatible
cohomology class.  Furthermore, the mapping 
$\vp\mapsto\Id_{\Su}^\vp$ is a group morphism 
$H^0(\Su,\C/2\Z)\to\operatorname{Aut}_{\Cob}(\Su)$ i.e. 
$\Id_\Su^0=\Id_\Su$ and $\Id_\Su^\vp\circ\Id_\Su^{\vp'}=\Id_\Su^{\vp+\vp'}$.

\subsection{The TQFT functor $V$}\label{SS:TQFTfunctor}
Let $\ds \Su$ be a decorated surface.  Define
 $$\cV(\ds \Su)=\text{Span}_\C\{\text{decorated cobordism from $\emptyset$ 
   to $\ds \Su$}\}$$ 
 $$\cV'(\ds \Su)=\text{Span}_\C\{\text{decorated cobordism from $\ds \Su$ 
   to $\emptyset$}\}.$$ 
  Clearly $\cV(\ds \Su)\subset \cV'(\wb \Su)$ and we have the pairing:
 $$\brk{\cdot,\cdot}_{\cV(\ds \Su)}: \cV'(\ds \Su)\times \cV(\ds \Su)\to 
 \C \text{ given by } (M_1,M_2)\mapsto \Zr({M}_1\circ M_2).$$ 
 Let $\V(\ds \Su)$ be the vector space $\cV(\ds \Su)$ modulo the right kernel
 of $\brk{\cdot,\cdot}_{\cV(\ds \Su)}$.  
  Similarly, let $\V'(\ds \Su)$
 be the vector space $\cV'(\ds \Su)$ modulo the left kernel of
 $\brk{\cdot,\cdot}_{\cV(\ds \Su)}$.
 If $W\in \cV(\ds \Su)$ then we will represent it image in  $\V(\ds \Su)$ by $[W]$.   

 The pairing $\brk{\cdot,\cdot}_{\cV(\ds \Su)}$ induces a non degenerate
 pairing $\brk{\cdot,\cdot}_{\ds \Su}:\V'(\ds \Su)\times \V(\ds \Su)
 \to \C$.  The results above imply that the following assignment determine
 a functor $\V$:
$$\{\text{decorated surface}\} \to \{\text{vector spaces}\},  \; \; \Su
\mapsto \V(\Su),$$ 
\begin{align*}
  \{\text{decorated cobordisms}\} \to \{\text{linear maps}\}, \; \; \ds M
  \mapsto \V(\ds M)
\end{align*}
where $ \V(\ds M): \V(\partial_- \ds M) \to \V(\partial_+ \ds M)$ is
given on the class $[W]$ of $W\in\cV(\partial_- \ds M)$ by
$\V(\ds M)([W])=\V(M \circ  W)$. 
\begin{rem}
Although the above definitions of $\cV$ and $\cV'$ seem symmetric they are not because of the very last condition in the definition of cobordisms (Definition \ref{D:Cobordisms}). More explicitly, the manifolds in $\cV$ should all be admissible but the manifolds in $\cV'$ are not necessarily admissible.  This is one of the new features of the theory introduced in the present paper.
\end{rem}

\subsection{$\V$ is lax monoidal}\label{SS:Vdunion}
Recall that $\Cob$ is the category of decorated cobordisms defined in
Sections \ref{S:deco2} and \ref{S:deco3}.  
 The category $\Cob$ is a strict symmetric monoidal category where the tensor product is given by the disjoint union, the unit is the empty decorated cobordism and the symmetric braiding is given by the 
trivial cobordism
$\tau_\sqcup:\Su_1\sqcup\Su_2\to\Su_2\sqcup\Su_1$.  
 We have the following:
\begin{prop}\label{P:dunion}
  For any decorated surfaces $\Su_1,\Su_2$, the disjoint union induces a
  well defined injective map
  \begin{equation}
    \label{eq:dunion}
    \begin{array}{rcl}
      \cdot\sqcup\cdot\,:\,\V(\Su_1)\otimes\V(\Su_2)&\to&\V(\Su_1\sqcup\Su_2)
      \\{}
      [M_1]\otimes [M_2]&\mapsto&[M_1\sqcup M_2].
    \end{array}
  \end{equation}
A similar statement holds for $\V'$.  
\end{prop}
\begin{proof} Let $M_i\in\cV(\Su_i)$. If $[M_1]=0$ then for any
  $N\in\cV'(\Su_1\sqcup\Su_2)$ we have $\brk{N,M_1\sqcup M_2}=\brk{N_1, M_1}=0$ where $N_1=N\circ (\Id_{\Su_1}\sqcup  M_2)\in\cV'(\Su_1)$.  
Similarly, if $[M_2]=0$ then $[M_1\sqcup M_2]=0$.  This prove that $\cdot\sqcup\cdot\,:\,\V(\Su_1)\otimes\V(\Su_2)\to\V(\Su_1\sqcup\Su_2)$  is well defined.

  To prove injectivity, remark that the restriction of
  $\brk{\cdot,\cdot}_{\Su_1\sqcup\Su_2}$ to
  $(\V(\Su_1)\sqcup\V(\Su_2))\otimes(\V'(\Su_1)\sqcup\V'(\Su_2))$ is
 a non-degenerate pairing given by the product of the two non-degenerate pairings  $\brk{\cdot,\cdot}_{\Su_1}$ and $\brk{\cdot,\cdot}_{\Su_2}$.
 Thus, an element in the kernel of $\sqcup$ is in
  the right kernel of $\brk{\cdot,\cdot}_{\Su_1}\brk{\cdot,\cdot}_{\Su_2}$
  and so is zero.
\end{proof}

It is clear that 
$\sqcup$ is a natural transformation satisfying the strict associativity constraints and the unitality constraints associated to the map $\C\stackrel{\sim}\to\V(\emptyset)$.
  Nevertheless, 
 $  \cdot\sqcup\cdot$ in Equation \eqref{eq:dunion} is not a natural isomorphism  and  in general,
$\V(\Su_1)\otimes\V(\Su_2)\not\simeq\V(\Su_1\sqcup\Su_2)$ (for
example, see Proposition \ref{prop:vspheres}(2)). 
Thus $\V$ fails to be a (strong) monoidal functor :
these kind of functors are sometimes called lax monoidal.  

\begin{defi}\label{D:EquivCob}
  Let $M_1$ and $M_2$ be linear combinations of decorated cobordisms.  We
  will say that $M_1$ is \emph{equivalent} to $M_2$ and will write $M_1\equiv M_2$ if
  for any decorated cobordism $M'$, one has $\V(M_1\sqcup M')=\V(M_2\sqcup
  M')$ as linear maps.
\end{defi}
Clearly, $M_1\equiv M_2\implies \V(M_1)=\V(M_2)$ but as $\V$ is not monoidal,
the converse is not true.  Proposition \ref{prop:vspheres}(2) will show that for some decorated surfaces $\hS_{\pm k}$, one has
$\V(\Id_{\hS_{k}})=0$ because $\V({\hS_{k}})=\{0\}$ but $\Id_{\hS_{k}}\not\equiv
0$ as $\V(\Id_{\hS_{k}}\sqcup\Id_{\hS_{-k}})\neq0$.  
\begin{defi}\label{D:admS}
  Let $\Su$ be a decorated surface.  We say that $\Su$ is admissible,
  if each of its connected component $\Su_i$ satisfies at least one of
  the conditions: 1) the restriction of the cohomology class to
  $\Su_i$ is non-integral or 2) $\Su_i$ has a point marked with a
  projective module of $\cat$.
\end{defi}
\begin{rem}\label{rem:admsurfprimenonprime}
If $\Su_-$ is admissible then
$\cV'(\wb\Su_-)\simeq\cV(\Su_-)$.  Furthermore, if $M:\Su_-\to\Su_+$
is any decorated cobordism with $\Su_-$ admissible, then $M$ can be
seen as an element of $M_\cV\in\cV(\Su_+\sqcup\wb{\Su_-})$.  Indeed $M_\cV:\emptyset\to\Su_+\sqcup\wb{\Su_-}$ satisfies the
last axiom of Definition~\ref{D:Cobordisms} since the components of
$M$ that intersect $f(\wb\Su_-)$ are automatically admissible as
$\Su_-$ is admissible.
\end{rem}

\begin{lemma}\label{L:cob-as-vect}
  Let $\Su_-$ be an admissible surface and let $M:\Su_-\to\Su_+$ be a
  (linear combination of) decorated cobordism.  Let $M_\cV$ be the
  element of $\cV(\Su_+\sqcup\wb{\Su_-})$ corresponding to $M$,
  described 
  above.
  Then $$[M_\cV]=0\in\V(\Su_+\sqcup\wb{\Su_-})\implies M\equiv0.$$  %
\end{lemma}
\begin{proof}
  The fact that $\Su_-$ is an admissible surface implies that the
  cylinder $\Id_{\Su_-}=\Su_-\times[0,1]$ 
  can be also seen as a decorated cobordism $C$ from
  $\emptyset$ to $\Su_-\sqcup\wb{\Su_-}$.  The same cylinder can also
  be seen as an element $C':\wb{\Su_-}\sqcup\Su_-\to\emptyset$ and it holds $$\V\left((\Id_{\Su_-}\sqcup C')\circ (C\sqcup \Id_{\Su_-})\right)=\V(\Id_{\Su_-}).$$ 
  Then
  $M_\cV=(M\sqcup\Id_{\wb{\Su_-}})\circ
  C\in\cV(\Su_+\sqcup\wb{\Su_-})$.  Now assume that $[M_\cV]=0$ and
  let $M':\Su'_-\to\Su'_+$ be any decorated cobordism, then for any
  $N\in\cV(\Su_-\sqcup\Su'_-)$, we have $\V(M\sqcup M')([N])=[(M\sqcup
  M')\circ N]=\V(\Id_{\Su_+}\sqcup\, C'\sqcup M')([M_\cV]\sqcup[N])=0$
\end{proof}

\section{Properties of the  TQFT functor $V$}\label{S:PropTQFT}
This section contains the main technical results concerning the functor $\V$ and the contravariant functor $\V'$. 
Before getting to the details, let us give an overview of the  main ideas and results of the next subsections. 

In the standard Reshetikhin-Turaev TQFT, one of the main ideas is that, even if the vector space associated to a surface is 
formally spanned by all the manifolds bounding it, one can actually restrict only to one such manifold at the cost of allowing links inside the manifold. 
This leads directly to the definition of the Kauffman skein module, which then surjects onto the standard TQFT vector spaces. 

Something similar happens in our case. Subsection \ref{SS:Skein} is
devoted to formalize the notion of \emph{skein module} of a fixed
manifold and showing that for a \emph{connected} surface $\Su$, the skein
module $\Skein(M,\Su)$ of a connected manifold bounding $\Su$
surjects over $\V(\Su)$.  The technicalities in this subsection are
due to the fact that our cobordisms are not just manifolds but
decorated ones, so roughly speaking $\Skein(M,\Su)$ is a suitable
quotient of the span of all the decorated cobordisms (possibly
containing ribbon graphs, and with any compatible cohomology class)
bounding $\Su$ whose underlying $3$-manifold is $M$.  The quotient is
generated by local relations (as in the standard case) and a new kind
of relations of non-local nature which we called $\sigma$-equivalence.
In particular we show that the skein module of a handlebody is finite
dimensional. Also, pay attention to the fact that  in order to apply
a skein relation in a ball, one must make sure that the complement 
of the ball is admissible according to Definition \ref{D:Admissiblethreeuple} (see Remark \ref{rem:skeinlocal}).

Subsection \ref{subsec:1surg} shows that if two manifolds are related
by a connected sum \emph{of a particular kind} (taking into account
the ribbon graphs in the manifolds) then their images by $\V$ are
proportional. The limit of this statement is due to its hypotheses and
it is the technical reason why many difficulties arise later on:
essential spheres in manifolds cannot be cut and filled by balls
unless some conditions are satisfied.  On contrast, Subsection
\ref{subsec:2surg} is devoted to show that if two manifolds are
related by 2-surgeries then their skein modules generate the same
space in TQFT (see Proposition \ref{P:EqualSkien}).

Then in Subsection \ref{sub:findimconnected} we show that if $\Su$ is connected then $\V(\Su)$ is finite dimensional. The guiding idea is similar to the standard one: a manifold bounding a connected surface is related by $2$ surgeries to a handlebody and so its class in $\V(\Su)$ equals the class of a certain skein in the handlebody. But since the skein module of a handlebody is finite dimensional also $\V(\Su)$ is. 
A key point in the above argument is that we could use only $2$-surgeries: this accounts for the hypothesis $\Su$ being connected. 
Showing finite dimensionality of $\V(\Su)$ without this hypothesis will require more work. 

In Subsection \ref{SS:SphereS_k} we focus on some important examples: when $\Su$ is a sphere three points $p_i$ whose colors are $V_0,V_0$ and $\sigma^n$ and the class $\coh$ determined with these colors. We call these decorated surfaces $\hS_n$. They play a key role in our constructions as they provide the first examples showing that $\V$ is not monoidal: indeed we show that $\V(\hS_k)=0$ (if $k\neq 0$) but $\V(\hS_k\sqcup \hS_{-k})$ is one dimensional. 
As stated above, the origin of this is the fact that ``cutting along essential spheres'' is not always allowed, and this is the simplest example of this phenomenon. Actually we show that in general this cutting problem reduces exactly to the presence of some sphere $\hS_k$ embedded in a cobordism. 
In this section we also introduce natural cobordisms which we call ``pants'' $\Pa_{k,l}^{k+l}:\V(\hS_k\sqcup \hS_l)\to \V(\hS_{k+l})$ and similar morphisms ``upside down'' $\Pa_{k+l}^{k,l}$. 

Subsection \ref{SS:FiniteDimVdiscon} contains the main result of this section, Theorem \ref{teo:findimvv} which roughly states that for a disconnected surface $\Su=\Su_1\sqcup \ldots \sqcup \Su_n$ every element of $\V(\Su)$ is the image of some elements not of $\otimes_i \V(\Su_i)$ but of $\otimes_i\V(\Su_i \sqcup \hS_{k_i})$ for some finite set of integers $k_i$. This, together with a direct analysis of $\V(\Su_i\sqcup \hS_{k_i})$ allows to prove that $\V(\Su)$ is finite dimensional in general. The key idea here is that if a manifold bounding $\Su$ cannot be reduced via surgeries to a disjoint union of manifolds bounding the components $\Su_i$ separately, this is due to the presence of some essential sphere $\hS_{k}$: instead of cutting and filling it, we then just cut along it, and get a manifold in $\V(\Su\sqcup \hS_k\sqcup \overline{\hS_k})$. 

The final Subsection \ref{sub:grading} prepares to the next section. Its goal is to use the degree cylinders defined in Subsection \ref{R:CylDef} to let $H^0(\Su;\C/2/2\Z)$ act on $\V(\Su)$. The spectral decomposition of this action yields a grading which turns out to be integer-valued. This grading allows to re-understand the non-monoidality of $\V$: the image of the map $\sqcup :\V(\Su_1)\otimes \V(\Su_2)\to \V(\Su_1\sqcup \Su_2)$ is exactly the degree $0$ part of $\V(\Su_1\sqcup \Su_2)$, but in general non-zero degree parts may exist.

\subsection{Skein modules}\label{SS:Skein}
The Reshetikhin-Turaev functor $F$ assigns morphisms in $\cat$ to
$\cat$-colored ribbon graphs in $\R^2\times[0,1]$.  A linear
combination of graphs whose value by this functor is zero defines a
\emph{skein relation}.  
We are going to consider skein modules which are the span of all graphs in a fixed manifold modulo certain skein relations.

Let $\Su$ be a decorated surface.  Let  $M$ and $M'$ be oriented 3-manifolds
(with no decoration) which bound $\Su$ and $\wb{\Su}$,  respectively.  Let $\cV(M,\Su)$ be the subspace of
$\cV(\Su)$  generated by decorated
cobordims whose underlying 3-manifold is $M$.  Similarly, 
let $\cV'(M',\Su)$ be the subspace of
$\cV(\Su)$ generated by decorated
cobordims whose underlying 3-manifold is  $M'$.  

Let $\wbU=((U_1,\ve_1),\ldots,(U_n,\ve_n))$ be a sequence of
homogeneous modules $U_k\in\cat$ and signs $\ve_k\in\{\pm1\}$ such that $F(\wbU)\in\cat_{\bar0}$.  Let
$S^2_{\wbU}$ be the sphere $S^2\simeq\C\cup\{\oo\}$  
with framed dots $\{p_k\}_{k=1,...,n}$ where $p_k$ is colored by $(U_k,\ve_k)$ and the dot $p_k\in S^2\simeq\C\cup\{\oo\}$ is positioned at the point $k\in \C$ and the framing points in the positive imaginary direction.
Since $F(\wbU)\in\cat_{\bar0}$ then $S^2_{\wbU}$ admits a unique
compatible cohomology class in $H^1(S^2\setminus\{p_i\},\C/2\Z)$.
With the zero Lagrangian, $S^2_{\wbU}=(S^2, \{p_k\},\coh,\{0\} )$ is a
decorated surface.
 
Choose an embedding $f:B^3\to \R^2\times[0,1]$ identifying
$\R^+\subset\C\subset S^2$ with $\R^+\times \{0\}\times
\{1\}\subset\R\times \R\times[0,1]$.
 If $\rho\in \cV'(B^3,S^2_{\wbU})$ then $f(\rho)$ is a $\cat$-colored ribbon graph with $n$ strands meeting the boundary in $\R^+\times \{0\}\times \{0\}$.
 Consider the following map:
\begin{equation}\label{E:DefcF}\cF:\cV'(B^3,S^2_{\wbU})\to\Hom_\cat(F(\wbU),\unit) \; \text{ given by } \; \rho \mapsto F(f(\rho)).
\end{equation}
Let $R_{\wbU}$ be the kernel of this map.  The module $R_{\wbU}$ does
not depend of $f$.  

Let $R$ be the span in $\cV(M,\Su)$ of the
elements of the form 
$(\rho\sqcup\Id_\Su)\circ M'$ 
where $M'$ is a decorated
cobordism from $\emptyset$ to $S^2_{\wbU}\sqcup\Su$,   $\wbU$ is a tuple and
$\rho\in R_{\wbU}$ such that 
$(\rho\sqcup\Id_\Su)\circ M'\in  \cV(M,\Su)$.
 We say that $x,y\in \cV(M,\Su)$ are {\em skein
  equivalent} if they differ by an element of $R$.
  
\begin{rem}\label{rem:skeinlocal}
 Here topologically $M'=M\setminus B^3$.  A simple way to construct skein equivalent elements is to fill this $B^3$ in two ways with two elements of $\cV'(B^3,S^2_{\wbU})$ whose difference is in $R_{\wbU}$.  
 Also, for $M'$ to be a decorated cobordism it must be admissible (as in the last point of Definition \ref{D:Cobordisms}). 
This implies that before applying a skein relation one must always make sure that \emph{the complement of the ball where the move is applied is admissible}.
\end{rem}

Let $N\in \cV(M,\Su)$ such that  $T=T'\cup K$ where $K$ is a framed oriented knot
colored by $\sigma$.  Let $N'$ be the same decorated cobordism with
$K$ removed (remark that the cohomology class $\coh$ of $N$ is zero on
the meridian of $K$ thus $\coh$ induces a cohomology class on $N'$).
Let $R'$ be the span 
 in  $\cV(M,\Su)$ of elements $N'-(-1)^{r-1}q^{-2r'\omega(K_P)}N$
where $K_P$ is a parallel copy of $K$ and 
$N$ and $N'$ are any decorated cobordism as described above.  
We say that $x,y\in \cV(M,\Su)$ are {\em
  $\sigma$-equivalent} if they differ by an element of $R'$.

Extending $\V$ by linearity we have a map $\V:\cV(M,\Su)\to\V(\Su)$.
\begin{prop}\label{P:RR'kernel}
  The submodules $R$ and $R'$ are in the kernel of $\V$.
\end{prop}
\begin{proof}
Equation \eqref{eq:Zdefi} implies we can assume that $M$ is connected.  
 Let us prove that $R$ is in the kernel.
First, consider an element $X$ in $R$ of the form $(\rho\sqcup\Id_\Su)\circ M'$ where $\rho\in R_{\wbU}$ and $M'\in\cV(\emptyset,S^2_{\wbU}\sqcup\Su)$ are as described above.  We will show that for any $M''\in \cV'(\Su)$ it holds $\Zr(M''\circ X)=0$.  

Here $M''\circ X=\rho \circ N$ where $N=(\Id_{S^2_{\wbU}}\sqcup M'')\circ M'$ is an admissible cobordism.  So it suffices to compute $\Zr(\rho \circ N,T,\coh)$ where $T$ and $\coh$ are the  $\cat$-colored ribbon graph and cohomology class of the cobordism $\rho \circ N$.  The graph $T= f(\rho)\circ T''$ where  $T''=T\cap \rho$.  
Since $\rho$ is a ball 
then there exists a surgery presentation $L\subset S^3$ of $(\rho \circ N,T,\coh)$ such that  $T'' \cup L$ is in the lower hemisphere of $S^3$ and ${f(\rho)}$ is in the upper hemisphere of $S^3$.  We will assume that $\coh$ is non-integral, otherwise a similar argument works after modifying the triple $(\rho \circ N,T,\coh)$ as discussed after   Definition \ref{D:Admissiblethreeuple}.  Following the proof of Proposition  1.5 in  \cite{CGP} we can slide a generically colored edge over the components of $L$ to make it computable.  Since $f(\rho)$ contains no surgery components this can be done completely in the lower hemisphere of $S^3$ and thus not touching $f(\rho)$.  

To show $\Zr(M''\circ X)=0$ it is enough to see that $F'(L\cup T'')=0$.  
To compute $F'(L\cup T'')$ we choose a (1,1)-ribbon graph $T_{V_\alpha}$ obtained by cutting an edge $e$ of $L\cup T''$ 
colored by $V_\alpha$ for some $\alpha\in \Cp$.
   Then since $f(\rho)$ is in the upper hemisphere of $S^3$ we have 
$$F(T_{V_\alpha})= \left(\Id_{V_\alpha} \otimes F\left(f(\rho)\right)\right)\circ g$$
where $g$ is a morphism from $V_\alpha$ to $V_\alpha\otimes F(\wbU)$.
But $F(f(\rho))=0$ as $\rho \in R_\wbU$ implying $F(T_{V_\alpha})=0$.  Thus, $F'(L\cup T')$ and $\Zr(M''\circ X)$ are both zero.  Finally, since a general element of $R$ is a sum of elements of the form considered above and $\V$ is linear we have that $R$ is in the kernel of $\V$.

Next, let us prove that $R'$ is in the kernel.  
Let $N\in \cV(M,\Su)$ be such that  $T=T'\cup K$ where $K$ is a framed oriented knot
colored by $\sigma$, and $N'$ be the same decorated cobordism with
$K$ removed. For any $M''\in \cV'(\Su)$ we will show that $\Zr(M''\circ N')=(-1)^{r-1}q^{-2r'\coh(K_P)}\Zr(M''\circ N)$, where $K_P$ is a parallel of $K$. If $V\in \cat_{\wb{\alpha}}$ then we have
\begin{equation}
  \label{eq:sigmabraid}
  F\left( \put(5,17){$\sigma$}\put(18,17){$V$}\epsh{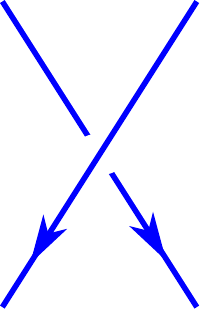}{9ex}\right)=
  q^{2r' \wb{\alpha}}F\left(\put(5,17){$\sigma$}\put(18,17){$V$}\epsh{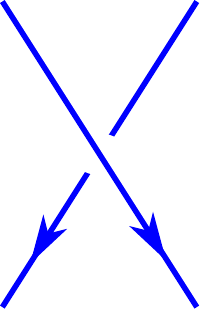}{9ex}
  \right),\ \ F\left( \put(8,10){$\sigma$}\epsh{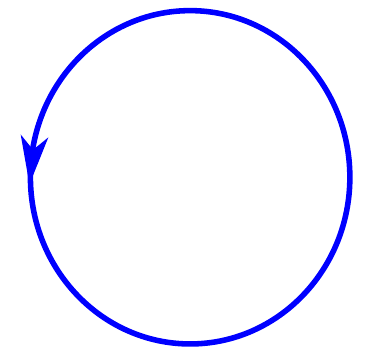}{9ex} \right)=(-1)^{r-1}.
\end{equation}
So unlinking $K$ from the surgery diagram of $M''\circ N$ and then removing it the resulting invariant changes by an overall factor $(-1)^{r-1}q^{2r'\coh(K_P)}$: indeed one can check that the value $\coh(K_P)$ is the ``$\C/2\Z$-valued linking number of $K$ with the surgery link,'' for a detailed proof, see  \cite{CGP}, Lemma 4.7. This concludes as the result is a surgery diagram of $M''\circ N'$.
\end{proof}

\begin{defi}[Admissible skein modules]
  A $\cat$-colored ribbon graph in a $3$-manifold is \emph{admissible} if at
  least one of its edges is colored by a projective module.  
 If $M$ is connected, let $\cV_A(M,\Su)$ be the subspace of $\cV(M,\Su)$ generated by cobordisms $(M,T,\coh)$ where $T$ is admissible.  In general, if $M=\bigsqcup_i M_i$ where $M_i$ are the connected components of $M$ then let  $\cV_A(M,\Su)=\bigsqcup_i\cV_A(M_i,M_i\cap\Su)$.  Define $\Skein(M,\Su)$ as the \emph{admissible skein module}:
  $$\Skein(M,\Su)=\cV_A(M,\Su)/(\cV_A(M,\Su)\cap (R+R')).$$
  Proposition \ref{P:RR'kernel} implies that 
  $\V$ 
  induces a map
  \begin{equation}\label{E:DefV}
  \V:\Skein(M,\Su)\to\V(\Su).
  \end{equation}
\end{defi}
Next, we define a map $\cF^\#$ which is kind of dual to the map $\cF$ defined in Equation~\eqref{E:DefcF}.  Let $(S^2_{\wbU}, \{p_k\},\coh, \La)$ be the decorated surface defined as above. 
Choose an embedding $g:B^3\to \R^2\times[0,1]$ identifying $\R^+\subset\C\subset S^2$ with $\R^+\times \{0\}\times \{1\}\subset\R\times \R\times[0,1]$.
If $\rho\in \cV(B^3,S^2_{\wbU})$ then $g(\rho)$ is a $\cat$-colored ribbon graph.  Consider the following map:
\begin{equation}\label{E:DefcFdual}
\cF^\#:\cV(B^3,S^2_{\wbU})\to\Hom_\cat(\unit,F(\wbU)) \; \text{ given by } \; \rho \mapsto F(g(\rho)).
\end{equation}

\begin{lemma}\label{lem:skeinhom}
  If $F(\wbU)$ is projective in $\cat$, then $\cF^\#$ induces an isomorphism:
  $$\Skein(B^3,S^2_{\wbU}) \to \Hom_\cat(\unit,F(\wbU)).$$
\end{lemma}
\begin{proof}
 First, let us show that $\cF^\#|_{\cV_A(B^3,S^2_{\wbU})}$ is surjective.  
 To do this consider the ribbon  graph $\Gamma_{\wbU}:\wbU\to\wbU$ in $\R^2\times[0,1]$ composed of two coupons colored by the identity of $F(\wbU)$ related by a  straight strand colored by $F(\wbU)$: 
 $$\text{schematically for }n=4,\qquad \Gamma_{\wbU}=\epsh{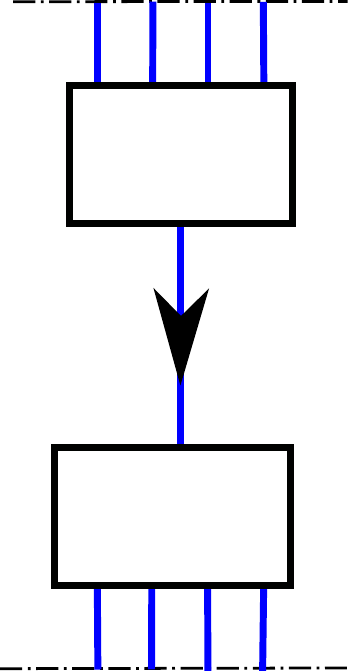}{8ex}\put(-5,0){\ms{F(\wbU)}}\put(-12,10){\ms{\Id}}\put(-12,-9){\ms{\Id}}.$$
  The Reshetikhin-Turaev  functor sends this graph to the identity of $F(\wbU)$.  
  Let $G\in \Hom_\cat(\unit,F(\wbU))$. 
The $\cat$-colored ribbon graph $\Gamma_{\wbU}\circ g(G)$
is admissible because it contains an edge colored by the  projective module $F(\wbU)$.  Then $(B^3,\Gamma_{\wbU}\circ g(G), (g|_{S^2})^{-1}, 0, 0)$ is an element of $\cV_A(B^3,S^2_{\wbU})$ which  is mapped to $G$ under $\cF^\#$. 

  Next we will show that the kernel of the morphism $\cF^\#$ is exactly $R+R'$.  
  The module $R\subset\cV(B^3,S^2_{\wbU})$ is generated by elements of the form $(r\sqcup\Id_{S^2_{\wbU}})\circ M'$ where $r\in R_{V}$,
  $M'\in\cV(S^2\times[0,1],S^2_{V}\sqcup S^2_{\wbU})$ and 
  $V=((V_1,\ve'_1),\ldots,(V_n,\ve'_{n'}))$ is a sequence of
  homogeneous modules $V_i\in\cat$ with signs $\ve'_i\in\{\pm1\}$.
  Clearly $\cF^\#$ sends such elements to morphisms of 
  $\Hom_\cat(\unit,F(\wbU))$ that factor through the zero morphism of
  $\Hom_\cat(\unit,F(\wb V))$ where $\wb V=((V_1,-\ve'_1),\ldots,(V_n,-\ve'_{n'}))$.  Thus, $R$ is  contained in the  kernel of $\cF^\#$.

 Next, let $N$ be in the kernel of $\cF^\#$ then $N$ 
 can be represented by a linear combination $T$ of admissible $\cat$-colored ribbon graphs from $\emptyset$ to $\wbU$
  in $\R^2\times[0,1]$.  
  Since $F(\Gamma_{\wbU})=\Id_{F(U)}$ and the graphs in $T$ are admissible we have that $T$ is skein equivalent to $T'=\Gamma_{\wbU}\circ T$, see Remark \ref{rem:skeinlocal}.  Let $\wb \wbU=((U_1,-\ve_1),\ldots,(U_n,-\ve_n))$.  
   Now we can represent the element $T'\in
  \cV(B^3,S^2_{\wbU})$ as the composition of two cobordisms: the
  first goes from $\emptyset$ to $S^2_{\wbU}\sqcup S^2_{\wb{\wbU}}$
  and is topologically $S^2\times [0,1]$ containing $\Gamma_{\wbU}$;
  the second is the cobordism $S^2_{\wb{\wbU}}\to \emptyset$ which is a
  ball containing the dual $T^*\in R_{\wb{\wbU}}$ of the ribbon graph
  representing $T$.
  Thus, $T'\in R$ and $T-T'\in R$ so $T\in R$.  (Notice that here we
  used the fact that $F(\wbU)$ is projective in order to state that
  the first cobordism is admissible!). To conclude observe that since
  we are working in $B^3$ the relations in $R'$ are contained in those
  of $R$.
\end{proof}

\begin{prop}\label{P:Skein-finite} 
  Let $\Su_g$ be a connected decorated surface of genus $g$ and let
  $H_g$ be a handlebody with boundary $\Su_g$.  Then
  $\Skein(H_g,\Su_g)$ is a finite dimensional complex vector space.
\end{prop}
  \newcommand{\pro}{W}
\begin{proof}
We will describe a cell decomposition of $H_g$ into three 3-balls $B_0, B_1$ and $B_2$ obtained  by cutting it along $g+2$ disks $\{D_k\}_{k=0}^{g+1}$ with $\partial D_k\subset  \Sigma_g$.
  \begin{figure}[!]  
    \centering
     \(\epsh{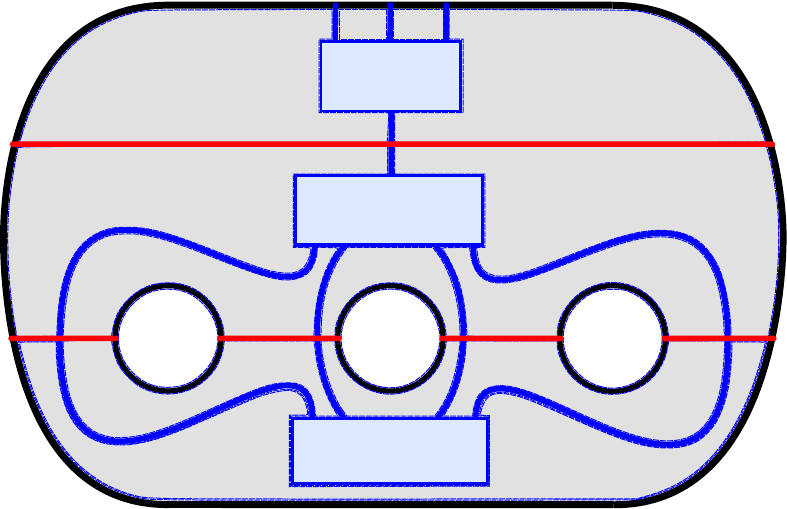}{20ex}\put(-2,30){$B_0$} 
    \put(5,0){$B_1$}
     \put(-2,-30){$B_2$}
     \)
    \caption{
    A slice
      of a handlebody bounding a genus 3 surface
      with 3 framed dots containing $\Gamma$. 
      The five cutting discs are represented as horizontal lines in red.   The dual graph is in blue.  Here a slice of the three balls $B_0, B_1$ and $B_2$ are represented in gray. }
    \label{F:Handlebody}
  \end{figure} 
  The disc $D_0$ obtained by pushing interior of  a disc in $\Su$ containing all the points $\{p_i\}_{i=1}^n$ into $H_g$.  The ball $B_0$ is formed by cutting along the disc $D_0$.  Hence $B_0$ contains the points $\{p_i\}_{i=1}^n$.  Cutting $B_0$ out of $H_g$ yields a handlebody of genus $g$, which we decompose into the union of two $3$-balls $B_1,B_2$ glued to each other along disjoint disks $D_1,\ldots D_g$ in their boundaries. 
Consider an uncolored ribbon graph $\Gamma$ which is dual to this cell decomposition described as follows.  
 Each of the 3-balls $B_0, B_1$ and $B_2$ contains a coupon with $n+1$, $g+2$ and $g+1$ legs, respectively.  The $n$ points $\{p_i\}$ are connected to $n$ of the legs in the coupon in $B_0$.  The final leg in the coupon in $B_0$ is connected to a leg in the coupon in $B_1$ (intersecting the disk $D_0$ once).   For $1\leq k \leq n$, the $k$th leg in the coupon in $B_2$ is connected to a leg in the coupon in $B_1$ in such a way that it intersects the boundary of $B_1$ exactly once in the disk $D_k$.  
 An example of a projection of $\Gamma$ in handlebody with 3 holes and 3 framed points is given in Figure \ref{F:Handlebody}.
  Next we show that $\Skein(H_g,\Su_g)$ is generated by 
  all
  $\cat$-colorings of $\Gamma$. 
 We will complete the proof by showing that actually only finitely many of these colorings generate $\Skein(H_g,\Su_g)$.

  Let $T$ be an admissible graph in $H_g$.  Up to entangling $T$ in $H_g$ by an isotopy, we can suppose that 
  each disc $D_k$ intersects at least one edge
  of $T$ which is colored by a projective module.  We now modify $T$
  in a neighborhood of $D_k$ for which we have chosen an orientation:
  we can assume that $T$ intersects $D_k$ transversally on colinear
  $\cat$-colored dots forming a sequence $\wbU=((U_1,\ve_1),\ldots,(U_n,\ve_n))$ of homogeneous modules
  $U_i\in\cat$ with signs $\ve_i\in\{\pm1\}$ given by the orientations
  of the edges.  As at least one of the $U_i$ is projective thus the
  tensor product $F(\wbU)$ is projective.  Hence this module splits
  as a direct sum of indecomposable projective modules ${\pro}_\alpha$.  So
  there exists maps $f_\alpha:F(\wbU)\to {\pro}_\alpha$ and
  $g_\alpha:{\pro}_\alpha\to F(\wbU)$ such that $\Id_{F(\wbU)}=\sum_\alpha g_\alpha\circ f_\alpha$.  This implies that a
  tubular neighborhood of the oriented disc $D_k$ which consists in a
  cylinder containing $n$-strands colored by the $U_i$ is skein
  equivalent to the sum of graphs~:
  $$\sum_\alpha\epsh{fig11}{12ex}\put(-18,15){\ms{g_\alpha}}
  \put(-18,-13){\ms{f_\alpha}}\put(-10,0){\ms{{\pro}_\alpha}}\,.$$ 
  We do this substitution for all discs $D_k$ and we obtain a skein
  equivalent element consisting of graphs $T'$ which meet each $D_k$
  on a unique edge $e_k$ colored by a projective indecomposable module
  ${\pro}_k$.  Now we can use Lemma \ref{lem:skeinhom} and replace the
  content of each 3-ball by a skein equivalent coupon.  Thus any
  element of $\Skein(H_g,\Su_g)$ is equal to a linear combination of
  colorings of $\Gamma$ where the edge $e_k$ intersecting $D_k$ is colored
  by an indecomposable projective module ${\pro}_k$
  and the coupons are colored by some morphisms living in the finite
  dimensional space $\Hom_\cat(\unit,{\pro}_1\otimes\cdots\otimes
  {\pro}_{g+1})$, $\Hom_\cat({\pro}_1\otimes\cdots\otimes
  {\pro}_{g+1},{\pro}_0)$ 
  and 
  $\Hom_\cat({\pro}_0,\bigotimes_iX_i)\simeq \Hom_\cat(\unit,{\pro}_0^*\otimes\bigotimes_iX_i)$
  where $X_i$ is the fixed color of $p_i$.  
   To finish the proof,
  we now show that one can restrict the coloring of the edges to a
  finite set of indecomposable projective modules.

  Let $\wb\alpha_k=\omega([\partial D_k])$ where $\omega$ is the
  cohomology class of $\Su_g$.  The compatibility condition for
  triples implies that the color of $e_k$ is an object of
  $\cat_{\wb\alpha_k}$.  Fix a finite set $\{P^j_k\}$ of
  representatives of the $\sigma$-orbits of indecomposable projective
  modules of $\cat_{\wb\alpha_k}$ (i.e. any indecomposable module of
  $\cat_{\wb\alpha_k}$ is isomorphic to $P^j_k\otimes\sigma^n$ for
  some $j$ and some $n\in\Z$).  For $k=1,\ldots,g$, we can now add to
  the graph a $\sigma$-colored curve made of two unknotted strands
  respectively from $D_k$ to $D_{k+1}$ in one 3-cell and from
  $D_{k+1}$ to $D_k$ in the other 3-cell.
The union of this curve with a coloring $T$ of $\Gamma$ is $\sigma$-equivalent to $T$ and is skein equivalent to a coloring $T'$ of $\Gamma$ where the projective  modules ${\pro}_k$ and ${\pro}_{k+1}$ of $T$ have been replaced by ${\pro}_k\otimes \sigma$ and ${\pro}_{k+1}\otimes \sigma^{-1}$.  
  Doing this, we see that
  any coloring of $\Gamma$ is proportional in $\Skein(H_g,\Su_g)$ to a
  coloring where ${\pro}_1,\ldots,{\pro}_g$ belong to the finite set
  $\{P^j_k\}$: we are left to show that then also ${\pro}_{0}$ and ${\pro}_{g+1}$ must range in a finite set.  Now pick such a coloring and assume the coupons are
  colored by non zero morphisms.  We use the fact that these morphisms
  are $H$-linear maps to restrict the range of ${\pro}_0$ and ${\pro}_{g+1}$ to
  a finite set.  Indeed, for all but finitely many $n\in\Z$, the
  weights of ${\pro}_0^j\otimes\sigma^n$ are disjoint from those of
$\bigotimes_iX_i$ and so  $\Hom_\cat({\pro}_0^j\otimes\sigma^n,\bigotimes_iX_i)=\{0\}$. 
 Similarly,
  for all but finitely many $n\in\Z$, zero is not a weight of
  ${\pro}_1\otimes\cdots\otimes {\pro}_{g+1}^j\otimes\sigma^n$ thus
  $\Hom_\cat(\unit,{\pro}_1\otimes\cdots\otimes {\pro}_{g}\otimes
  {\pro}_{g+1}^j\otimes \sigma^n)=\{0\}$. 
\end{proof}

For any $\alpha\in\C$, the subset $\alpha+2\Z$ of complex numbers is
naturally totally ordered by comparing the real parts.  Hence a
homogeneous module $V\in\cat_{\wb\alpha}$ has a lowest weight that we
denote by ${\lw}(V)\in\alpha+2\Z$ and a highest weight
${\hw}(V)\in\alpha+2\Z$.  If $\Su=(\Su,\{p_i\}_{i},\coh, {\La})$ is a
decorated surface where $p_i$ is decorated by a module $V_i\in\cat$, 
then we set ${\lw}(\Su):={\lw}(\bigotimes_iV_i)$ and
${\hw}(\Su):={\hw}(\bigotimes_iV_i)$.
\begin{lemma}\label{L:boundweight}
  Let $\Su_g$ be a connected decorated surface of genus $g$ and let
  $H_g$ be a handlebody with boundary $\Su_g$.  Then ${\lw}(\Su)$ and
  ${\hw}(\Su)$ are even integers.  Furthermore $\Skein(H_g,\Su_g)=0$
  unless
  $${\lw}(\Su)\le 4(g+2)(r-1)\et {\hw}(\Su)\ge -4(g+2)(r-1).$$
\end{lemma}
\begin{proof} 
  We use the notation of the proof of Proposition
  \ref{P:Skein-finite}.  First remark that $\partial D_0$ is a
  bounding curve in $\Su\setminus\{p_i\}$ thus $[\partial D_0]=0$ and
  $\wb \alpha_0=\wb0$.  This implies that the set of weights of
  $\bigotimes_iV_i$ is included in $2\Z$ so
  $\lw(\Su),\hw(\Su)\in2\Z$.  Now for any coloring of $\Gamma$,
  ${\pro}_0$ and ${\pro}_1\otimes\cdots\otimes {\pro}_{g+1}$ also belong to
  $\cat_{\wb 0}$. 
   Then,  by Proposition \ref{P:proj-mod_New}, we have that
  ${\hw}({\pro}_0)\le {\lw}({\pro}_0)+4(r-1)$ and similarly,
  ${\hw}({\pro}_1\otimes\cdots\otimes {\pro}_{g+1})\le
  {\lw}({\pro}_1\otimes\cdots\otimes {\pro}_{g+1})+4(g+1)(r-1)$.  Assuming that
  the coupons of the coloring are colored by non zero morphisms, then for each ball in the decomposition of $H_g$ as in the proof of Proposition \ref{P:Skein-finite} we have two inequalities:
  $$
  \begin{array}{ccc}
B_0: &   {\hw}(\bigotimes_iV_i)\ge {\lw}({\pro}_0)&{\lw}(\bigotimes_iV_i)\le {\hw}({\pro}_0)\\
B_1: &   {\hw}(\bigotimes_{k\ge1}{\pro}_i)\ge {\lw}({\pro}_0)&
  {\lw}(\bigotimes_{k\ge1}{\pro}_i)\le {\hw}({\pro}_0)\\
  B_2: &    {\hw}(\bigotimes_{k\ge1}{\pro}_i)\ge 0&{\lw}(\bigotimes_{k\ge1}{\pro}_i)\le 0
  \end{array}
  $$
  As a consequence, we have that ${\hw}(\bigotimes_iV_i)\ge-4(g+2)(r-1)$
  and ${\lw}(\bigotimes_iV_i)\le 4(g+2)(r-1)$. Indeed by the relations $B_2$ we have that $$0\leq {\hw}({\pro}_1\otimes\cdots\otimes {\pro}_{g+1})\le
  {\lw}({\pro}_1\otimes\cdots\otimes {\pro}_{g+1})+4(g+1)(r-1)$$ and so $ {\lw}({\pro}_1\otimes\cdots\otimes {\pro}_{g+1})\geq -4(g+1)(r-1)$. Hence  ${\hw}({\pro}_0)\geq -4(g+1)(r-1)$. But ${\hw}({\pro}_0)\leq {\lw}({\pro}_0)+4(r-1)$ so ${\lw}({\pro}_0) \geq -4(g+2)(r-1)$. One concludes using the relations coming from $B_1$. 
  (A similar argument shows the inequality for $\lw(\Su)$.)
\end{proof}
\begin{lemma}[Adding typical colors]\label{lem:addingtypicals} 
Let $\alpha$ be any element of $\Cp=(\C\setminus \Z)\cup r\Z$ then the
   skein module $\Skein(M,\Su)$ is
  generated by 
  the set of decorated cobordisms containing a ribbon graph with an 
  edge colored by $V_\alpha$ in each connected component of $M$.\end{lemma}
\begin{proof} 
Let $M_i$ be the connected components of $M$.  By definition $\Skein(M, \Su)$ is generated by cobordisms $(M,T,\coh)$ such that $T\cap M_i$ is admissible for each $i$, i.e. each component $M_i$ has an edge colored by a projective module $P_i$.   Let $T$ be such a graph.  
  Consider the epimorphisms
  $f_i:V_\alpha\otimes V_{-\alpha}\otimes P_i\to P_i $ given by the evaluation morphism 
  tensor the identity.  Since $P_i$ is projective the morphism $f_i$ has a
  left inverse $g_i$, i.e. $f_i\circ g_i=\Id_P$.  Thus, for each connected component $M_i$, we can replace a small
  portion of the edge colored with $P_i$ by two connected coupons
  colored with $g_i$ and $f_i$.  The obtained graph has an edge colored by $V_\alpha$ in each connected component and  is skein equivalent to $T$.  Thus, in $\Skein(M,\Su)$, the cobordisms $(M,T,\coh)$ is equal to a cobordism with a graph having an edge colored by $V_\alpha$ in each connected component.
\end{proof}

\subsection{1-surgery}\label{subsec:1surg}
Recall the definition of surgery in terms of cobordisms.  Let $M$ and
$M'$ be two 3-manifolds seen as cobordisms from $\emptyset$ to the
surface $\Su$.  Let $S^{p-1}\times B^{4-p}$ and $B^{p}\times
S^{3-p}$ be the two obvious manifolds bounding $S^{p-1}\times
S^{3-p}$.  We think of these manifolds as cobordisms from $\emptyset$
to the surface $S^{p-1}\times S^{3-p}$.
We say that $M'$ is obtained by a \emph{$p$-surgery} on $M$ if there exists a
cobordism $N:S^{p-1}\times S^{3-p}\to \Su$ such that
$$M=N\circ (S^{p-1}\times B^{4-p}:\emptyset\to S^{p-1}\times S^{3-p})
\text{ and } M'=N\circ (B^{p}\times S^{3-p}:\emptyset\to S^{p-1}\times
S^{3-p}).$$ 

In this subsection, we will show that if $M'$ is obtained from $M$ by a 1-surgery then any vector in $\V(\Sigma)$ represented by a
decorated cobordism in $M$ can be expressed as a decorated cobordism
in $M'$.

The following proposition describes the effect of a 1-surgery on a
vector of $\V(\Su)$.  For $\alpha\in\Cp$, let
$S_\alpha=(S^2,\{p_1,p_2\}, \coh, {\La})$ be the decorated surface
defined as follows.  Here $S^2\simeq\C\cup\{\oo\}$ is the sphere and
the points $p_1, p_2$ are both colored by $V_\alpha$ and oriented by
$+$ and $-$, respectively.  The Lagrangian ${\La}$ is trivial and
$\coh$ is the unique cohomology class relative to the base point which
is compatible with colorings of $\{p_1,p_2\}$.

\begin{prop}
  \label{P:1-surg}
  Let $u_\alpha\in\cV(S_\alpha)$ be the class of a ball $B^3$ containing an unknotted arc colored
  by $\alpha\in\Cp$
  with the unique compatible cohomology class.  
  Let $S_\alpha\times[0,1]\in\cV(S_\alpha\sqcup \wb{S_\alpha})$ be the cylinder  with any compatible cohomology class.  Then
  $$\eta\,\qd(\alpha)\,\V(S_\alpha\times[0,1])=\V(u_\alpha\sqcup \wb{u_\alpha}).$$
Visually this equality can be represented as:
  \[\eta\,\qd(\alpha)\,\V\bp{\epsh{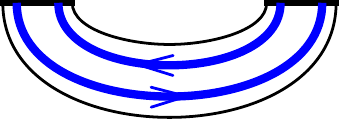}{6ex}}
  \put(-62,-20){$S^2\times B^1$\ } =\V\bp{\epsh{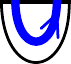}{3ex}\quad
    \epsh{fig25}{3ex}\ } \put(-45,-20){$B^3\times S^0$}\]
\end{prop}
\begin{proof}
  To prove the statement we glue $M'\in \cV'(S_\alpha\sqcup
  \wb{S_\alpha})\simeq\cV(S_\alpha\sqcup \wb{S_\alpha})$ (see Remark \ref{rem:admsurfprimenonprime}) to both $S_\alpha\times[0,1]$ and $u_\alpha\sqcup \wb{u_\alpha}$ then compute the invariant $\Zr$.
    As $M'$
  contains an edge colored by the projective module $V_\alpha$, it can
  be replaced as in Lemma \ref{lem:addingtypicals} by a skein
  equivalent element with an edge colored by an element $\beta\in\C\setminus\Z$.   

  Let us first suppose that the two spheres are contained in the same
  component of $M'$.  For simplicity, we may assume that $M'$ is
  connected.  Let $M_1=M'\circ (S_\alpha\times[0,1])$ and
  $M_2=M'\circ (u_\alpha\sqcup \wb{u_\alpha})$. 
Recall that a decorated surface has two kinds of points: a base point $*$ and the colored points $\{p_i\}$.   We will construct a surgery presentation of $M_2$ which can be extended to a surgery presentation of $M_1$.

To do this we need to find a ``nice'' 1-cycle in $M_1$ as follows.  
Choose a path $\gamma$ in $M'$ from the base point $*$ of $S_\alpha$ to the base point $*$ of $\wb{S_\alpha}$ which becomes close to the edge of $M'$ which is labeled by $\beta$.  Let $\gamma'$ be the trivial path $\{*\}\times[0,1]$ in $S_\alpha\times[0,1]$ from the base point $*$ of $\wb{S_\alpha}$ to the base point $*$ of $S_\alpha$.   Let $\coh_1$ be the $\C/2\Z$-valued cohomology class of the cobordism of $M_1$.  We can and will assume that $\coh_1(\gamma\circ \gamma')\notin \Z/2\Z$ where $\gamma\circ \gamma'$ is the concatenation of $\gamma$ and $\gamma'$ (otherwise, we can modify the choice of $\gamma$ by taking a connected sum of $\gamma$  with a meridian $m$ of the edge labeled by $\beta$; then $ \coh_1((\gamma\#m)\circ \gamma')=\beta+\coh_1(\gamma\circ \gamma')\notin \Z/2\Z$).

  Let $L$ be a computable surgery presentation of $M_2$.  Let $B$ be a
  3-ball in $S^3$ corresponding to the union of a neighborhood of
  $\gamma$ and the two 3-balls added to $M'$.  Then a surgery
  presentation of $M_1$ is obtained by
  adding a new component to $L$ as follows.  Replace the ball $B$  
  by a ball with a $0$-framed trivial knot on whose meridian the value of $\coh$ is  $\coh_1(\gamma\circ \gamma')$ and whose Seifert disc intersects the two parallel strands colored respectively by $\alpha,-\alpha$.
  Then Lemma \ref{lem:old71} with $n=0$ implies that
  $$\Zr(M_1)=\qd(\alpha)^{-1}r^3\lambda\Zr(M_2)=\qd(\alpha)^{-1}\eta^{-1}\Zr(M_2)$$
  and the equality follows.

  Next consider the second case: the two spheres are contained in
  two different components $M'_1$ and $M'_2$ of $M'$.  Again for
  simplicity, assume that $M'=M'_1\sqcup M'_2$.
Let $N_1=M'_1\circ u_\alpha$ and $N_2= M'_2\circ  \wb{u_\alpha}$.  Then the cobordism 
$M'\circ (S_\alpha\times[0,1])$ is the  banded connected sum $N_1\#_{e_1,e_2}N_2$ of $N_1$ and $N_2$ along  edges $e_i$ colored by $\alpha$.
      Then  \cite[Proposition 3.11]{CGP} implies  $\Zr(N_1\#_{e_1,e_2}N_2)=\qd(\alpha)^{-1}\eta^{-1}\Zr(N_1\sqcup  N_2)$ where the factor $\eta$ comes from  $b_0(N_1\#_{e_1,e_2}N_2)=b_0(N_1\sqcup N_2)-1$, see Equation  \eqref{eq:Zdefi}.
\end{proof}

\subsection{2-surgery}\label{subsec:2surg}
 Here we consider the effect of 2-surgery along a knot on a vector of $\V(\Su)$.  We use this to show that if $M$ and $M'$ are two manifolds   which are related by a 2-surgery with $\partial M=\partial M'$ then their skein modules have the same image through $\V$, see Proposition~\ref{P:EqualSkien}.  

We consider the following decorated surface and cobordisms.  
  Let $\wb\alpha\in\C/2\Z\setminus\Z/2\Z$ and $\omega$ be the cohomology class on $S^1\times S^1$ defined by $\omega([\{1\}\times S^1])=0$ and $\omega([S^1\times\{1\}])=\wb\alpha$ (here $[\gamma]\in H_1(S^1\times S^1;\Z)$ is the class of a curve $\gamma$).  Let $\La$ be the Lagrangian of $H_1(S^1\times S^1,\R)$  
generated by the homology class $[S^1\times \{1\}]$ or equivalently the kernel of the map $H_1(S^1\times S^1)\to H_1(B^2\times S^1)$ and $\Su$ be the decorated surface $\Su=(S^1\times
  S^1,\emptyset,\omega,\La)$.  Let  $(\wb{S^1\times B^2},\emptyset,\Id,\omega_1,n)$ (the $\overline{\cdot}$ is due to our outward vector first convention applied to $B^2$ and to $S^1\times B^2$)
  and $(B^2\times S^1,K_{\wb   \alpha},\Id,\omega_2,n)$ be the two decorated cobordisms from $\emptyset$ to $\Su$ defined as follows. The ribbon graph $K_{\wb\alpha}$ is the zero framed knot $[0,\frac12]\times S^1$ in $B^2\times S^1$  colored with a Kirby color of degree $\wb\alpha$; $\omega_1$  and $\omega_2$ are the unique cohomology classes compatible with $\omega$.  Finally, $n$ is any integer.  

\begin{prop}
\label{P:2-surg}  With the notation of the previous paragraph we have
  $$\V(\wb{S^1\times B^2},\emptyset,\Id,\omega_1,n)=
  \lambda\V(B^2\times S^1,K_{\wb
    \alpha},\Id,\omega_2,n)\in\V(\Su)$$
  \[\V\bp{\epsh{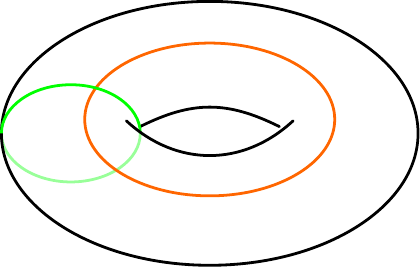}{12ex}\put(-60,-17){$\circlearrowright$}}
  \put(-75,-35){$-{\color{Orange}S^1}\times{\color{green}B^2}$} =
  \lambda\,\V\bp{\epsh{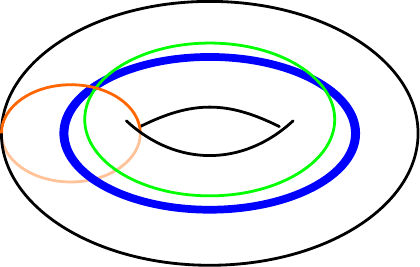}{12ex}
    \put(-19,12){\ensuremath{\color{blue}\Omega_{\wb\alpha}}}
    \put(-60,-17){$\circlearrowleft$}}
  \put(-75,-35){${\color{Orange}B^2}\times{\color{green}S^1}\supset
    {\color{blue}K_{\wb \alpha}}$}.
  \]
\end{prop}
\begin{proof}
Let $N=(N,T,f,\omega_N,n_N)$  be a connected decorated cobordism from $\Su$ to $
  \emptyset$.  Let $M=N\circ
  (B^2\times S^1,K_{\wb \alpha},\Id,\omega_2,n)$ and $M'= N\circ
  (\wb{S^1\times B^2},\emptyset,\Id,\omega_1,n)$.  We need to show that
  $\Zr(M')=\lambda\Zr(M)$.  

  Let 
  $L\subset S^3$ be a computable surgery
  presentation of $(M,T\cup K_{\wb \alpha})$.  Then $L\cup
  K_{\wb \alpha}$ is a computable surgery presentation for
  $(M', T)$.  Let $W_L$ be the oriented 4-manifold bounding $M$ obtained by
  gluing $2$-handles $B^2\times B^2$ along a tubular neighborhood of $L$
  in $S^3=\partial B^4$.  Then $W_{L\cup K_{\wb \alpha}}=(\wb{B^2\times B^2})\cup_{B^2\times K_{\wb
      \alpha}}W_L$ is a 4-manifold bounding $M'$.  Conventions on orientations imply that 
  $$\partial (B^2\times B^2)=S^1\times B^2\ \cup\ B^2\times S^1,\quad 
  \partial (B^2\times S^1)=S^1\times S^1=\wb{\partial (S^1\times B^2)}.$$
 Consider the following two Lagrangians of $H_1(\Su,\R)$:
  \begin{itemize}
  \item $\La_-=\ker\bp{H_1(\Su,\R)\to H_1(S^1\times B^2,\R)}=\R.[\{1\}\times S^1]$
  \item  $\La_+=\ker\bp{H_1(\Su,\R)\to H_1(N,\R)}$.
  \end{itemize}
  By Wall's theorem on signatures of $4$-manifolds,
  one has
  $$\sigma(W_{L\cup K_{\wb \alpha}})=\sigma(\wb{B^2\times
  B^2})+\sigma(W_L)-\mu(\La_-,\La,\La_+)$$ where
  $\sigma(W_L)=\sigma(L)$ is the signature of the 4-manifold (which is
  given by the signature of the linking matrix of $L$).     As $\sigma(B^2\times B^2)=0$, we have $\sigma(L\cup K_{\wb
    \alpha})=\sigma(L)-\mu(\La_-,\La,\La_+)$.  By definition of the
  composition of decorated cobordisms, the signature-weight of $M$ is given
  by $$n_M=n_N+n-\mu(\La,\La,\La_+)=n_N+n$$ while the signature-weight of $M'$
  is $$n_{M'}=n_N+n-\mu(\La_-,\La,\La_+)=n_M-\mu(\La_-,\La,\La_+).$$
  By definition $\Zr(M)$
  and $\Zr(M')$ are both the product of $F'(L\cup K_{\wb \alpha} \cup T)$ with
  a normalizing term (see Equations \eqref{eq:Zrweighted} and
  \eqref{eq:Zrconnected}).  Thus one has
  $\Zr(M')=\dfrac{\eta\lambda^{m+1}\dep^{n_{M'}-\sigma(L\cup K_{\wb
        \alpha})}} {\eta\lambda^{m}\dep^{n_M-\sigma(L)}}\Zr(M)
  =\lambda\Zr(M)$.
\end{proof}

\begin{lemma}
  \label{L:trick}
 Let $\wb\alpha \in \C/2\Z \setminus \Z/2\Z$.   
  Consider two parallel copies $K^\pm$ of the framed knot
  $K=[0,\frac12]\times S^1\subset B^2\times S^1$ identical except
  $K^+$ has framing one more and $K^-$ has framing one less.  We color
  $K^+$ with a Kirby color of degree $-\wb\alpha$ and $K^-$ with a
  Kirby color of degree $\wb\alpha$.  Then
  $$\V(B^2\times S^1,\emptyset,\Id,\omega_1,n)=\lambda^{2}
  \V(B^2\times S^1,K^+\cup K^-,\Id,\omega_2,n)\in\V(S^1\times
  S^1,\emptyset,\omega,\La)$$ where $\omega$ is the cohomology class
  with value $0$ on the meridian $S^1\times\{1\}$ and $\wb \alpha$ on
  the longitude $\{1\}\times S^1$, $\omega_1$ and $\omega_2$ are the unique
  compatible cohomology classes, $n$ is any integer and $\La$ any
  Lagrangian.
\[\V\bp{\epsh{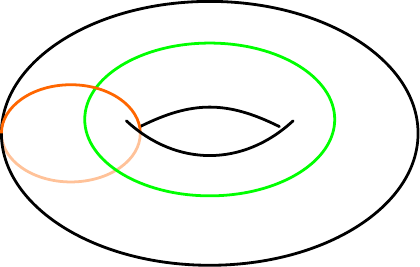}{12ex}}
\put(-75,-35){${\color{Orange}B^2}\times{\color{green}S^1}$}
=\lambda^2\,\V\bp{\epsh{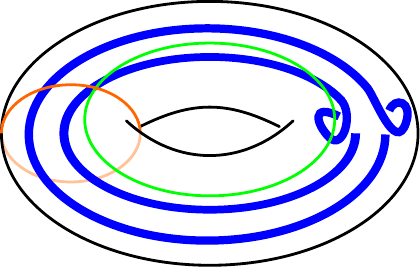}{12ex}
  \put(-34,7){\ms{\color{blue}\Omega_{-\!\wb\alpha}}}\put(-7,10){\ms{\color{blue}\Omega_{\wb\alpha}}}}
\put(-75,-35){${\color{Orange}B^2}\times{\color{green}S^1}\supset
  {\color{blue}K_{+}\cup K_-}$}
\]
\end{lemma}
\begin{proof}
  First $(B^2\times S^1)\setminus(K^+\cup K^-)$ is the cartesian
  product of $S^1$ with a disc minus two points.  Hence by K\"unneth
  formula $H_1\bp{(B^2\times S^1)\setminus(K^+\cup K^-)}=\Z
  m_+\oplus\Z m_-\oplus \Z \ell$ where $\ell=[K]$ is the longitude and
  $m_\pm$ is the class of a meridian of $K_\pm$.  For the meridian
  $m=[S^1\times\{1\}]$ of the torus we have $m=m_++m_-$.  Clearly it holds 
  $-\omega_2(m_+)=\omega_2(m_-)=\omega_2(\ell)=\wb\alpha$ as $\coh_2$ is compatible.  
  The homology class of the
  parallel to $K_+$ is $\ell_+=\ell+m_+$ while the homology class of
  the parallel to $K_-$ is $\ell_-=\ell-m_-$.  Hence
  $\omega_2(\ell_+)=\omega_2(\ell_-)=0$ and we can apply twice
  Proposition \ref{P:2-surg} to both the tubular neighborhood of $K_+$
  and $K_-$.  But surgery on $K_+$ cancels the surgery on $K_-$ thus
  yielding the same 3-manifold $B^2\times S^1$.
\end{proof}

\begin{prop}\label{P:EqualSkien}
  Let $M$ and $ M'$ be two manifolds related by a 2-surgery, bounding
  the same decorated surface $\Su$.  If $(M,T,f,\coh,n)\in
  \Skein(M,\Su)$ then there exists a $(M',T',f,\coh',n)\in
  \Skein(M',\Su)$ such that
  $$\V(M,T,f,\coh,n)=\V(M',T',f,\coh',n)\in\V(\Su)$$
where $\V$ is the map given in Equation \ref{E:DefV}.  It follows $\V(\Skein(M,\Su))=\V(\Skein(M',\Su))$. 
\end{prop}
\begin{proof}
  In the following, we omit the parametrization of the boundaries
  $f:\Su\to\partial M=\partial M'$ and the signature-weight that are
  constant.
Suppose that $M'$ is obtained from $M$ by doing a 2-surgery along a knot $K$.  Let $(M,T,\coh)\in \Skein(M,\Su)$.  If  $\omega([K])\notin\Z$ then we can apply Proposition \ref{P:2-surg} to obtain a  cobordisms $(M',T',\coh')\in \Skein(M',\Su)$ such that $\V(M,T,\coh)=\V(M',T',\coh')$.  If $\omega([K])\in\Z$ then by the proof of Lemma \ref{lem:addingtypicals}  we can assume $T$ has an edge colored with $\alpha\in\C\setminus\Z$.  Let $K'$ be a knot in $M$ obtained by taking the connected sum of $K$ with a meridian of this edge of $T$.  Doing a 2-surgery along $K'$ we obtain $M'$ and a cobordisms $(M',T',\coh')\in \Skein(M',\Su)$.  Then Proposition \ref{P:2-surg} implies that $\V(M,T,\coh)=\V(M',T',\coh')$.
Finally, if $(M',T',\coh')$ is any cobordism in $ \Skein(M',\Su)$, then the above argument implies that there exists a cobordism $(M,T,\coh)\in \Skein(M,\Su)$ such that $\V(M,T,\coh)=\V(M',T',\coh')$.  Thus, the last statement of the proposition follows.  
\end{proof}

\subsection{Finite dimensionality of $\V$ for connected surfaces}\label{sub:findimconnected}
 In this subsection, we use 1~and~2-surgeries to show that the map $\V:\Skein(M,\Su)\to\V(\Su)$ is surjective for connected $M$ and consequently the vector space $\V(\Su)$ is finite dimensional.

\begin{prop}\label{P:trick} For any $\alpha\in\Cp$, 
  the image of the map $\cV(M,\Su)\to\V(\Su)$ is generated by the
  image under $\V$ of the set of decorated cobordisms containing a graph with one
  edge colored by $V_\alpha$ in each connected component of $M$.
\end{prop}

\begin{proof}
Let $M_T=(M,T,f,\omega,n):\emptyset\to\Su$ be a cobordism in $\cV(M,\Su)$.  We have the following two cases.  \\
Case 1.  Suppose $T$ has an edge colored by a projective object $V$.  Then $M_T\in\Skein(M,\Su)$ and the proposition follows from Lemma \ref{lem:addingtypicals}.\\
Case 2.  If  $T$ does not have an edge colored by a projective object, then the
  admissibility condition on 
  the cobordism $M_T$ 
  implies that 
 that there exists a curve $\gamma$ such that  $\omega(\gamma)\in (\C\setminus \Z)/2\Z$.  
  So we can apply Lemma \ref{L:trick}
  on a tubular neighborhood of this curve to make appear two knots colored by
  Kirby colors of degree $\pm\omega([\gamma])$ without changing the underlying
  manifold $M$ nor the value of $\V(M_T)$ (up to the factor $\lambda^2$).
  Then we can apply Case 1 to this new decorated manifold.
\end{proof}

\begin{theo}\label{teo:skeinsurjects}
  Let $\Su$ be a decorated surface, then for any connected manifold
  $M$ bounding $\Su$, the map $\V: \Skein(M,\Su)\to\V(\Su)$, given in Equation \eqref{E:DefV}, is surjective.
  As a consequence, if $\Su$ is connected $\dim_\C(\V(\Su))<\oo$.
\end{theo}
\begin{proof} 
Let $M$ be a connected manifold bounding $\Su$.  Let 
$N\in\cV(\Su)$ be a cobordism where the underlying manifold has $c$ connected components.  Fix $\alpha\in\Cp$, then by Proposition \ref{P:trick}, there exists a cobordism 
$N'\in \Skein(N,\Su)$ such that $\V(N)=\V(N')$ where the graph of $N'$ 
has an  edge colored by $\alpha$ in each connected component.  Now we apply  the 1-surgery of Proposition~\ref{P:1-surg} $c-1$-times to get a cobordism 
 $N''$ whose the underlying manifold of is connected, whose graph is still admissible  and such that $\V(N)=(\eta\qd(\alpha))^{c-1}\V(N'')\in \cV(\Su)$.   
Now the underlying manifold of $N''$ is related to the manifold $M$ by a finite sequence of 2-surgeries.  Thus, Proposition \ref{P:EqualSkien} implies there exists a cobordism $N'''\in \Skein(M,\Su)$ such that $\V(N'')=\V(N''')$.  Thus, $N'''$ maps onto $\V(N)$. 
  Finally, the last sentence of the  theorem follows from Proposition \ref{P:Skein-finite} since $\dim_\C(\Skein(M,\Su))<\oo$ when $M$ is a handlebody bounding $\Su$.
\end{proof}

\begin{cor}
Let $\Su$ be a decorated connected surface, then $\V'(\Su)$ is
  finite dimensional.
\end{cor}
\begin{proof}
  The pairing $\V'(\Su)\otimes \V(\Su)\to\C$ is
  non-degenerate so $\dim_\C(\V'(\Su))=\dim_\C(\V(\Su))$.
\end{proof}

\subsection{The spheres $\hS_k$}\label{SS:SphereS_k}
From Theorem \ref{teo:skeinsurjects}  we know $\V(\Su)$ is finite dimensional if $\Su$ is connected.  In Subsection \ref{SS:FiniteDimVdiscon} we will extend this result to surfaces $\Su$ which are disconnected.  In this subsection we consider some special decorated spheres which will be used in Subsection  \ref{SS:FiniteDimVdiscon}.  
 
 Let $V_0$ be the simple projective highest module of weight $r-1$ and let $\sigma^k=\C^H_{2kr'}$, see Subsection \ref{SS:Quantsl2}.  Recall the definition of the decorated sphere $S^2_{\wbU}$ given in Subsection \ref{SS:Skein}. For $k\in\Z$, let $\hS_k=S^2_{\wbU}$ where 
$$\wbU =
\begin{cases}
((V_{0},+1),(\sigma^k,+1),(V_0,-1)), & \text{if }k\neq 0 \\
((V_{0},+1),(V_0,-1)), & \text{if }k=0
\end{cases}.
$$
 Let $f$
be an orientation reversing diffeomorphism of $S^2$ exchanging the points labeled with $V_0$ and, if $k\neq 0$, fixing the point labeled with $\sigma^k$.  
\begin{prop} \label{prop:easyspheres}
  \begin{enumerate}
  \item Let $\wbU=((V_0,+),(\sigma^k,+))$, then $\V(S^2_{\wbU})=\Skein(B^3,S^2_{\wbU})=0$.
  \item $\V(S^2_{\emptyset})=\Skein(B^3,S^2_{\emptyset})=\C$ is
    generated by an unknot in $B^3$ 
    colored by $V_\alpha$ for any $\alpha \in \Cp$.
  \item If $k\neq0$ then $\V(\hS_k)=\Skein(B^3,\hS_k)=\{0\}$.
  \item $\V(\hS_0)=\Skein(B^3,\hS_0)=\C $ is generated by $u_0$: an unknotted $V_0$-colored arc in $B^3$ connecting the two $V_0$-colored points in $\hS_0$.
  \end{enumerate}
\end{prop}
\begin{proof}
By Theorem \ref{teo:skeinsurjects}, we have $\Skein(B^3,S^2_{\wbU})$ surjects onto $\V(S^2_{\wbU})$.  We will use this fact to prove (1)--(4).  In particular, 
for all $ k\neq 0$ we have 
$$\Hom(\unit,V_0\otimes \sigma^k)=0 \text{ and } \Hom(\unit,V_0\otimes V_0^*\otimes \sigma^k)=0.$$ 
Then Lemma \ref{lem:skeinhom} and Theorem \ref{teo:skeinsurjects} implies  
(1) and (3).

 To prove (2)  observe that by Lemma \ref{lem:addingtypicals} any element of $\Skein(B^3,S^2_{\emptyset})$ is represented by a non-empty $\cat$-colored ribbon graph with at least one edge colored by a projective module $V_\alpha$.  
Now, the graph outside a small sub-arc of this edge is skein equivalent to an endomorphism of $V_\alpha$ and so to a scalar. 
 So $\Skein(B^3,S^2_{\emptyset})$ is generated by unknots colored by $V_\alpha$.
We are left to prove that the map $\Skein(B^3,S^2_{\emptyset})\to \V(S^2_{\emptyset})$ is non-zero.  This is seen by pairing a ball containing a $V_\alpha$-colored unknot with an empty ball bounding $S^2$ the result is $S^3$ with a $V_\alpha$-colored unknot whose invariant is $\qd(\alpha)\neq 0$.

Similarly to prove $(4)$, Lemma \ref{lem:skeinhom} implies that $\Skein(B^3,\hS_0)=\C$ as $\Hom(\unit, V_0\otimes V_0^*)=\Hom(V_0,V_0)=\C$. Then the map $\Skein(B^3,\hS_0)\to \V(\hS_0)$ is non-zero as coupling the vector represented by one strand in $B^3$ colored by $V_0$ with a similar vector in $\V'(\hS_0)$ gives $S^3$ containing an unknot colored by $V_0$ whose invariant is non-zero. 
\end{proof}

The surface obtained by removing two little discs from a disc is usually called a pant.  We consider its 3-dimensional analog:
\begin{defi}[3-dimensional pants] Let
  $\Pa_{k,\ell}^{k+\ell}:\hS_k\sqcup\hS_\ell\to \hS_{k+\ell}$ be the
  cobordism described as follows.  The underlying manifold of
  $\Pa_{k,\ell}^{k+\ell}$ is the closed Euclidean ball in 
  $\R^3$ of radius $4$ centered at $(0,0,0)$ minus two open balls
  of radius $1$ centered at $(\pm2,0,0)$.  The graph inside $\Pa_{k,\ell}^{k+\ell}$
  is contained in the plane $y=0$ of $\R^3$ and is drawn in Figure \ref{F:pants}.    The coupon in the graph is colored with the identity morphism $\sigma^k~\otimes ~\sigma^\ell\to \sigma^{k+\ell}$.   
  \begin{figure}[!]  
    \centering
    $\epsh{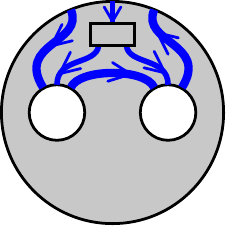}{16ex}\qquad\qquad\qquad
    \epsh{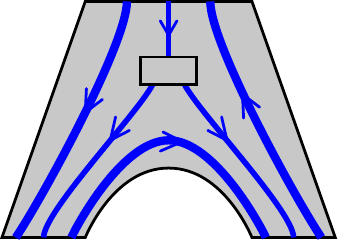}{14ex}
    \put(-56,35){\ms{V_0\sigma^{k+\ell}V_0^*}}
    \put(-87,-35){\ms{V_0\sigma^kV_0^*}}
    \put(-20,-35){\ms{V_0\sigma^\ell V_0^*}}$
    \caption{$\Pa_{k,\ell}^{k+\ell}\cap\{y=0\}$ and its schematic
      representation. }
    \label{F:pants}
  \end{figure}
  The decorated cobordism $\Pa_{k,\ell}^{k+\ell}$ is equipped with the
  obvious parametrization of its boundary  and has signature-weight $0$.  The cohomology class of this cobordism is defined as follows.   Let us assume that the basis point $*$ of a sphere $\hS_i$ is at the south pole.  If $p$ is any path in the southern hemisphere of  $\Pa_{k,\ell}^{k+\ell}$ between basis points of the boundary spheres in  $\Pa_{k,\ell}^{k+\ell}$  then set $\coh(p)=0$.  This assignment extends to a  unique compatible cohomology class $\coh$ for $\Pa_{k,\ell}^{k+\ell}$ (which is zero if $r$ is odd).  

 Similarly, let $\Pa^{k,\ell}_{k+\ell}:\hS_{k+\ell}\to\hS_k\sqcup\hS_\ell$ be the decorated cobordism described as follows.  Consider reflecting $\Pa_{k,\ell}^{k+\ell}$ through the $z=0$ plane.   Reversing the arrows of the graph inside the obtained manifold and replacing the coupon with  the identity morphism $\sigma^{k+\ell}\to \sigma^k~\otimes ~\sigma^\ell$ we have  $\Pa^{k,\ell}_{k+\ell}$.   As in the case of $\Pa_{k,\ell}^{k+\ell}$, $\Pa^{k,\ell}_{k+\ell}$ has a unique compatible cohomology class and signature-weight $0$.  
\end{defi}
\begin{prop}\label{P:pants}Let $k,\ell\in\Z$, then 
  $$\Pa^{k,\ell}_{k+\ell}\circ \Pa_{k,\ell}^{k+\ell}\equiv\frac{(-1)^{r-1}}{\eta}\Id_{\hS_k\sqcup\hS_\ell}$$
   where the equivalence $\equiv$ is defined in Definition \ref{D:EquivCob}.
\end{prop}
\begin{proof} By Lemma \ref{L:cob-as-vect}, it is enough to show that
  $[\Pa^{k,\ell}_{k+\ell}\circ
  \Pa_{k,\ell}^{k+\ell}]=\frac{(-1)^{r-1}}{\eta}[\Id_{\hS_k}\sqcup\Id_{\hS_\ell}]$
  as vectors of
$V(\hS_k\sqcup\hS_\ell\sqcup\wb{\hS_{\ell}}\sqcup\wb{\hS_{k}})$.  
Here is a
  graphical proof:
  $$\bc{\Pa^{k,\ell}_{k+\ell}\circ
  \Pa_{k,\ell}^{k+\ell}}=\bc{\epsh{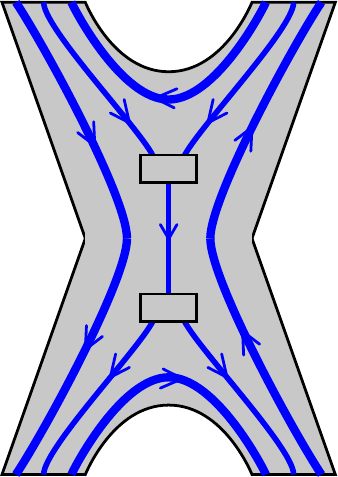}{12ex}}
  =\bc{\epsh{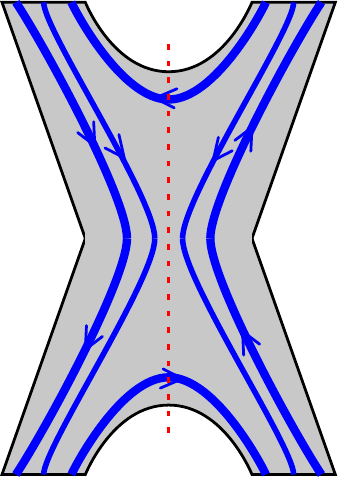}{12ex}}=(\eta\qd(0))^{-1}\bc{\epsh{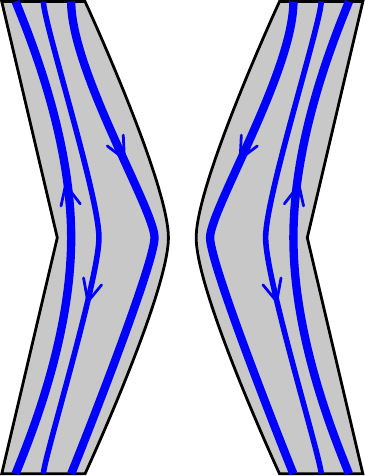}{12ex}}.$$
  The second equality is a skein equivalence and the last equality is an
  application of Proposition~\ref{P:1-surg} with $\alpha=0$ in a
  neighborhood of the sphere in dashed red on the diagram.
\end{proof}
\begin{prop}\label{prop:vspheres}
  Let $k,\ell\in\Z$ and $\beta\in\C/2\Z=H^0(\hS_k,\C/2\Z)$.  Recall the definition of the  cylinder  $\Id_{\hS_k}^\beta$ given in Subsection \ref{R:CylDef}.    Let $C_k^\beta$ be 
 $\Id_{\hS_k}^\beta$ seen as an element of
 $\cV(\hS_k\sqcup\wb{\hS_{k}})$.
  \begin{enumerate}
\item  \label{Ei1:CylProp} If $k\neq \ell$, then $\V(\hS_k\sqcup\wb{\hS_{\ell}})=\{0\}$.
\item \label{Ei2:CylProp}$\V(\hS_k\sqcup\wb{\hS_{k}})=\C$ is generated by $[C_k^0]$.
 \item \label{Ei3:CylProp}$[C_k^\beta]=q^{-2r'k\beta}[C_k^0]\et\Id_{\hS_k}^\beta\equiv q^{-2r'k\beta}\Id_{\hS_k}.$
 \item \label{Ei4:CylProp} If $\Su$ is any decorated surface then the vector spaces $\V(\Su \sqcup \wb{\hS_k})$ and $ \V(\Su \sqcup \hS_{-k})$ are isomorphic.   Thus, $\C=\V(\hS_k\sqcup\hS_{-k})\ncong \V(\hS_k)\otimes\V(\hS_{-k})=\{0\}.$
  \end{enumerate}
\end{prop}

\begin{proof}
  We start by proving the first statement of \eqref{Ei4:CylProp}.
  Consider the decorated cobordism $c_k: \wb{\hS_k}\to \hS_{-k}$
  defined as follows.  The underlying manifold of $c_k$ is $S^2\times
  [0,1]$.  If $k=0$ then $c_0$ is the mapping cylinder of an
  orientation reversing diffeomorphism of $S^2$ (say a reflexion of
  $\C\subset S^2$) which is an orientation preserving diffeomorphism
  $\wb{\hS_0}\to \hS_{0}$.  If $k\neq 0$ then $c_k$ is similar to
  $c_0$ but the graph in $c_k$ has a third additional
  parallel strand between the two $V_0$ colored strands.  This new
  strand has a coupon filled with an isomorphism
  $(\sigma^k)^*\stackrel\sim\to\sigma^{-k}$.  We pull back the
  cohomology class of $\hS_{-k}$ through the projection of the
  cylinder on $\hS_{-k}$.  Finally, this cobordism has
  signature-weight $0$. Using the inverse isomorphism
  $\sigma^{-k}\stackrel\sim\to(\sigma^k)^*$ we build similarly a
  cobordism $c'_k: \hS_{-k}\to\wb{\hS_k}$ such that $c_k\circ
  c'_k\equiv\Id_{\hS_{-k}}$ and $c'_k\circ
  c_k\equiv\Id_{\wb{\hS_{k}}}$.
  Thus, if $\Su$ is
  any decorated cobordism, the decorated cobordism $\Id_\Su\sqcup
  c_k$ induces an isomorphism $\V(\Su \sqcup \wb{\hS_k})\to \V(\Su
  \sqcup \hS_{-k})$.  

To  prove \eqref{Ei1:CylProp} and \eqref{Ei2:CylProp} we use the fact that  Proposition \ref{P:pants} implies that the map
  $\V(\Pa_{k,-\ell}^{k-\ell}):\V(\hS_k\sqcup\hS_{-\ell})\to \V(\hS_{k-\ell})$ is
  injective for all $k,l\in \Z$.  If $k\neq \ell$ then Proposition \ref{prop:easyspheres} implies $\V(\hS_{k-\ell})=0$ and $\V(\hS_k\sqcup\hS_{-\ell})=0$.  Now the isomorphism at the end of the previous paragraph implies 
  $\V(\hS_k\sqcup\wb{\hS_{\ell}})\cong\V(\hS_k\sqcup\hS_{-\ell})$ which is zero.  Thus, we have proved  \eqref{Ei1:CylProp}.  
  
  For \eqref{Ei2:CylProp} we have $\dim(\V(\hS_{0}))=1$ and so
  $\dim(\V(\hS_k\sqcup\hS_{-k}))\leq 1$.  To see that equality holds,
  notice that $\Pa_{k,-k}^{0}\circ (\Id_{\hS_k}\sqcup c_k) \circ
  C_k^\beta$ is $\sigma$-equivalent to
  $\bp{(-1)^{r-1}q^{-2r'\beta}}^k\Pa_{0,0}^{0}\circ
  (\Id_{\hS_0}\sqcup c_0) \circ C_0^\beta$. 
  Observe furthermore that applying twice
  Proposition \ref{P:1-surg} to both $C_0^0$ and $C_0^\beta$, gives 
  $[ C^0_0]=(\eta\qd(0))^{-1}[u_0 \sqcup \wb{u_0}]=[ C_0^\beta]$.
  Then we have
  \begin{align}\label{E:PantsAndCyl}
    [\Pa_{k,-k}^{0}\circ (\Id_{\hS_k}\sqcup  c_k)\circ C_k^\beta]
  &=\bp{(-1)^{r-1}q^{-2r'\beta}}^k[\Pa_{0,0}^{0}\circ (\Id_{\hS_0}\sqcup c_0) \circ C_0^\beta]\notag\\
  & =\bp{(-1)^{r-1}q^{-2r'\beta}}^k[\Pa_{0,0}^{0}\circ (\Id_{\hS_0}\sqcup  c_0) \circ C_0^0]\notag\\
  &=q^{-2r'k\beta}[\Pa_{k,-k}^{0}\circ (\Id_{\hS_k}\sqcup  c_k) \circ C_k^0]
  \end{align}
  where in the last equality we applied again a $\sigma$-equivalence. 
 Now 
  Proposition \ref{P:1-surg}, implies 
  $[(\Id_{\hS_0}\sqcup c_0) \circ C^0_0]=(\eta\qd(0))^{-1}[u_0 \sqcup u_0]$ and so    
  $$(\eta\qd(0))[\Pa_{0,0}^{0}\circ (\Id_{\hS_0}\sqcup  c_0) \circ C_0^0]=[\Pa_{0,0}^{0}\circ
  (u_0 \sqcup u_0)]=[u_0]\neq0.$$
  Thus, Equation \eqref{E:PantsAndCyl} is an 
  equality of non-zero numbers.  So $\V(\hS_k\sqcup\wb{\hS_{k}})$
  is one dimensional and generated by $C_k^0$.  Furthermore, Equation
  \eqref{E:PantsAndCyl} implies $[C_k^\beta]=q^{-2r'k\beta}[C_k^0]$.
  Then \eqref{Ei3:CylProp} follows from applying Lemma
  \ref{L:cob-as-vect} to $C_k^\beta-q^{-2r'k\beta}C_k^0$.
  
  Finally, the second statement of \eqref{Ei4:CylProp} follows from combining \eqref{Ei2:CylProp}, Proposition \ref{prop:easyspheres}  and the first statement of \eqref{Ei4:CylProp}. 
\end{proof}
\begin{cor}\label{C:pants2} Let $k,\ell\in\Z$, then
    $$\Pa_{k,\ell}^{k+\ell}\circ\Pa^{k,\ell}_{k+\ell}\equiv
    \frac{(-1)^{r-1}}{\eta}\Id_{\hS_{k+\ell}}.$$
\end{cor}
\begin{proof}
  By Lemma \ref{L:cob-as-vect}, it is enough to show that
  $\bc{\Pa_{k,\ell}^{k+\ell}\circ\Pa^{k,\ell}_{k+\ell}}
  =\frac{(-1)^{r-1}}{\eta}\bc{\Id_{\V(\hS_{k+\ell})}}$ as vectors of
  $\V(\hS_{k+\ell}\sqcup\wb{\hS_{k+\ell}})$.  But Proposition \ref{prop:vspheres} implies $\dim(\V(\hS_{k+\ell}\sqcup\wb{\hS_{k+\ell}}))=1$.   So we have 
  $$\bc{\Pa_{k,\ell}^{k+\ell}\circ\Pa^{k,\ell}_{k+\ell}}=c\bc{\Id_{\V(\hS_{k+\ell})}}$$ 
  for some constant $c$.  Pre-composing 
  this equality with $\V(\Pa_{k,\ell}^{k+\ell})$ and post-composing it with
  $\V(\Pa^{k,\ell}_{k+\ell})$ we have $ \bc{\Pa^{k,\ell}_{k+\ell}\circ\Pa_{k,\ell}^{k+\ell}\circ\Pa^{k,\ell}_{k+\ell}\circ\Pa_{k,\ell}^{k+\ell}}=c\bc{ \Pa^{k,\ell}_{k+\ell}\circ \Pa_{k,\ell}^{k+\ell}}$. Then Proposition~\ref{P:pants} implies $c=\frac{(-1)^{r-1}}{\eta}$.\end{proof}
\begin{exo} \label{exo:pants}  This exercise is a generalization of Proposition~\ref{prop:vspheres}.  For $k,\ell,m\in\Z$, let  $\Pa_{k,\ell,m}^{k+\ell+m}=
    \Pa_{k+\ell,m}^{k+\ell+m}\circ(\Pa_{k,\ell}^{k+\ell}\otimes\Id_{\hS_m})$.   Recursively define
    $$\Pa_{k_1,\ldots,k_n}^{k_1+\cdots+k_n}:\bigsqcup_i\hS_{k_i}\to\hS_{\sum k_i}.$$ 
    Similarly define their analogs
    $\Pa^{k_1,\ldots,k_n}_{k_1+\cdots+k_n}:\hS_{\sum
      k_i}\to\bigsqcup_i\hS_{k_i}$.
  \begin{enumerate}
  \item Check that 
    $\Pa_{k,\ell,m}^{k+\ell+m}$
    and
    $\Pa_{k,\ell+m}^{k+\ell+m}\circ(\Id_{\hS_k}\otimes\Pa_{\ell,m}^{\ell+m})$
    are skein equivalent for any $k,\ell,m\in\Z$. 
  \item Show that 
    $$\Pa^{k_1,\ldots,k_n}_{k_1+\cdots+k_n}\circ
    \Pa_{k_1,\ldots,k_n}^{k_1+\cdots+k_n}\equiv
    \frac{(-1)^{(r-1)(n-1)}}{\eta^{n-1}}
    \Id_{\hS_{k_1}\sqcup\cdots\sqcup\hS_{k_n}}\et$$
    $$\Pa_{k_1,\ldots,k_n}^{k_1+\cdots+k_n}\circ
    \Pa^{k_1,\ldots,k_n}_{k_1+\cdots+k_n}\equiv
    \frac{(-1)^{(r-1)(n-1)}}{\eta^{n-1}}
    \Id_{\hS_{\sum_i k_i}}$$
     for any
    $k_1,\ldots,k_n\in\Z$.
   Thus, the pants $\Pa_{k_1,\ldots,k_n}^{k_1+\cdots+k_n}$ lead to
 an isomorphisms
    \begin{equation}
      \label{eq:isohS}
      \V(\Su\sqcup\hS_{k_1}\sqcup\cdots\sqcup\hS_{k_n})\cong\V(\Su\sqcup\hS_{\sum_i k_i}
      )
    \end{equation}
    where  $\Su$ is any decorated surface.
      \end{enumerate}
\end{exo}
\begin{lemma} \label{lem:pantpermutation}
Consider the permutation
  $\tau_\sqcup:\hS_{k}\sqcup\hS_{\ell}\to \hS_{\ell}\sqcup\hS_{k}$ of the two
  components.  We can view $\tau_\sqcup$ as a cylinder cobordism.  
  Then $$\Pa_{k,\ell}^{k+\ell}\equiv((-1)^{r-1})^{k\ell}\Pa_{\ell,k}^{k+\ell}\circ\tau_\sqcup$$
  where the equivalence $\equiv$ is defined in Definition \ref{D:EquivCob}.
\end{lemma}
\begin{proof}
  According to the Exercise \ref{exo:pants}, $[\Pa_{k,\ell}^{k+\ell}]$ is a
  generator of the one dimensional space
  $\V(\hS_{k}\sqcup\hS_{\ell}\sqcup \wb{\hS_{k+\ell}})$.  Hence there exists
  $\lambda_{k,\ell}\in\C$ such that
  \begin{equation}\label{E:Pkltwist}
  [\Pa_{\ell,k}^{k+\ell}\circ\tau_\sqcup]=\lambda_{k,\ell}[\Pa_{k,\ell}^{k+\ell}]
  \end{equation}
  as vectors of $\V(\hS_{k}\sqcup\hS_{\ell}\sqcup\wb{\hS_{k+\ell}})$.  The
  cobordisms associated to both sides of this equality are diffeomorphic.  Such
  a diffeomorphism can be obtained by an isotopy permuting the two removed
  3-balls of Figure \ref{F:pants}.  This isotopy modifies the graph in
  $\Pa_{\ell,k}^{k+\ell}\circ\tau_\sqcup$ which is then no longer contained in the
  plane $y=0$ but in a neighborhood of $y=0$: see Figure \ref{F:twistedPants}.
  \begin{figure}
    \centering
    $\epsh{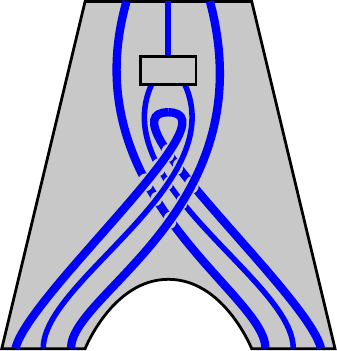}{12ex}$
    \caption{Twisted pants.}
    \label{F:twistedPants}
  \end{figure}
  To compute $\lambda_{k,\ell}$, we glue
  $\Pa_{k+\ell}^{k,l}\in\cV'(\hS_{k}\sqcup\hS_{\ell}\sqcup\wb{\hS_{k+\ell}})$
  to both sides of Equation \ref{E:Pkltwist}. 
  This yields two closed 3-manifolds $Y_{k,\ell}$ and $Y'_{k,\ell}$.  The
  strategy is then to compare $\Zr(Y'_{k,\ell})$ with $\Zr(Y'_{0,0})$ and
  $\Zr(Y_{k,\ell})$ with $\Zr(Y_{0,0})$.  This will give
  $\lambda_{k,\ell}=(s_r)^{k\ell}\lambda_{0,0}$.  Then we use 1-surgery to
  compute  $\lambda_{0,0}$. 

  The R-matrix acts on the module $\sigma\otimes\sigma$ by the scalar
  $q^{(2r')^2/2}=(q^{2r'})^{r'}$.  This scalar is just a sign and as
  $r'$ is always odd, it is equal to $s_r=q^{2r'}=(-1)^{r-1}$.  This implies that
  $$F\bp{\epsh{fig26}{8ex}
    \put(-21,-5){\ms{\sigma}}\put(-4,-5){\ms{\sigma}}} =s_r
  F\bp{\epsh{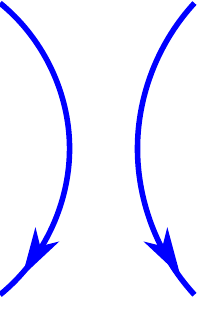}{8ex}
    \put(-21,-5){\ms{\sigma}}\put(-4,-5){\ms{\sigma}}} \et
  F\bp{\epsh{fig26}{8ex}
    \put(-22,-5){\ms{\sigma^k}}\put(-4,-5){\ms{\sigma^\ell}}}
  =(s_r)^{k\ell}\Id_{\sigma^{k+l}}. $$
  Applying these relations,
  we have that $Y'_{k,\ell}$ is $\sigma$-equivalent to $(s_r)^{k\ell+k+\ell}$
  times $Y'_{0,0}$ (see Equation \eqref{eq:sigmabraid}).  Also, $Y_{k,\ell}$
  is $\sigma$-equivalent to $(s_r)^{k+\ell}$ times $Y_{0,0}$.  Hence
  $$\lambda_{k,\ell}=(s_r)^{k\ell}\lambda_{0,0}.$$
  Finally, we compute $\lambda_{0,0}$: using Proposition \ref{P:1-surg} twice
  and Lemma \ref{L:cob-as-vect}
  we have that $\Pa_{0,0}^{0}\equiv (\eta\qd(0))^{-2} u_0\circ({u_0^*}\sqcup
  {u_0^*})$ where $u_0^*$ is the cobordism from
  $\hS_0$ to $\emptyset$ whose underlying manifold is the manifold of $u_0$
  with the opposite orientation.  Now  $({u_0^*}\sqcup
  {u_0^*})\circ\tau_\sqcup$ is diffeomorphic to ${u_0^*}\sqcup
  {u_0^*}$ thus $\Pa_{0,0}^{0}\circ\tau_\sqcup\equiv\Pa_{0,0}^{0}$ and $\lambda_{0,0}=1$.
\end{proof}
\subsection{Finite dimensionality of $\V$}\label{SS:FiniteDimVdiscon} In this subsection we will show that $\V(\Su)$ is finite dimensional when $\Su$ is disconnected. 

\begin{lemma}\label{lem:finitedimensionalityeasycase}
  Let $\Su$ be a connected decorated surface.  Then $\dim(\V(\Su\sqcup\hS_n))<\oo$  for
  any $n\in\Z$.  Moreover, $\V(\Su\sqcup\hS_n)=\{0\}$ for all but
  finitely many $n\in\Z$.
\end{lemma}
\begin{proof}
  Let $n\in\Z$ and $\hat\Su_n$ be $\Su$ with three more framed dots
  colored as the dots on $\hS_n$.  There is a cobordism
  $C_n:\hat\Su_n\to\Su\sqcup\hS_n$ given by removing a little 3-ball
  in $\Su\times[0,1]$ and joining vertically the dots of $\hS_n$ to
  the corresponding dots of $\hat\Su_n$.  We claim that
  $\V(C_n):\V(\hat\Su_n)\to\V(\Su\sqcup\hS_n)$ is surjective.  Indeed,
  by Theorem \ref{teo:skeinsurjects}, $\V(\Su\sqcup\hS_n)$ is
  generated by connected manifolds.  Let $M$ be a connected manifold
  in $\cV(\Su\sqcup\hS_n)$ and $c\subset M$ be an arc connecting the base-points of $\Su$ and of $\hS_n$.  Up to modifying $M$ by gluing the cylinder of
  Proposition \ref{prop:vspheres}(3), we may suppose that $\omega([c])=0$. The tubular neighborhood $N$ of $\Su\cup c\cup \hS_n$ in $M$ is then
  diffeomorphic to $C_n$ and its complement can be parametrized as an
  element $M'$ of $\cV(\hat\Su_n)$ such that $M=C_n\circ M'$.  Thus
  $[M]=\V(C_n)([M'])$.  It follows from Theorem
  \ref{teo:skeinsurjects} that
  $\dim(\V(\Su\sqcup\hS_n))\le\dim(\V(\hat\Su_n))<\oo$.

  Now
  $\hw(\hat\Su_n)=\hw(\Su)+2r'n+2r-2$ and
  $\lw(\hat\Su_n)=\lw(\Su)+2r'n-2r+2$ thus, applying Lemma
  \ref{L:boundweight} to $\hat\Su_n$, we have
  $$\V(\Su\sqcup\hS_n)\neq0\implies-\dfrac{(4g+10)(r-1)+\hw(\Su)}{2r'}\le n\le 
  \dfrac{(4g+10)(r-1)-\lw(\Su)}{2r'}.$$
\end{proof}

\begin{lemma}\label{lem:zerospheres}
  Let $k\in\Z\setminus r'\Z$, and $\alpha\in\C\setminus\Z$.  Let
  $S=S^2_{((V_{\alpha+2k},+),(V_{\alpha},-))}$ then 
   $\Id_S\equiv0$.
\end{lemma}
\begin{proof}
  Let $M$ be 
  the cobordism $\Id_S\equiv0$ seen as an element of 
  $\cV(S\sqcup\wb S)$ (we can since $S$ is admissible).  By
  Lemma \ref{L:cob-as-vect} it is enough to show that
  $[M]=0\in\V(S\sqcup\wb S)$.

  Let $M'\in \cV'(S\sqcup\wb S)$.  First, suppose that the connected component of $M'$ containing $S$
  does not contain $\wb S$.  In this case one may consider a surgery
  presentation of $M'\circ M$ which is cut by a sphere in $S^3$
  intersecting the graph in two strands colored by $V_{\alpha+2k}$ and
  $V_{\alpha}^{*}$.  Since $\Hom_\cat(\unit,V_{\alpha+2k}\otimes
  V_{-\alpha})=\{0\}$ if $k\neq0$ then $\Zr(M'\circ M)=0$. 
  
On the other hand, if a
  connected component of $M'$ contains both $S$ and $\wb S$, then we may
  present $M'\circ M$ by a surgery diagram containing a zero-framed
  meridian  encircling two strands
  colored by $V_{\alpha+2k}$ and $V_{-\alpha}$.  Here the zero-framed
  meridian corresponds to an $S^2\times S^1$ connected summand in
  $M'\circ M$ whose essential sphere is $S$.  Now applying Lemma
  \ref{lem:old71} we see that $\Zr(M'\circ M)=0$. Thus, we have proved that $[M]=0\in \V(S\sqcup\wb S)$.
\end{proof}
\begin{prop}\label{V=VV_0}
  For any decorated surface $\Su$, one has
  $\V(\Su\sqcup\hS_0)\cong\V(\Su)$.
\end{prop}
\begin{proof}
  Recall the decorated cobordism $u_0$ from $\emptyset$ to $S_0=\hS_0$ defined
  in Proposition \ref{P:1-surg}.  Similarly, let $u_0^*$ be the cobordism from
  $\hS_0$ to $\emptyset$ whose underlying manifold is the manifold of $u_0$
  with the opposite orientation.  Consider the decorated cobordisms
  $\Id_\Su\sqcup u_0:\Su\to\Su\sqcup\hS_0$ and $\Id_\Su\sqcup
  u_0^*:\Su\sqcup\hS_0\to\Su$.  Now $u_0^*\circ u_0$ is an unknot colored by $V_0$ in $S^3$ which has value $\eta\qd(0)$.  Thus $(\eta\qd(0))^{-1}\V(\Id_\Su\sqcup u_0^*)$ is a left inverse of
  $\V(\Id_\Su\sqcup u_0)$.  
Furthermore, Proposition \ref{P:1-surg} implies  $u_0\circ u_0^*=\eta\qd(0)\Id_{\hS_0}$.  Then Lemma \ref{L:cob-as-vect} implies $u_0\circ u_0^*\equiv \eta\qd(0)\Id_{\hS_0}$.  Thus, $\V(\Id_\Su\sqcup u_0)\circ(\eta\qd(0))^{-1}\V(\Id_\Su\sqcup  u_0^*)=\Id_{\Su\sqcup\hS_0}$ and we have the desired isomorphism.  
  \end{proof}
\begin{defi}\label{D:Pac}
  For each $c$-uple of decorated surfaces $(\Su_1,\ldots,\Su_c)$ and
  $\vec{n}=(n_1,...,n_c)\in\Z^c$, let $\Su=\bigsqcup_i\Su_i$ and $n=\sum_i n_i$.  
 For such data, define the map
  $$\Pac_{\vec{n}}:
  \V(\Su_1\sqcup \hS_{n_1})\otimes \cdots \otimes \V(\Su_c\sqcup
  \hS_{n_c})\to \V(\Su\sqcup \hS_{n})$$ given by
  $\V\bp{\Id_{\sqcup_i\Su_i}\sqcup \Pa_{\vec n}^{n}}\circ f\circ
  \bigsqcup$ where $\Pa_{\vec n}^{n}$ is the map defined in Exercise
  \ref{exo:pants}, $f$ is a reordering of the components (putting first $\Su_1$, then $\Su_2$ etc.) 
  and the big disjoint union is the generalization of the map of Proposition
  \ref{P:dunion}. 
 If $n=0$, let also define $\wh{\Pac_{\vec n}}:\bigotimes_i\V(\Su_i\sqcup
  \hS_{n_i})\to \V(\Su)$ as the map $\Pac_{\vec n}$ composed
  with the isomorphism of Proposition \ref{V=VV_0}.
\end{defi}

\begin{theo}\label{teo:findimvv}
  Let $\Su=\Su_1\sqcup \cdots \sqcup \Su_c$ be a decorated surface
  where each $\Su_i$ is connected.  
  There exists a finite set $I\subset \Z^c$ of $c$-uples such that for each $(n_1,\ldots n_c)\in I$ we have $\sum_i n_i=0$
  and
  the map:
  $$ \wh{\Pac}:=\bigoplus_{\vec n\in I} \wh{\Pac_{\vec n}}:\bigoplus_{\vec{n}\in I}  \left(\V(\Su_1\sqcup \hS_{n_1})\otimes 
    \cdots \otimes \V(\Su_1\sqcup \hS_{n_c})\right)\to \V(\Su)
  $$
  is surjective. Thus, $\V(\Su)$  is finite dimensional.
\end{theo}
\begin{proof}
  By Theorem \ref{teo:skeinsurjects}, it is enough to consider
  cobordisms of $\cV(\Su)$ represented by a connected sum
  $H=H_1\#H_2\#\cdots \#H_c$ of handlebodies (with $\partial
  H_i=\Su_i$) with some $\cat$-colored graph and some cohomology
  class.  Topologically this connected sum is diffeomorphic to the
  manifold obtained by gluing a pant $P$ with $c$ boundary spheres
  (i.e. the complement of $c$ open balls in $S^3$) to $\sqcup_{i=1}^c
  H'_i$ where $H'_i:=H_i\setminus \mathring{B}^3$ (for some open ball
  $B^3\subset H_i$).  By Proposition \ref{P:trick} we may suppose that
  there always is a strand of the $\cat$-colored graph which is
  colored by a generic projective module, $V_\alpha,\alpha \in
  \C\setminus \Z$.  Up to isotopy we may suppose that this strand
  intersects all the $c$ spheres in $\partial P$.  Then Corollary
  \ref{C:factor V0^2_New} implies that we may reduce to the case where
  these spheres are isomorphic to either $\hS_{n_i}$ or
  $S^2_{(V_{\alpha+2n_i},+),(V_\alpha,-)}$ with $n_i\in\Z\setminus
  r'\Z$.  In this last case, Lemma \ref{lem:zerospheres} implies that
  the cobordism is equivalent to $0$. Moreover, since the element of
  $\V(\hS_{-n_1}\sqcup\cdots\sqcup{\hS_{-n_c}})$ represented by $P$
    (and the part of graph contained in $P$) is a scalar multiple
  of $[u_0\circ\Pa_{n_1,\ldots n_c}^0]\in
  \V(\hS_{-n_1}\sqcup\cdots\sqcup{\hS_{-n_c}})$ (because
    $\V(\hS_{-n_1}\sqcup\cdots\sqcup{\hS_{-n_c}})$ has dimension $1$
      if $\sum n_i=0$ and $0$ else), then we may replace $P$ (along
      with the part of graph in it) with
      $u_0\circ\Pa_{n_1,\ldots,n_c}^0$ (recall that $u_0$ is the generator of $\V(\hS_0)$ represented by a $B^3$ containing an unknotted arc colored by $V_0$).
  Cutting along the $c$ spheres, shows that the vector space spanned
  by all the cobordisms whose underlying manifold is $H$ is in the
  image of the map $$\wh{\Pac}:\bigoplus_{n_1,\ldots n_{c}\in \Z}
  \V(\Su_1\sqcup \hS_{n_1})\otimes\cdots\otimes \V(\Su_c\sqcup
  \hS_{n_{c}})\to \V(\Su)$$ obtained by gluing along the essential
  spheres.  But the first statement of Lemma
  \ref{lem:finitedimensionalityeasycase} ensures that each of the
  $n_1,\ldots, n_c$ must range in a finite set and that
  $\dim(\V(\Su_i\sqcup \hS_{n_i}))<\infty$ for all $n_i$.
\end{proof}
\subsection{A grading on $\V$}\label{sub:grading}
Here we introduce a multigrading on $\V(\Su)$ given by the spectral
decomposition of the action of $H^0(\Su,\C/2\Z)$ on $\V(\Su)$.

Let $M=(M,T,f,\coh,n)\in\cV(\Su)$ and $\vp\in H^0(\Su,\C/2\Z)$.
Recall the cylinder $\Id_\Su^\vp$ given by the equality
\eqref{eq:IdH0}, then $\Id_\Su^\vp\circ M=(M,T,f,\coh+\partial\vp,n)$
where here the boundary map is part of the long exact sequence in
cohomology for the pair $(M\setminus
T,*)$:
$$H^0(\Su;\C/2\Z)=H^0(*;\C/2\Z)\stackrel{\partial}\longrightarrow 
H^1(M\setminus T,*;\C/2\Z).$$ This same exact sequence ensures that
$\partial\vp=0$ if and only if $\vp$ is the restriction of a
cohomology class of $H^0(M\setminus T;\C/2\Z)$.  This is in particular
the case for constant $\vp$.

Let $\Su$ be a decorated surface with $n$ ordered components.  The
ordering of the components of $\Su$ gives a canonical isomorphism
$H_0(\Su,\Z)\cong\Z^n$; explicitly if $\Su=\Su_1\sqcup \cdots \sqcup\Su_n$ and $*_i\in \Su_i$ then $\degr=\sum_i \de_i[*_i]\in H_0(\Su;\Z)$; we will also write $\degr=(\de_1,\ldots ,\de_n)\in \Z^n$ and let $|\degr|=\sum_i \de_i\in \Z$.
We consider the pairing
$H^0(\Su,\C/2\Z)\times H_0(\Su,\Z)\to\C^*$ given by $(\vp,\degr)\mapsto
q^{-2r'\vp(\degr)}$.
\begin{defi} For $\degr\in H_0(\Su,\Z)$, let 
  \begin{equation}
    \label{eq:Vh}
    \V_\degr(\Su)=\{[M]\in\V(\Su):\ \forall \vp\in H^0(\Su,\C/2\Z),\,
    \V(\Id_\Su^\vp)([M])=q^{-2r'\vp(\degr)}[M]\}
  \end{equation}
   \begin{equation}
    \label{eq:V'h}
    \V'_\degr(\Su)=\{[M]\in\V'(\Su):\ \forall \vp\in H^0(\Su,\C/2\Z),\,
    \V'(\Id_\Su^\vp)([M])=q^{2r'\vp(\degr)}[M]\}.
  \end{equation} 
  \end{defi}

\begin{prop}\label{prop:grading}
  Let $\Su$ be a decorated surface with $n$ connected components.
  \begin{enumerate}
  \item \label{propI1:grading} If  $\V_\degr(\Su)\neq\{0\}$ then $ |\degr |=0$; in particular, if
    $\Su$ is connected then $\V(\Su)=\V_0(\Su)$.
  \item \label{propI2:grading} If $\Su'$ is a decorated surface, 
  using the morphism defined in Proposition \ref{P:dunion}, it holds:  
    $$\V_\degr(\Su)\sqcup\V_{\degr'}(\Su')\subset\V_{(\degr,\degr')}(\Su\sqcup\Su').$$
  \item \label{propI3:grading}
   If  $\V_\degr(\Su\sqcup\hS_k)\neq\{0\}$ then $ \de_{n+1}=k$ where $\de_{n+1}$
    is the last coordinate of
    $\degr\in\Z^{n+1}=H_0(\Su\sqcup\hS_k,\Z)$.
 \item \label{propI4:grading} $\V(\Su)=\bigoplus_{\degr \in\Z^n}\V_\degr(\Su).$
 \item \label{propI5:grading} If $M\in \V_{\degr_1}(\Su)$ and $ N\in \V'_{\degr_2}(\Su)$ then
   $\brk{N,M}=0$ unless $\degr_1+\degr_2=0$. Moreover, for any $\degr\in
   \Z^n$ the pairing $\V'_{-\degr}(\Su)\otimes \V_{\degr}(\Su)\to \C$
   is non-degenerate.
 \end{enumerate}
\end{prop}
\begin{proof}
  \eqref{propI1:grading} Let $[M]\in \V_{\degr}(\Su)$ and $|\degr|\neq0$.  
 If  $\vp\in H^0(\Su\setminus \{p_i\};\C/2\Z)$ is a class with a constant generic value $\wb\alpha\notin\Q/2\Z$ on each base point then $\partial\vp=0\in
  H^1(M\setminus T,*;\C/2\Z)$ and so $\Id_\Su^\vp\circ M=M$.  
Thus, $\V(\Id_\Su^\vp)([M])=V(\Id_\Su^\vp\circ M)=[M]$.  
 But $[M]\in \V_{\degr}(\Su)$ and so by definition $\V(\Id_\Su^\vp)$ acts on $[M]$ by $q^{-2r'\vp(\degr)}=q^{-2r'\alpha |\degr|}\neq1$.  Hence $[M]=0$.

  \eqref{propI2:grading} Let $[M]\in \V_\degr(\Su), [M']\in \V_{\degr'}(\Su')$ and $\vp\in
  H^0(\Su\sqcup \Su'\setminus\{p_i\};\C/2\Z)$.  Then
  $\Id_{\Su\sqcup\Su'}^\vp=\Id_{\Su}^{\vp_1}\sqcup\Id_{\Su'}^{\vp_2}$
  where $\vp_1=\vp|_{\Su}$ and $\vp_2=\vp|_{\Su'}$.  Thus
  \begin{align*}
    \V(\Id_{\Su\sqcup\Su'}^\vp)([M]\sqcup
  [M'])&=\V(\Id_{\Su}^{\vp_1})([M])\sqcup\V(\Id_{\Su'}^{\vp_2})([M]')
  =q^{-2r'\vp_1(\degr)}q^{-2r'\vp_2(\degr')}[M]\sqcup [M']\\
  &=q^{-2r'\vp(\degr,\degr')}[M]\sqcup [M'].
  \end{align*}

  \eqref{propI3:grading} Let $\vp\in
  H^0(\Su\sqcup \hS_k\setminus\{p_i\};\C/2\Z)$ be the class which is
  zero on all components except for $\hS_k$ where it is $\beta$.  Then  $\Id^\vp_{\Su\sqcup\hS_k}=\Id_{\Su}\sqcup\Id^\beta_{\hS_k}$.  Moreover, by Proposition \ref{prop:vspheres}\eqref{Ei3:CylProp},
  we have $\Id^\beta_{\hS_k}\equiv 
  q^{-2r'\beta k}\Id_{\hS_k}$ and so 
  \begin{equation}\label{E:IdSuSkcont}
  \Id^\vp_{\Su\sqcup\hS_k}\equiv
  q^{-2r'\beta k}\Id_{\Su\sqcup\hS_k}.
  \end{equation}  Finally, by definition
  $\V(\Id^\vp_{\Su\sqcup\hS_k})$ acts by $q^{-2r'\vp(\degr)}=q^{-2r'\beta \de_{n+1}}$ on
  $\V_\degr(\Su\sqcup \hS_k)$.  Thus, if $[M]\in \V_\degr(\Su\sqcup \hS_k)$ then $\V(\Id^\vp_{\Su\sqcup\hS_k})[M]=0$  unless $\de_{n+1}=k$. 

  \eqref{propI4:grading} First, we prove the statement in the case when $\Su=\Su_1\sqcup \hS_k$ with
  $\Su_1$ connected.  In this case $H^0(\Su\setminus\{p_i\};\C/2\Z)=(\C/2\Z)^2$.  
  If $\vp=(\alpha, \alpha)\in
  (\C/2\Z)^2=H^0(\Su\setminus\{p_i\};\C/2\Z)$ is a constant cohomology
  class then $\partial\vp=0$.
  Therefore,
  for such $\vp=(\alpha, \alpha)$, the action of the cylinder
  $\V(\Id_\Su^\vp)$ on $\V(\Su)$ is trivial.  Next, consider classes
  of the form $\vp=(0,\beta)\in (\C/2\Z)^2$.  For such
  $\vp=(0,\beta)$, Equation \eqref{E:IdSuSkcont} implies that
  $\V(\Id_\Su^\vp)$ acts on $\V(\Su)$ by $q^{-2r'\beta k}$.  Finally,
  if $\vp=(\alpha, \alpha+\beta)\in (\C/2\Z)^2$ is a general
  $0$-cohomology class then from the discussion above $\V(\Id_\Su^\vp)$
  acts on $\V(\Su)$ by 
  $q^{-2r'\beta k}=q^{-2r'(-k \alpha +k(\alpha+\beta) )}$.  Hence
  $\V(\Su)\subset \V_{(-k,k)}(\Su)$ and so $\V(\Su)=\V_{(-k,k)}(\Su)$.

  In general, let $\Su=\Su_1\sqcup \cdots \sqcup \Su_c$ and  $\vp\in
  H^0(\Su\setminus\{p_i\};\C/2\Z)$.  By Theorem
  \ref{teo:findimvv} the map
 \begin{equation*}
   \wh{\Pac}=\oplus_{\vec n}\wh{\Pac_{\vec n}}:
   \bigoplus_{\vec n=(n_1,\ldots n_{c})\in \Z^n,\, \sum n_i =0}
   \V(\Su_1\sqcup \hS_{n_1})\otimes\cdots\otimes \V(\Su_c\sqcup
   \hS_{n_{c}})\to \V(\Su)
  \end{equation*}
  is surjective.  
  Now, we have 
  \begin{equation}\label{E:VIdcommPn}
  \V(\Id_\Su^\vp)\circ\wh{\Pac_{\vec n}}
  =\wh{\Pac_{\vec n}}\circ\bp{\otimes_j\V(\Id^{\vp_j}_{\Su_j}\sqcup\Id_{\hS_{n_j}})}
  =q^{2r'\vp(\vec n)}\wh{\Pac_{\vec n}}
  \end{equation}
  where $\vp_j=\vp_{|\Su_j}$ and the second equality follows from the discussion in the previous paragraph applied to each summand.  
 So the image of $\wh{\Pac_{\vec n}}$ is in $\V_{-\vec n}(\Su)$.  Thus, since the map  $\wh{\Pac}=\oplus_{\vec n}\wh{\Pac_{\vec n}}$ is surjective the result follows.  
 
 \eqref{propI5:grading} For all $\vp\in H^0(\Su\setminus\{p_i\};\C/2\Z)$ we have  $$\brk{[N],[M]}=\brk{[N\circ\Id_{\Su}^{-\vp}], [\Id_{\Su}^{\vp}\circ M]}
 =q^{-2r'\vp(\degr_1+\degr_2)}\brk{[N],[M]}.$$
 This implies that $\brk{N,M}=0$ if $\degr_1+\degr_2\neq 0$.  Finally, by construction we know that the pairing $\V'(\Su)\otimes
 \V(\Su)\to \C$ is non-degenerate, and we just proved that it is zero
 on pairs of vectors whose degrees do not sum to $0$, hence the last
 statement follows.
\end{proof}

\section{The monoidal TQFT $\VV$}\label{sec:VV}
If $\Su$ is a decorated surface with $n$ connected components then Proposition \ref{prop:grading} implies $\V(\Su)$ is a finite
dimensional $\Z^{n}$-graded vector space, completely contained in the
span of vectors whose multidegree has total sum $0$.   This is actually
the reason why $\V$ is not a monoidal (see Proposition \ref{prop:vspheres}). 
Thus we are led to define a new functor $\VV$ (see Definition \ref{def:VV}) which we prove to be a monoidal functor with values in the category 
of graded vector spaces and degree $0$-morphisms in Theorem \ref{teo:monoidalVV}. 
One of the new interesting features of the functor $\VV$ is that in order to get monoidality one should use a symmetric 
braiding on the latter category which is trivial if $r$ is odd and else it is that of super vector spaces (see Remark \ref{rem:tau}). 
In this sense, for $r$ even our TQFT is a ``super TQFT''. 
Similar results hold for the dual version of $\VV$, denoted $\VV'$. 

Finally, in Subsection \ref{sub:verlinde} we use the properties of this functor to compute the 
graded dimension of $\VV(\Su)$ when $\Su$ is an admissible decorated surface. 

\subsection{Definition of the functor}
\begin{defi}[The 
functor $\VV$]\label{def:VV} For a decorated 
  surface  $\Su$ we define the $\Z$-graded vector space
  $$\VV(\Su)=\bigoplus_{m\in\Z}\VV^m(\Su)\quad\text{ where }\quad
  \VV^m(\Su)=\V(\Su\sqcup \hS_{-m}).$$ By Lemma
  \ref{lem:finitedimensionalityeasycase} the dimension of $\VV(\Su)$
  is finite.  
\end{defi}
A cobordism $M:\Su\to \Su'$ induces a \emph{degree $0$}-map
$\VV(M):\VV(\Su)\to \VV(\Su')$ by acting on the degree $-m$ submodules 
via $$\V\big(M \sqcup \Id_{\hS_m}\big):\V(\Su\sqcup \hS_{m}) \to \V(\Su'\sqcup \hS_{m}).$$
Hence $\VV$ defines a functor from the category $\Cob$ of decorated
cobordisms to the category of $\Z$-graded vector spaces.
Let $\Su$ be a decorated surface with $n$ connected components.  The
$\Z$-grading of $\VV(\Su)$ can be refined to a $\Z^n$-grading:  If
$\degr\in\Z^n$, recall that $|\degr|\in\Z$ is the sum of the
coordinates of $\degr$.  Let
$\VV_\degr(\Su)=\V_{(\degr,-|\degr|)}(\Su\sqcup \hS_{-|\degr|})$, then
Proposition \ref{prop:grading} implies that
$\VV^m(\Su)=\bigoplus_{\degr\in\Z^n,\,|\degr|=m}\VV_\degr(\Su)$.
Hence one has
$$\VV(\Su)=\bigoplus_{\degr\in\Z^n}\VV_\degr(\Su).$$
We will refer to this $\Z^n$-grading as the \emph{multidegree} of
elements of $\VV(\Su)$.

\begin{defi}[The functor $\VV'$]
  Let $\Su$ be a decorated surface.  We define the $\Z$-graded space
  $$\VV'(\Su)=\bigoplus_{m\in\Z}{\VV'}^m(\Su)\quad\text{ where }\quad
  {\VV'}^m(\Su)=\V'(\Su\sqcup \hS_{m}).$$
\end{defi}
As above, for a decorated surface with $n$ connected
components $\Su$, we define the \emph{multidegree} $\degr\in\Z^n$
subspace of $\Su$ as
$\VV'_\degr(\Su)=\V'_{(\degr,-|\degr|)}(\Su\sqcup\hS_{-|\degr|})$.

A cobordism $M:\Su\to \Su'$ induces a \emph{degree $0$}-map
$\VV'(M):\VV'(\Su')\to \VV'(\Su)$ by acting on the degree $m$
submodules via $ \V'\big(M \sqcup \Id_{\hS_m}\big).$ We get a
contravariant functor $\VV'$ from $\Cob$ with values in graded vector
spaces.
\begin{defi}[Homogeneous pairing]\label{D:VVpairing}
  We define a pairing $\brkk{\cdot,\cdot}_\Su:\VV'(\Su)\otimes
  \VV(\Su)\to \C$ by extending
  bilinearly $$\brkk{M',M}_\Su=\delta_{a,-b}\brk{M',M}_{\Su\sqcup\hS_a},
  \ \ \ \forall M\in {\VV}^b(\Su),\ \forall M'\in {\VV'}^a(\Su).$$ It
  is a direct consequence of the last point of Proposition
  \ref{prop:grading} that $\brkk{\cdot,\cdot}_\Su$ is non-degenerate
  and realizes isomorphisms $\VV'_{-\degr}(\Su)\cong\VV_\degr(\Su)^*$ for
  any $\degr\in\Z^n=H_0(\Su)$.  In particular, ${\VV'}(\Su)$ is a
  finite dimensional vector space.
\end{defi}

\subsection{Monoidality of $\VV$}
Let $\Su,\Su'$ be two
decorated surfaces.  Consider the map
$$\Pac:\VV(\Su)\otimes \VV(\Su')\cong\bigoplus_{m,n\in\Z}\V(\Su\sqcup\hS_m)
\otimes \V(\Su'\sqcup\hS_n)\longrightarrow \VV(\Su\sqcup \Su')$$ given by the
essentially finite direct sum $\Pac=\bigoplus_{m,n\in\Z}\Pac_{(m,n)}$
of the maps\\ $$\Pac_{(m,n)}: \V(\Su\sqcup\hS_m)\otimes
\V(\Su'\sqcup\hS_n)\to \V(\Su\sqcup\Su'\sqcup\hS_{m+n})$$ introduced in
Definition \ref{D:Pac}.
\begin{theo}[Monoidality of $\VV$]\label{teo:monoidalVV}
  Let $\Su,\Su'$ be two (possibly disconnected) decorated
  surfaces.  Then 
  $$\Pac:\VV(\Su)\otimes \VV(\Su')\to \VV(\Su\sqcup \Su')$$
  is an isomorphism of multigraded vector spaces.
 Moreover if $M_1:\Su_1\to \Su'_1$ and $M_2:\Su_2\to \Su'_2$ are
  decorated cobordisms then
  \begin{equation}\label{E:VcommPU} 
  \VV(M_1\sqcup M_2)\circ \Pac= \Pac\circ\left(\VV(M_1)\otimes_{\C} \VV(M_2)\right).
  \end{equation}
\end{theo}
\begin{proof}
  Suppose that $\Su$ and $\Su'$ have $c$ and $c'$
  components, respectively.  By Theorem \ref{teo:findimvv}, $\VV(\Su), \VV(\Su')$,
  and $\VV(\Su \sqcup \Su')$ are all finite dimensional vector spaces over
  $\C$.

  We first prove the last statement of the theorem.  As mentioned above, both $M_1$ and $M_2$
  induce degree $0$ morphisms $$\VV(M_i): \VV(\Su_i)\to \VV(\Su'_i), \text{ for } 
   i=1,2.$$  So it is sufficient to prove the statement for fixed
  degrees $m_1,m_2\in\Z$.  Let
  $$\sigma_\sqcup:\Su_1\sqcup\hS_{-m_1}\sqcup\Su_2\sqcup\hS_{-m_2}
  \to\Su_1\sqcup\Su_2\sqcup\hS_{-m_1}\sqcup\hS_{-m_2}$$ and
  $$\sigma'_\sqcup:\Su'_1\sqcup\hS_{-m_1}\sqcup\Su'_2\sqcup\hS_{-m_2}
  \to\Su'_1\sqcup\Su'_2\sqcup\hS_{-m_1}\sqcup\hS_{-m_2}$$ be the
  trivial cobordisms permuting the components and
  $$\Pa=\Pa_{-m_1,-m_2}^{-m_1-m_2}:\hS_{-m_1}\sqcup\hS_{-m_2}\to\hS_{-m_1-m_2}.$$
  On $\VV^{m_1}(\Su_1)\otimes \VV^{m_2}(\Su_2)$ one has
  \begin{align*}
    \Pac\circ\bp{\VV(M_1)\otimes_{\C} \VV(M_2)}&
    =\V\bp{\Id_{\Su'_1\sqcup \Su'_2}\sqcup\Pa}\circ\V(\sigma'_\sqcup)
    \circ\sqcup\circ\bp{\V(M_1\sqcup\Id_{\hS_{-m_1}})\otimes \V(M_2\sqcup\Id_{\hS_{-m_2}})}\\
    &=\V\bp{\Id_{\Su'_1\sqcup \Su'_2}\sqcup\Pa}\circ \V\bp{M_1\sqcup
      M_2\sqcup\Id_{\hS_{-m_1}}\sqcup \Id_{\hS_{-m_2}}}
    \circ\V(\sigma_\sqcup)\circ\sqcup\\
    &=\V(M_1\sqcup M_2\sqcup\Pa)\circ\V(\sigma_\sqcup)\circ\sqcup\\
    &=\V\bp{M_1\sqcup M_2\sqcup\Id_{\hS_{-m_1-m_2}}}
    \circ\V\bp{\Id_{\Su_1\sqcup \Su_2}\sqcup\Pa}
    \circ\V(\sigma_\sqcup)\circ\sqcup\\
    &=\VV(M_1\sqcup M_2)\circ \Pac
  \end{align*}
which proves Equation \eqref{E:VcommPU}.  

 Next we prove that $\Pac$ is an isomorphism in several steps.  First, we claim that $\Pac$ is homogeneous with respect to the
  multigradings, i.e. for any $\degr\in \Z^c,\ \degr'\in \Z^{c'}$ we
  have
  \begin{equation}\label{eq:Pacmultideg}
    \Pac:\VV_{\degr}(\Su)\otimes \VV_{\degr'}(\Su')\to
    \VV_{(\degr,\degr')}(\Su\sqcup \Su'). 
  \end{equation}
  This follows from Proposition \ref{prop:grading}(2) and the fact
  that for  $\vp\in H^0(\Su\sqcup\Su',\C/2\Z)$ the map $\V(\Id_{\Su \sqcup \Su'}^\vp)$
  commutes with the
  composition with a pants glued on the spheres (also see Equation \eqref{E:VIdcommPn} in the proof of
  Proposition \ref{prop:grading}(4)).

Equation \eqref{eq:Pacmultideg} implies that
 it is enough to show that $\Pac$ is
  bijective in each multidegree.  Let $\degr\in\Z^c$ and
  $\degr'\in\Z^{c'}$ and let $k=-|\degr|$, $k'=-|\degr'|$.  The map
  $\Pac:\VV_{\degr}(\Su)\otimes \VV_{\degr'}(\Su')\to
  \VV_{(\degr,\degr')}(\Su\sqcup \Su')$ is the composition of the
  ``disjoint union map'' (see Proposition \ref{P:dunion}), of a
  permutation of the components and of the map $\V(\Id_{\Su\sqcup
    \Su'}\otimes \Pa_{k,k'}^{k+k'})$.  Since the last two maps are
  bijective, it is enough to show that the disjoint union
  map 
  \begin{equation}
    \label{eq:sqcupdeg}
    \V_{(\degr,k)}(\Su\sqcup\hS_k)\otimes
    \V_{(\degr',k')}(\Su'\sqcup\hS_{k'})\stackrel{\sqcup}\longrightarrow
    \V_{(\degr,k,\degr',k')}(\Su\sqcup\hS_k\sqcup\Su'\sqcup\hS_{k'})
  \end{equation}
  is bijective.  It is injective by Proposition \ref{P:dunion}.
  
  We now claim that this map is surjective. We consider $\vec
  n_1=(-\degr,-k)$, $\vec n_2=(-\degr',-k')$ and $\vec n=(\vec
  n_1,\vec n_2)\in\Z^{c+c'+2}$ and apply Theorem \ref{teo:findimvv} to
  the sequence of decorated surfaces
  $$\Gamma=(\Su_1,\ldots,\Su_c,\hS_{k},
  \Su'_1,\ldots,\Su'_{c'},\hS_{k'}).$$  Hence we have a surjective map
  onto $\V(\Su\sqcup\hS_{k}\sqcup\Su'\sqcup\hS_{k'})$.  As above, this
  map commutes with 
   the map $\V(\Id_{\Su\sqcup\hS_{k}\sqcup\Su'\sqcup\hS_{k'}}^\vp)$
   for  $\vp\in H^0(\Su\sqcup\hS_{k}\sqcup\Su'\sqcup\hS_{k'},\C/2\Z)$.
  Thus,  Proposition~\ref{prop:grading} implies that the only summand
  contributing in multidegree ${(\degr,k,\degr',k')}$ is
  $$\wh{\Pac_{\vec n}}=u_0^*\circ\Pac_{\vec n}:
  \bigotimes_i\V(\Gamma_i\sqcup\hS_{n_i})\to
  \V_{(\degr,k,\degr',k')}(\Su\sqcup\hS_{k}\sqcup\Su'\sqcup\hS_{k'})$$
  which is consequently also surjective (here $\Gamma_i$ is the
  $i$\textsuperscript{th} surface of the sequence $\Gamma$).
  Now Proposition \ref{P:1-surg} implies that $u_0^*\circ
  \Pa_{0,0}^0\equiv(\eta\qd(0))^{-1}(u_0^*\sqcup u_0^*)$.  We also have  that $\Pa^0_{\vec n}$ and
  $\Pa^0_{(0,0)}\circ(\Pa^0_{\vec n_1}\sqcup\Pa^0_{\vec n_2})$ are skein
  equivalent (see Exercise \ref{exo:pants}).  These two relations imply that for any
  $\otimes_ix_i\in\bigotimes_i\V(\Gamma_i\sqcup\hS_{n_i})$,
  \begin{align*}
    \wh{\Pac_{\vec n}}(\otimes_ix_i)&=
    \V(\Id\sqcup u_0^*)\circ\Pac_{\vec n}(\otimes_ix_i)\\
    &= \V(\Id\sqcup u_0^*)\circ\Pac_{(0,0)}\circ(\Pac_{\vec
      n_1}\otimes\Pac_{\vec
      n_2})(\otimes_ix_i)\\
    &=(\eta\qd(0))^{-1}\wh{\Pac_{\vec
        n_1}}(\otimes_{i=1}^{c+1}x_i)\sqcup\wh{\Pac_{\vec
        n_1}}(\otimes_{i=1}^{c'+1}x_{c+1+i})
  \end{align*}
  which is in $\V_{(\degr,k)}(\Su\sqcup\hS_k)\sqcup
    \V_{(\degr',k')}(\Su'\sqcup\hS_{k'})$.
  Hence the map of Equation \eqref{eq:sqcupdeg} is surjective, so
  $\Pac$ is bijective.
\end{proof}
\begin{cor}\label{C:FunctorExists}
  The functor $\VV:\Cob\to\GrVect$ is monoidal with coherence maps
  $\Pac$ and
  $u_0:\C\stackrel{\sim}\longrightarrow\VV(\emptyset)=\V(\hS_0)$
 where $\GrVect$ is the category of finite dimensional $\Z^n$-graded $\C$-vector spaces.

\end{cor}
\begin{proof}
  The previous theorem shows that $\Pac$ is a natural isomorphism.  The
  associativity and unitality constraints are easily checked using
  Exercise \ref{exo:pants} and $\Pa_{0,k}^k\circ (u_0\sqcup
  \Id_{\hS_k})= \Id_{\hS_k}=\Pa_{k,0}^k\circ (\Id_{\hS_k}\sqcup u_0)$.
\end{proof}
\begin{rem}
  Proposition \ref{V=VV_0} can be reformulated as follows: for any
  decorated surface $\Su$, there is a natural isomorphism of multigraded spaces
  $$\VV^0(\Su)\stackrel\sim\longrightarrow\V(\Su).$$
  Furthermore, if $M$ is a closed decorated manifold seen as a
  cobordism $\emptyset\to\emptyset$ then
  \begin{equation}
    \label{eq:VV=Z}
      \VV(M)=\Zr(M)\Id_{\VV(\emptyset)}.
  \end{equation}
\end{rem}

The pairing $\brkk{\cdot,\cdot}$ of Definition \ref{D:VVpairing}
identifies $\VV'$ and 
the dual $\Hom_\C(\VV(\cdot),\C)$ of $\VV$.
Hence, the  transpose of $\Pac$ gives the natural isomorphisms for the
monoidality of $\VV'$ (see the following corollary).
\begin{cor}[Monoidality of $\VV'$]
  Let $\Su_1,\Su_2$ be two (possibly disconnected) decorated
  surfaces. There is an isomorphism of multigraded vector spaces
  $${\Paca}:\VV'(\Su_1\sqcup \Su_2)\to \VV'(\Su_1)\otimes_{\C} \VV'(\Su_2)$$
  defined as the adjoint of the isomorphism $\Pac$ relative to the
  bilinear pairing $\brkk{\cdot,\cdot}$.  

  Moreover if $M_1:\Su_1\to \Su'_1$ and $M_2:\Su_2\to \Su'_2$ are
  decorated cobordisms then
  $$\Paca\circ\VV'(M_1\sqcup M_2)= 
  \left(\VV'(M_1)\otimes_{\C} \VV'(M_2)\right)\circ \Paca.$$ 
\end{cor}
\begin{proof}
  The pairing $(\VV'(\Su_1)\otimes
  \VV'(\Su_2))\otimes(\VV(\Su_1)\otimes \VV(\Su_2))\to\C$ is defined
  as the product of the two pairings
  $\brkk{\cdot,\cdot}_{\Su_1}\otimes\brkk{\cdot,\cdot}_{\Su_2}$.  By
  definition, for any $N_1\in\cV(\Su_1)$, $N_2\in\cV(\Su_2)$ and
  $N'\in\cV'(\Su_1\sqcup \Su_2)$,
  $\brkk{\Paca([N']),[N_1]\otimes[N_2]}=\brkk{[N'],\Pac([N_1]\otimes[N_2])}$.
  The map $\Paca$ is an isomorphism because $\Pac$ is an isomorphism
  and the pairing are non degenerate.  The second part of the
  corollary also follows from a similar statement for $\VV$.  Remark
  that Proposition \ref{P:pants} implies that in degree $m_1,m_2\in\Z$, 
  $(\Paca)^{-1}:\VV'(\Su_1)\otimes_{\C} \VV'(\Su_2)\to
  \VV'(\Su_1\sqcup \Su_2)$ is given by
  $$(\Paca)^{-1}=(\eta\qd(0))^{-1}\V'\bp{\Id_{\Su_1\sqcup \Su_2}\sqcup\Pa^{m_1,m_2}_{m_1+m_2}}
  \circ\V'(\Id\sqcup\tau_\sqcup\sqcup\Id)\circ\sqcup$$ where
  $\tau_\sqcup:\Su_2\sqcup\hS_{m_1}\to \hS_{m_1}\sqcup\Su_2$ is the
  trivial cobordism. 
\end{proof}

\begin{rem}[The braiding of the monoidal category]\label{rem:tau}
  Recall that because the connected components of decorated surfaces
  are ordered, we have $\Su_1\sqcup \Su_2\neq \Su_2\sqcup \Su_1$.  But
  these two surfaces are isomorphic through the reordering that can be
  realized as the trivial cobordism $\tau_\sqcup$.  Composing this
  trivial braiding with the isomorphisms $\Pac$, we get a map
  $$\tau=(\Pac)^{-1}\circ \V(\tau_\sqcup)\circ\Pac:
  \VV(\Su_1)\otimes\VV(\Su_2)\to\VV(\Su_2)\otimes\VV(\Su_1)$$ Because
  of the use of $\Pa_{\de_1,\de_2}^{\de_1+\de_2}$ in the definition of
 $\Pac$ and because of Lemma
  \ref{lem:pantpermutation},  
  we get 
  $$\tau(x\otimes y)=(-1)^{\de_1\de_2(r-1)} y\otimes x$$
  for homogeneous vectors $x\otimes y\in
  \VV^{\de_1}(\Su_1)\otimes \VV^{\de_2}(\Su_2)$.   
  A similar statement holds for
  $\tau^*=\Paca\circ \V'(\tau_\sqcup)\circ(\Paca)^{-1}$.

  Observe then that if $r$ is odd this boils down to a trivial
  morphism coinciding with the standard flip and $\VV$ is a braided functor.  On the other hand, if $r$ is even then a super commutation rule appears.    Hence for $r$ even, $\VV$ is a braided monoidal functor with target
  the category of super vector spaces.
\end{rem}

\subsection{The Verlinde formula and the graded dimensions of $\VV(\Su)$.}\label{sub:verlinde}
The category $\Cob$ is not pivotal because non admissible surfaces
do not have a dual.  Nevertheless its full sub-category of admissible
surfaces (see Definition \ref{D:admS}) is pivotal and even ribbon.  Moreover,  one can see that
$\VV$ is a ribbon functor on this sub-category.  Then $\VV$ commutes with
the categorical trace and, as we now explain, leads to a kind of graded
Verlinde formula.

Let $\Su=(\Su,\{p_i\},\coh,\La)$ be a connected decorated surface that is
admissible (cf Definition \ref{D:admS}).  For
each ${\wb\beta}\in \C/2\Z$ let $\Id_\Su^\vp:\Su\to\Su$ 
where $\vp$ is the cohomology class of $H^0\left(\Su;\C/2\Z\right)$ which is
$\wb \beta$ on the base point.

Let $\cup:\emptyset\to \Su\sqcup\overline{\Su}$ and
$\cap:\overline\Su\sqcup{\Su}\to\emptyset$ be given by the cylinder of $\Su$.
Then 
\begin{equation}
  \label{eq:Sudual}
  ( \Id_{\Su}\sqcup\ \cap)\circ (\cup\sqcup\Id_{\Su})=\Id_{\Su}.
\end{equation}
Now for $\wb\beta\in\C/2\Z$, let $\Su\times S^1_{\wb\beta}$ be the decorated
3-manifold obtained as follows: the 3-manifold is $\Su\times S^1$, it contains
the link $T=\{p_i\}\times S^1$ and its cohomology class is given by
$\coh\oplus\wb\beta\in H^1(\Su\times S^1\setminus T,\C/2\Z)\simeq
H^1(\Su\setminus \{p_i\},\C/2\Z)\oplus H^1(S^1,\C/2\Z)\simeq H^1(\Su\setminus
\{p_i\},\C/2\Z)\oplus \C/2\Z$ (the first isomorphism comes from K\"unneth
formula: $H_1((\Su\setminus \{p_i\})\times S^1)\cong H_1(\Su\setminus
\{p_i\})\oplus H_1(S^1)$).  It is clear that $\Su\times S^1_{\wb\beta}$
can be decomposed as
$$\Su\times S^1_{\wb\beta}=\cap\circ \tau_\sqcup\circ 
(\Id_\Su^{{\wb\beta}}\sqcup \Id_{\overline{\Su}})\circ\cup.$$

 Let now
$(e_i)_i$ be a homogeneous basis of $\VV(\Su)$ and write
$(\Pac)^{-1}([\cup\sqcup u_0])=\sum_ie_i\otimes
e^i\in\VV(\Su)\otimes\VV(\wb\Su)$.  Applying the morphisms of
\eqref{eq:Sudual} to $e_j$ we have $\VV(\cap)(\Pac(e^i\otimes
e_j))=\delta_i^j[u_0]$.  Let $\de_i$ be the degree of $e_i$, then $e^i$ has
degree $-\de_i$ because
$\VV(\cup)([u_0])\in\VV^0(\Su\sqcup\overline{\Su})$.  So we compute
\begin{align*}
  \VV(\Su\times S^1_{\wb\beta})([u_0])&= \VV(\cap)\circ \Pac\circ \tau\circ
  (\VV(\Id_\Su^{{\wb\beta}})\otimes\Id_{\VV(\overline{\Su})})
  (\sum_ie_i\otimes e^i)\\
  &=\sum_iq^{-2r'\de_i\beta}\VV(\cap)\circ\Pac\circ \tau(e_i\otimes
  e^i)\\
  &=\sum_i(-1)^{(r-1)\de_i}q^{-2r'\de_i\beta}\VV(\cap)\circ\Pac(e^i\otimes
  e_i)\\ 
  &=\bp{\sum_{\de\in\Z}(-1)^{(r-1)\de}q^{-2r'\de\beta}\dim(\VV_\de(\Su))}[u_0]
\end{align*}
and this with the identity \eqref{eq:VV=Z} leads to a graded version of the
Verlinde formula:

\begin{theo}[Verlinde formula for graded dimensions]\label{P:verlinde}
  Let $\Su=(\Su,\{p_1,\ldots p_n\},\coh,{\La})$ be a connected admissible
  surface of genus $g$.
  Define the graded (super)-dimension of $\VV(\Su)$ by
  $$\dim_t(\VV(\Su))=\left\{
    \begin{array}{ll}
      \sum_{{\de}\in \Z} \dim(\VV_{{\de}}(\Su))t^{\de} &\text{ if $r$ odd}\\
      \sum_{{\de}\in \Z} (-1)^{\de}\dim(\VV_{{\de}}(\Su))t^{\de} &\text{ if $r$ even.}
    \end{array}\right.$$
  $$\text{Then}\qquad\dim_{q^{-2r'\beta}}(\VV(\Su))=\Zr(\Su\times S^1_{\wb\beta}).$$
  If $W=\otimes_iW_i$ is the tensor product of the
  colors of the $p_i$, then 
  $$\Zr(\Su\times S^1_{\wb\beta})=\frac1r(r')^{g} \sum_{k\in
    \Hr}\bp{\frac{\qn{r\beta}}{\qn{\beta+k}}}^{2g-2}\Phi_{W,\beta+k}$$
  where $\Phi_{W,\beta}$ is the scalar
   associated  to an open Hopf link whose long and closed strands are colored by $V_\beta$ and $W$, respectively 
   (see Equation \eqref{eq:Phi}). 
   In particular,
  if $W_i=V_{c_i}$ for $c_1,\ldots,c_n\in\Cp$ then 
  \begin{displaymath}
    \Zr(\Su\times S^1_{\wb\beta})=\frac{(-1)^{n(r-1)}}r(r')^{g}q^{c\beta}\sum_{k\in \Hr}q^{ck}  
    \bp{\frac{\qn{r\beta}}{\qn{\beta+k}}}^{2g-2+n}
  \end{displaymath}
  where $c=\sum_ic_i$.
\end{theo}
\begin{proof}
The first statement of the theorem is justified directly above the statement of theorem.  So we start by proving the second statement.  

Since $b_1(\Su\times S^1)>0$, in Section 2.5 of \cite{CGP} it is shown that the
  invariants $\Nr_r(\Su\times S^1_{\wb\beta})$ depend holomorphically
  on the cohomology class.  
  Therefore, we may restrict ourselves to computing the invariants in
  the ``generic'' case, namely when all the colors 
  have non-integral degree.  We now compute $\Zr(\Su\times
  S^1_{\wb\beta})$ via the following surgery presentation:
$$
\epsh{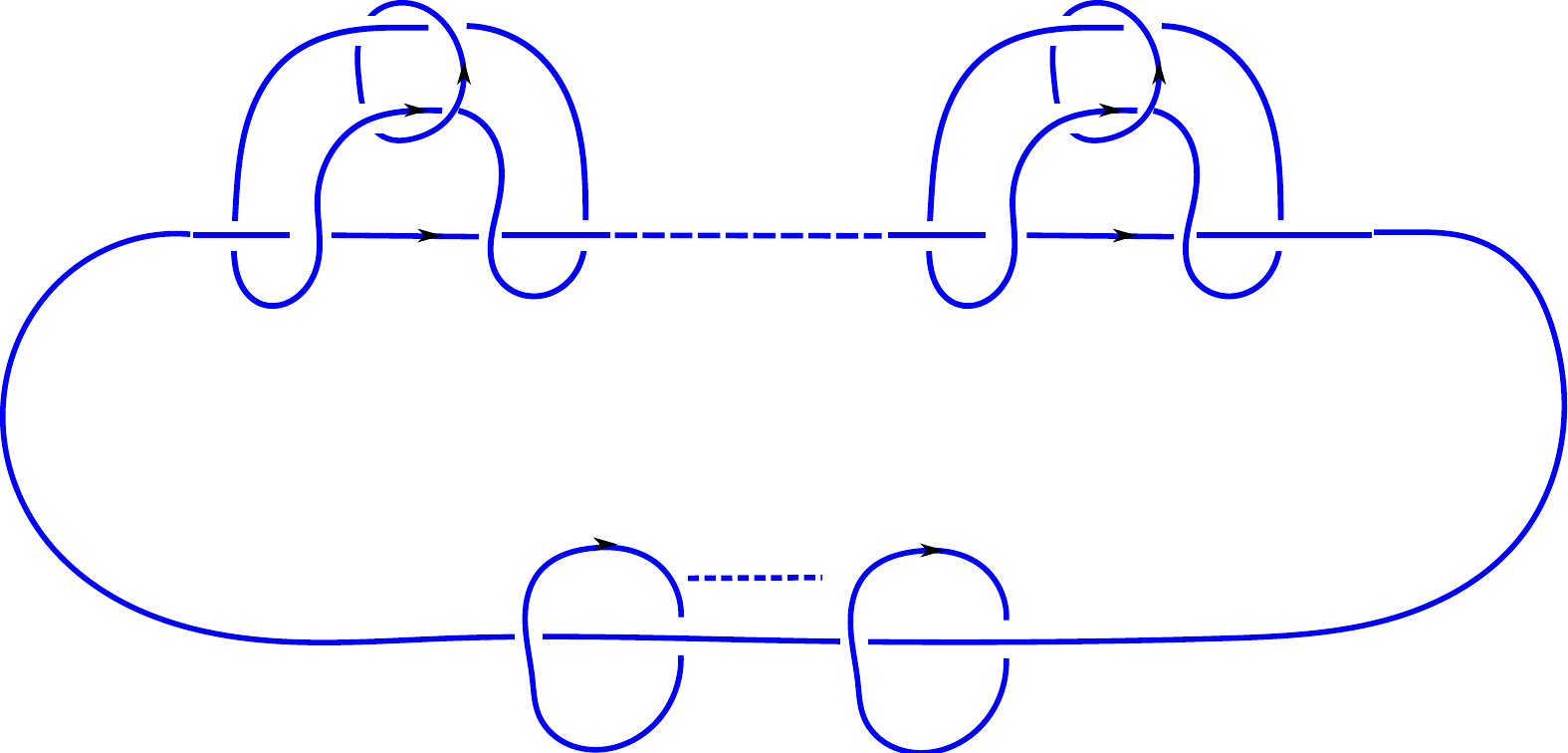}{30ex}
\put(-200,55){\ms{\wb{\beta_1}}}\put(-238,40){\ms{\wb{\alpha_1}}}\put(-80,55){\ms{\wb{\beta_g}}}
\put(-118,40){\ms{\wb{\alpha_g}}}\put(-5,-29){\ms{\wb{\beta}}}\put(-180,-25){\ms{c_1}}
\put(-120,-25){\ms{c_n}}
$$
where $c_i$ is the color of the point $p_i$ in $\Su$ and the colors $\wb{\alpha_i}$ and $\wb{\beta_i}$ (for $i=1,\ldots
g$) are the Kirby colors of degree $\coh([a_i])$ and $\coh([b_i])$ if
$[a_i],[b_i],i=1,\ldots g$ form a symplectic basis of $H_1(\Su;\Z)$.
In particular,  a component marked with 
$\wb\gamma$ is colored by $\Omega_{\gamma}=\sum_{k\in \Hr} \qd(\gamma+k)V_{\gamma+k}$ of degree $\wb\gamma$. 

Then applying Lemma \ref{lem:old71} to the $\Omega_{\alpha_i}$-colored strand and then to the $\Omega_{\beta_i}$-colored ones we see that each local sub-tangle of the above presentation is equivalent to:
$$
\epsh{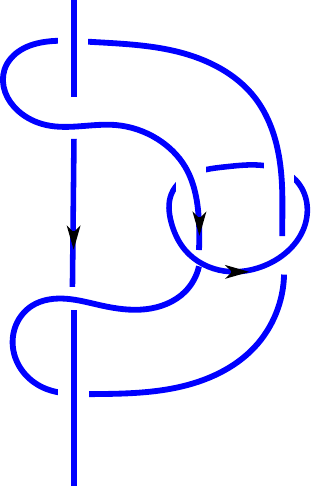}{20ex}
\put(-13,35){\ms{\wb{\alpha_i}}}\put(-15,-10){\ms{\wb{\beta_i}}}\put(-38,7){\ms{\beta'}}
=r^3\qd(\beta')^{-1}\epsh{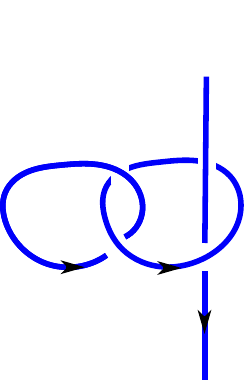}{20ex}
\put(-18,-10){\ms{\wb{\beta_i}}}\put(-18,-29){\ms{\beta'}}\put(-48,-20){\ms{\beta'}}
=r^6 \qd(\beta')^{-2}\Id_{V_{\beta'}}
$$
where $\beta'\in \Hr+\beta$ is any color appearing in  $\Omega_{\beta}$.  
Hence applying the same reduction on all the $g$ copies of the above
tangle we reduce the computation to a sum of $\beta'$-colored unknot,
encircled by $n$ meridians colored by the colors of $p_i$ as follows:
\begin{align*}
\Zr(\Su\times S^1_{\wb\beta})&=\eta\lambda^{b_1(\Su\times S^1)}
\sum_{k\in \Hr}\qd(\beta+k)^{2-2g}r^{6g}\Phi_{W,\beta+k}\\
&=\bp{\frac{\sqrt{r'}}{r^2}}^{2g+1}\frac{1}{r\sqrt{r'}}r^{6g}
\sum_{k\in \Hr}r^{2-2g}\bp{\frac{\qn{r\beta}}{\qn{\beta+k}}}^{2g-2}\Phi_{W,\beta+k}\\
&={\frac{{r'}^g}{r}}
\sum_{k\in \Hr}\bp{\frac{\qn{r\beta}}{\qn{\beta+k}}}^{2g-2}\Phi_{W,\beta+k}
\end{align*}
where 
$$\Phi_{V_{c_1}\otimes\cdots\otimes V_{c_n},V_{\beta+k}}=\qd(\beta+k)^{-n}(-1)^{n(r-1)}r^nq^{(\beta+k)\sum_{i=1}^n c_i}=(-1)^{n(r-1)}q^{(\beta+k)c}\bp{\frac{\qn{r\beta}}{\qn{\beta+k}}}^{n}$$ 
is the scalar corresponding to a strand colored by $\beta+k$ encircled
by $n$ meridians colored by $V_{c_1},\ldots V_{c_n}$.

In general, since $b_1(\Su\times S^1_{\wb\beta})>0$, as shown in \cite{CGP} (see Section 2.5) the invariant $\Nr_r(\Su\times S^1_{\wb\beta})$ depends holomorphically on $\coh$ and, when $\beta$ is generic is defined for every $\alpha_i,\beta_i,c_j$. So the above formula still applies. 
\end{proof}
\begin{rem}  \label{r:dim}
  For $r$ odd, the formula for $\Zr(\Su\times S^1_{\wb\beta})$ gives
  $\dim\VV(\Su)$ when $\beta$ goes to an integer value.  For $r$ even,
  the same is true when $\beta$ goes to an odd integer.  For example,
  for an empty surface of genus $g\ge2$,  
  applying the formula of the theorem for $W=\unit$ (so $\Phi_{W,\beta}=1$)
  $$\dim\VV(\Su)=\left\{
    \begin{array}{l}
      r^{3g-3}\text{ if $r$ is odd,}\\
      \frac{r^{3g-3}}{{2}^{g-1}}\text{ if $r$ is even.}
    \end{array}
  \right.
  $$
  and for the torus, it gives $\dim\VV(\Su)=r'$.
\end{rem}

\begin{example}\label{ex:SdXS1}[The case $\Su=\hS_{\de}$.]
 Recall the definition of $\hS_{\de}$ given in Subsection \ref{SS:SphereS_k}.    Let $M= \hS_{\de}\times S^1$.  By definition of $\hS_{\de}$, $M$
  contains a $3$-component link $L$ of the form $\{p_1,p_2,p_3\}\times S^1$
  whose components are colored by $V_0,V_0$ and $\sigma^{\de}$; let also
  $\coh\in H^1(M\setminus L;\C/2\Z)$ be any cohomology class compatible with
  the coloring of $L$ and let $\beta=\coh(*\times S^1)$.  Then 
  $\Zr(\hS_{\de}\times S^1)=(-1)^{(r-1){\de}}q^{-2r'\de\beta}$. Indeed,
  applying the $\sigma$-equivalence (see Subsection
  \ref{SS:Skein}) ${\de}$-times we may reduce to the computation to computing 
  $(-1)^{(r-1){\de}}q^{-2r'\de\beta}\Zr(\hS_0\times S^1)$.  Then  
  by Proposition \ref{P:1-surg}, $\Zr(\hS_0\times
  S^1)=(\eta\qd(0))^{-1}\eta\qd(0)=1$.  This is a particular case of the
  Verlinde formula when $\Su=\hS_{\de}$, in which case the whole
  $\VV(\Su)$ is concentrated in degree ${\de}$ and is $1$-dimensional.
\end{example}

\begin{rem}
Even if it is not apparent the above formula yields a Laurent polynomial in
$q^{\beta}$. Indeed $\frac{\qn{r\beta}}{\qn{\beta}}=\sum_{k\in \Hr} q^{\beta
  k}$.
\end{rem}
\begin{example}\label{Ex:dimTore}
  Suppose $\Su=(S^1\times S^1, \emptyset, \coh)$ is a decorated
  surface.  If $\coh$ is non-integral then $\VV(\Su)$ is concentrated
  in degree $0$ and has dimension $r'$ (i.e. $r$ if $r$ is odd and
  $\frac{r}{2}$ if $r$ is even).  Similarly (but less trivially) if
  $r$ is odd and $\Su_2$ has genus $2$: in this case one can check
  that $\VV(\Su_2)$ is concentrated in degree $0$ and that it has
  dimension $r^{3g-3}$.  If instead $g=2$ but $r$ is even then a
  direct computation shows that
  $$\dim_t(\VV(\Su_2))=(r')^3(-t^{-1}+2-t)
  \text{ so }\dim(\VV^{\pm1}(\Su_2))=(r')^3$$
\end{example}

\begin{question}
  If $\Su$ is a decorated surface which is not admissible, the
  invariant $\Zr(\Su\times S^1_{\wb\beta})$ is still well defined for
  generic values of $\beta$.  Is
  there a relation between $\Zr(\Su\times S^1_{\wb\beta})$ and
  $\dim_t(\VV(\Su))$?
\end{question}

\section{Examples and applications}\label{sec:applications}
In this section we provide various examples and applications of our
constructions.  We start with some examples: First, we consider the case when $\Su$ is a
sphere containing some points colored by at least one projective color, see  Subsection \ref{sub:generalspheres}. In Subsection
\ref{sub:torusccase} we then proceed to the case of $\Su$ being a
torus : first we compute $\VV(\Su)$ if $\coh$ is a ``generic'' (non integral) cohomology class; 
then we discuss some results in the non generic case.  Then, in
Subsection \ref{sub:genusgcase} we attack the general case of $\Su$
being a genus $g$ surface equipped with a generic cohomology class and
we provide a construction for a basis of $\VV_d(\Su)$ in each
degree $d\in \Z$.  Since in this construction $\coh$ is generic (and
in particular non-zero) the action of the mapping class group may not
fix $\coh$ so we get a representation in $\VV(\Su)$ of the stabilizer
of $\coh$ in the mapping class group rather than of the whole
group. In particular, in Subsection \ref{sub:torelli} we show that the
action of the Torelli group is more interesting than in the standard
case of the Reshetikhin-Turaev representations: indeed while in the
standard case the generators of the Torelli group act with order $r$, 
in our case they have \textit{infinite} order.
Subsection \ref{sub:kauffman} is dedicated to relating our constructions to the standard skein algebras: we show that for every $\coh\in H^1(\Su;\C/2\Z)$ the finite dimensional vector space $$\VVV(\Su)=\oplus_{h\in H^1(\Su;\Z_2)} \VV(\Su,\coh+h)$$ is acted on by the Kauffman skein algebra of $\Su$.

The above results are related with the genericity of $\coh$
but our techniques, theoretically speaking, allow a treatment also of
the non-generic cases, and in particular when $\coh=0$: in this case
we get a representation of the whole mapping class group.
Subsection \ref{sub:kashaevTQFT} is dedicated to study these cases: 
we relate the resulting TQFT as to the volume conjecture and show that 
the action of Dehn-twists along separating curves has infinite order. 
These non-generic cases are more difficult to analyze as they involve a
careful analysis of the algebra of the projective modules in $\cat_{\wb 0}$
and $\cat_1$ : we leave further results in this direction for a more algebraic work \cite{CGP3}, and here we bound ourselves
to formulate a conjecture on the graded dimensions of $\VV(S^1\times S^1)$ in the non generic cases (see Subsection \ref{sub:torusccase}).  
This analysis is much easier when $r=2$ where only two
projective modules exist (up to tensoring by powers of $\sigma$).  
Using this fact, in Subsection \ref{sub:r=2torsion}, we show that when $r=2$ the invariant $\Zr$ is related to the Reidemeister torsion, and in particular we show that it distinguishes lens spaces up to orientation preserving diffeomorphisms (Proposition \ref{prop:distinguishlens}).  
Then in Subsection \ref{sub:torus2} we provide a basis of
$\V(S^1\times S^1,\coh=0)$
for $r=2$.  Finally in Subsection \ref{sub:mcgrep} we show that
(already for $r=2$) the projective action of the mapping class group
on $\V(S^1\times S^1,\coh=0)$
is faithful modulo the center.

\subsection{General Spheres}\label{sub:generalspheres}
Let $\wbU=((U_1,\ve_1),\ldots,(U_n,\ve_n))$ be a sequence of
homogeneous modules and $S^2_{\wbU}$ be the decorated 2-sphere
introduced in Section \ref{SS:Skein}.

\begin{prop}
 If one of the $U_i$ is a
  projective object, then $\V(S_{\wbU}^2)=\Hom_\cat(\unit,F(\wbU))$.
\end{prop}
\begin{proof}
  The module $P=F(\wbU)$ is projective and Theorem
  \ref{teo:skeinsurjects}
gives is a surjective map $\Skein(B^3,S^2_{\wbU})\to \V(S_{\wbU}^2)$.    By Lemma \ref{lem:skeinhom}, $\Skein(B^3,S^2_{\wbU})$ is
  isomorphic to $\Hom_\cat(\unit,P)$.  Now using a coupon we have a map from $ \Hom_\cat(P,\unit)\to \V'(S_{\wbU}^2)$.  Finally, Proposition~\ref{P:qtnondegen} implies the pairing
  $\Hom_\cat(\unit,P)\otimes \Hom_\cat(P,\unit)\to\C$ is non-degenerate.  Combining the above facts we have   the map 
  $\Skein(B^3,S^2_{\wbU})\to\V(S_{\wbU}^2)$ is injective.
\end{proof}

\begin{cor} Recall $\Hr=\{-(r-1),-(r-3)\ldots, r-1 \}$.
  \begin{enumerate}
  \item 
  Let $\alpha,\beta\in\Cp$,    $\gamma\in \C\setminus \Z$ and $n\in\Z$.  Then $\V(S^2_{((V_{\alpha},+),(V_{\beta},+),(V_{\gamma},+),(\sigma^n,+)))})$ zero dimensional 
  if $\alpha+\beta+\gamma+2nr'  \notin   \Hr$ and is one dimensional otherwise.  
  \item If $i\in \{0,1,\ldots r-2\}$ and $\ n\in \Z$, then
    $\V(S^2_{((V_{0},+),(S_i,+),(\sigma^n,+))})=\{0\}$.
  \item Let $i,j\in \{0,1,\ldots r-2\},\ n\in \Z$.  Then
    $\V(S^2_{((P_{j},+),(S_i,+),(\sigma^n,+)})$ is zero dimensional if $i\neq j$ or
    $n\neq 0$ and is one dimensional if $i=j$ and $n=0$.
  \end{enumerate}
\end{cor}
\begin{proof}
  For the first item, we use $\Hom_\cat(\unit, V_\alpha\otimes
  V_\beta\otimes V_\gamma\otimes
  \sigma^n)\simeq\Hom_\cat(V_{-\gamma-2nr'}, V_\alpha\otimes V_\beta)$
  which is $0$ unless $k\in2\Z$.  If $k\in2\Z$, $V_\alpha\otimes
  V_\beta\simeq\bigoplus_{h\in\Hr}V_{\alpha+\beta+h}$ so $\V(S)$ is
  non zero if and only if there is an $h\in\Hr$ with
  $\alpha+\beta+h=-\gamma-2nr'$.  Similarly for the second item, $\Hom_\cat(\unit,V_{0}\otimes
  S_i\otimes\sigma^n)\simeq\Hom_\cat(V_{-2r'n},S_i)=\{0\}$.  Finally, $\Hom_\cat(\unit,P_{j}\otimes
  S_i\otimes\sigma^n)\simeq\Hom_\cat(P_{j},S_i\otimes\sigma^n)$ where
  $S_i\otimes\sigma^n$ is simple.  But the only simple quotient of
  $P_j$ is isomorphic to $S_j$.
\end{proof}

\subsection{The torus}\label{sub:torusccase} 
Let $\Su=S^1\times S^1$ with a non-integral cohomology class $\omega$.
Up to diffeomorphism, we can assume that $\omega$ takes a value
$\wb\alpha\in\C/2\Z$ for some $\alpha\in\C\setminus \Z$ on the meridian
$[S^1\times\{1\}]$.  Let $T_\alpha\simeq B^2\times S^1$ be the solid
torus whose core is the oriented framed knot $\{0\}\times S^1$ colored by $\alpha$.
\begin{prop}\label{P:torusccase} Let $\Hr=\{1-r,3-r,\ldots,r-1\}$ and for $r$
  even, let $\Hr^+=\Hr\cap\R^+$.
  \begin{enumerate}
  \item If $r$ is odd then $([T_{\alpha+k}])_{k\in\Hr}$ is a basis of $\V(\Su)=\VV_0(\Su)=\VV(\Su)$.
  \item If $r$ is even then $([T_{\alpha+k}])_{k\in\Hr^+}$ is a basis of $\V(\Su)=\VV_0(\Su)=\VV(\Su)$.
  \end{enumerate}
\end{prop}
\begin{proof}
  Let $\epsilon=0$ if $r$ is odd and $\epsilon=1$ if $r$ is even.
  By Theorem \ref{teo:skeinsurjects}, $\V(\Su)$ is generated by
  cobordisms whose underlying manifold is $B^2\times S^1$. 
 Moreover, every $\cat$-colored graph $\Gamma\subset T$ must
  intersect the meridian disc in a set of points $q_i$ with colors
  whose overall tensor product is in $\cat_{\wb\alpha}$.  Since $\alpha$ is
  generic, this cobordism is then skein equivalent to a linear
  combinations of cobordisms whose graph intersects the meridian disc
  in exactly one point whose color is 
  congruent to $\alpha+\epsilon$ modulo $2$ and then,
  by simplicity of $V_\alpha$, the set $\{T_{\alpha+\epsilon+k},k\in 2\Z\}$ is
  generating.  Moreover $[T_{\alpha+k+2r'n}]$ and $[T_{\alpha+k}]$ are
  $\sigma$-equivalent because $V_{\alpha+k+2r'n}=\sigma^n\otimes
  V_{\alpha+k}$.

  We are then left to prove that the vectors $[T_{\alpha+k}],\ k\in
  \Hr$ (respectively $k\in\Hr^+$) are independent.  To this purpose we
  pair them with the solid tori $[T'_{-\alpha+h}]\in\V'(\Su)$ obtained
  by reversing the orientation of $T_{\alpha-h}$ for
 $h\in \Hr$, and we will show that the pairing
  is non-degenerate, thus concluding.  Observe that there exists a
  simple oriented closed curve $\lambda\subset \Su$ intersecting the
  meridian $[S^1\times \{1\}]$ exactly once and such that
  $\coh([\lambda])=\beta$ is non-integral: indeed the values of $\coh$
  on the classes $[\{1\}\times S^1]$ and $[\{e^{it}\times e^{it}, t\in
  [0,2\pi]\}]$ cannot be both integral as $\coh([S^1\times
  \{1\}]=\overline{\alpha}\in (\C\setminus \Z)/\Z$.  Then the
  invariant $\Zr(T'_{-\alpha+h}\circ T_{\alpha+k})$ can be computed by
  a $0$-framed unknot in $S^3$ colored by a Kirby color of degree
  $\beta\in (\C\setminus \Z)/2\Z$ and two parallel unknots linked once
  to it with colors $\alpha+k$ and $-\alpha+h$ respectively.  By Lemma
  \ref{lem:old71} the value of the invariant 
  is non zero if and only if $k+h=2r'n$ for some integer $n$. But
  since $k,h\in \Hr$ this is equivalent to $h=-k$ if $r$ is odd and
  $h=r-k$ if $r$ is even.
  
  This shows that the pairing is non degenerate as the map $k\to
  -k$ (respectively $k\to r-k$) is a bijection of $\Hr$ (respectively
  of $\Hr^+$).
  
  Finally the vectors we obtained form a basis of $\V(\Su)$ whose
  cardinality equals the dimension of $\VV(\Su)$ provided by the
  Verlinde formula 
  (see Example \ref{Ex:dimTore}),  
  so $\V(\Su)=\VV(\Su)$.
\end{proof}

The above calculations were supposing the existence of a non-integral period of $\coh$. 
 In \cite{CGP3}, using algebraic arguments (i.e. computing the Hochschild homology of the algebra of endomorphisms of the projective modules in each category $\cat_\alpha,\alpha\in \C/2\Z$) we computed the dimensions of  $\Skein(B^2\times S^1)$ for any solid torus bounding $\Su$
thus obtaining an upper estimate for $\dim(\VV(\Su))$, but we actually observed that when $\coh$ is non-integral this estimate coincides with the 
dimension of $\VV(\Su)$, even in the case when $\coh([S^1\times \{1\}])\in \Z$ (to which the above lemma does not apply).
So it is natural to expect that also for non-generic $\coh$ the skein module of the solid torus is isomorphic 
to $\VV(\Su)$ for any solid torus bounding $\Su$ and in particular, using the computations in \cite{CGP3}, we conjecture the following :
\begin{conj}\label{conj:dimintegraltorus} Let $\Su=S^1\times S^1$ and $\coh\in H^1(S^1\times S^1;\Z/2\Z)$ then :
 \begin{enumerate}
  \item If $r\in2\Z+1$ it holds $ \dim_s\left(\VV(S^1\times S^1)\right)=\frac{3r-1}2$;
  \item If $r\in4\Z+2$ and $\coh\neq 0$ then $\dim_s\left(\VV(S^1\times S^1)\right)=\frac{3r-2}{4}$;
  \item If $r\in4\Z+2$ and $\coh=0$ then $\dim_s\left(\VV(\Su)\right)=s^{-1}+\frac{3r+2}4+s$,
  \end{enumerate}
  where $\dim_s$ is the sum for $k\in\Z$ of $s^k$ times the dimension
  of the degree $k$ subspace.
 
\end{conj}

\subsection{Basis for generic empty surfaces}\label{sub:genusgcase}
Let $\Su$ be a decorated surface of genus $g$ with $\{p_i\}=\emptyset$.
Let $H$ be an oriented handlebody bounding $\Su$ and let $\Gamma$ be an
oriented trivalent ribbon graph which is a spine of $H$ (i.e. $H$
collapses on $\Gamma$).  Let $E_\Gamma$ be the set of the $n_g$
oriented edges of $\Gamma$, where
$$n_g=\left\{
  \begin{array}[c]{l}
    1\text{ if }g=1,\\
    3g-3\text{ if }g\ge2.
  \end{array}\right.$$
Let $\coh\in H^1(\Su;\C/2\Z)\simeq H^1(H\setminus\Gamma;\C/2\Z)$ be
the cohomology class of $\Su$ and $g_\omega$ the induced $\C/2\Z$-coloring of the edges of $\Gamma$.  We say that $\Gamma$ is
\emph{$\omega$-regular}  if for any $e\in E_\Gamma$,
we have $g_\omega(e)\notin\Z/2\Z$.
\begin{prop}\label{prop:genericspine}
  There exists a handlebody $H$ bounding $\Su$ which contains a
  $\omega$-regular spine $\Gamma$ if and only if $2\coh$ is non-integral.
\end{prop}
\begin{proof}
  We will prove the statement for $g>1$ (the case of the torus is
  easier and left to the reader). 
  First, assume $2\coh$ is non-integral.   Let $m_1$ be a simple curve in
  $\Su$ such that $\coh(m_1)\not\in \frac1{2}{\Z}$. Let $H$ be any
  handlebody bounding $\Su$ (i.e. we have a diffeomorphism
  $f^+:\Su\stackrel\sim\to\partial_+H=\partial H$) such that $m_1$
  bounds a disc in $H$ and consider a spine of $H$ which is
  combinatorially equivalent to the oriented trivalent graph $\Gamma$
  obtained by gluing $g-1$ copies of the following graph
  $$\epsh{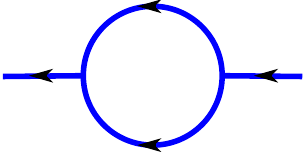}{10ex}\put(-10,6){$e_i$}\put(-85,6){$e_{i+1}$}
  \put(-50,12){$e'_i$}\put(-50,-12){$e''_i$} $$
  and finally closing up by identifying $e_g$ with $e_1$. (It is easy
  to check that it exists as soon as $g>1$.)  Up to
  self-diffeomorphisms of $H$ we may suppose that the meridian of $e_1$
  is $m_1$.  The meridians $m_1,\ldots,m_{g-1}$ of
  $e_1,\ldots,e_{g-1}$ represent the same element $x_0\in
  H_1(\Su;\Z)$.  Let $x_1,\ldots,x_{g-1}$ be the homology classes in
  $H_1(\Su;\Z)$ of the meridians $m'_1,\ldots ,m'_{g-1}$ of
  $e'_1,\ldots e'_{g-1}$, respectively.  Then $(x_0,\ldots,x_{g-1})$ is an integral
  basis of the Lagrangian $\La_H$ which is the kernel of the map
  $H_1(\Su;\Z)\to H_1(H;\Z)$ induced by $f^+$.  The values of $g_\coh$
  on $E_\Gamma$ are given by $g_\coh(e_i)=\coh(x_0)$,
  $g_\coh(e'_i)=\coh(x_i)$ and $g_\coh(e''_i)=\coh(x_0-x_{i})$.  Thus
  $\Gamma$ is $\coh$-regular  if and only if
  \begin{equation}\label{eq:reg}
    \coh(x_0),\ \coh(x_i),\ \coh(x_0-x_{i})\notin\Z/2\Z 
    \text{ for }i=1\cdots g-1.  
  \end{equation}
  By construction we have that $\coh(x_0)\notin (\frac12\Z)/2\Z$.  We now construct a new basis $x_0=y_0,y_1,\ldots y_{g-1}$ of $\La_H$ such that $\coh(y_i)\notin \Z/2\Z$ and $\coh(y_i-y_0)\notin\Z/2\Z$ arguing for each $i=1,\ldots g-1$ as follows:
  \begin{itemize}
  \item If $\coh(x_i)\in\Z/2\Z$, then for $y_i=x_i-x_0$ we have
    $\coh(y_i)\notin\Z/2\Z$ and
    $\coh(x_0-y_i)=2\coh(x_0)+\coh(x_i)\notin\Z/2\Z$, so replace $x_i$ with $y_i$.
  \item If $\coh(x_0-x_i)\in\Z/2\Z$, then for $y_i=x_i+x_0$ we have
    $\coh(y_i)=2\coh(x_0)-\coh(x_0-x_i)\notin\Z/2\Z$ and
    $\coh(x_0-y_i)=-\coh(x_0)+\coh(x_0-x_i)\notin\Z/2\Z$, so replace $x_i$ with $y_i$.
  \item Else, we just take $y_i=x_i$ (observe that since $\coh(x_0)\notin \Z/2\Z$ the preceding two cases cannot happen simultaneously).
  \end{itemize}
  Hence one can construct another $\Z$-basis $(y_i)_{i=0\cdots
    g-1}$ of $\La_H$ for which \eqref{eq:reg} holds.  Now the mapping
  class group of $\Su$ surjects on the symplectic group
  $SP(H_1(\Su;\Z))$ which acts transitively on the $\Z$-basis of any
  Lagrangian (because $GL_g(\Z)\subset SP_{2g}(\Z)$).  Hence there
  exists a diffeomorphism $h$ of $\Su$ which sends in homology the
  basis $(y_i)_{i=0\cdots g-1}$ on $(x_i)_{i=0\cdots g-1}$.  We
  reparametrize the boundary of $H$ using $f^+\circ
  h:\Su\stackrel\sim\to\partial H$, then the value of ${(f^+\circ
    h)^{-1}}^*(\coh)$ on the meridian of the edges in $E_\gamma$ are
  given by $\coh(y_i),\coh(y_0-y_i)$ which are not integral and
  $\Gamma$ is $\coh$-regular.

  Conversely, if $\coh$ has only half integral values, then
  for any trivalent graph $\Gamma$, let $v$ be any vertex of $\Gamma$
  ($g\ge2$ so $\Gamma$ has at least $2$ vertices).  Then the sum of
  the values of $g_\coh$ on the $3$ edges adjacent to $v$ is
  $0\in\C/2\Z$ so the three values can't be half integers and
  $g_\coh$ has at least one value in $\Z/2\Z$.  Hence $\Gamma$ is
  not $\coh$-regular.
\end{proof}
Let $H$ be as above and $\Gamma\subset H$ be an oriented spine
which is 
 $\coh$-regular.
Fix a lift $\wt\coh\in
H^1(\Su;\C)$ of $\coh\in H^1(\Su;\C/2\Z)$ and a maximal sub-tree
$T\subset \Gamma$
(for the graph $\Gamma$ in the above proof, one can take
$T=\{e''_i\}_{i=1\cdots g-1}\cup\{e_i\}_{i=2\cdots g-1}$).  
We now describe a basis of $\VV(\Su)$ associated to $(T,\wt\coh)$ by
means of colorings of a modified version of $\Gamma$.

Let $e_1,\ldots e _g$ be the edges of $\Gamma\setminus T$ and
$e_{g+1},\ldots ,e_{n_g}$ those of $T$.  Let $\Gamma'$ be the graph
obtained by adding three edges $e^{\pm 1,0}_{-1}$ to $\Gamma$ whose
endpoints are points $p^{\pm 1,0}_+\in H\setminus \Gamma$ and points
$p^{\pm 1,0}_-\in e_1$, so that $e_1$ is split by $p^{\pm 1,0}_-$ into
four sub-edges which we will call $e_0,e'_0,e'_1, e_1$:
$$\epsh{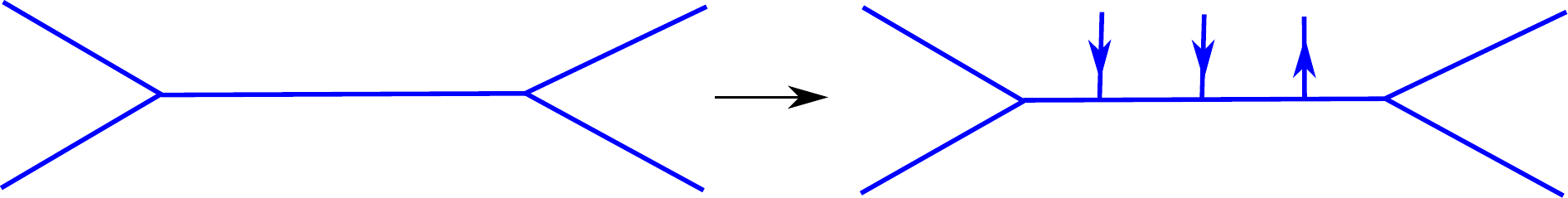}{10ex}\put(-270,13){$\Gamma$}\put(-30,13){$\Gamma'$}\put(-95,-05){$e'_{0}$}\put(-70,-05){$e'_{1}$}\put(-110,-05){$e_{0}$}\put(-50,-05){$e_{1}$}\put(-95,10){$e^{0}_{-1}$}\put(-73,10){$e^{1}_{-1}$}\put(-118,10){$e^{-1}_{-1}$}
$$
Here the orientation of the edges $e_0,e'_0,e'_1, e_1$ is induced by the orientation of the original edge $e_1$.  
Note the endpoints $p^{\pm 1,0}_+$ in $\Gamma'$ are univalent vertices.  
\begin{defi}[Degree $d$ colorings on $\Gamma'$]  Here we use the notation of the previous paragraph.  
  Let $\epsilon=0$ if $r$ is odd and $1$ otherwise.  A degree $d\in \Z$
  coloring on $\Gamma'$ is a map:
  $$\col:E_{\Gamma'}\to \cat$$
  assigning each edge $e_i$, for $i=1,...,n_g$, 
  the module $V_{\alpha_i}$ where
$\alpha_i\in\Cp$ is of degree $\coh(m_i)$ and
   $$\col(e_{-1}^{1})=V_0=\col(e_{-1}^{-1}),\ \
  \col(e_{-1}^0)=\sigma^{d}$$
   such
  that 
 $$\col(e_0)=\col(e_1)+2r'd,\ \ \col(e'_0)=\col(e_0)+\epsilon,\ \ \col(e'_1)=\col(e_1)+\epsilon,$$
  and
  at each trivalent
  vertex of $\Gamma'$ 
  we require that the algebraic sum of the colors
  of the three adjacent edges is in $H_r$.  Here and in what
  follows we write $\col(e)=\alpha\in\Cp$ to mean  $\col(e)=V_\alpha$.
\end{defi}
Recall that $\Hr^+=\Hr\cap\R^+$ if $r$ even and let $\Hr^+=\Hr$ if $r$
odd.  For each $d\in \Z$ let ${\mathcal B}:=\bigcup_{d\in \Z}
{\mathcal B}_d$ where
$${\mathcal B}_d=\{{\rm degree}\ d\ {\rm colorings}|\ 
\col(e_i)\in \wt\coh(m_i)+\Hr^+,\  \forall i=1,\ldots , g\}.$$
Observe that to each element $b\in \mathcal B_d$ we may associate a
vector of $\V(\Su\sqcup \hS_d)$ as follows: let $B^3\subset
\mathring{H}$ be a small ball disjoint from $\Gamma$.  Embed $\Gamma'$
in $H\setminus \mathring{B}^3$ by fixing $p^{\pm 1,0}_+$ to belong to
$\partial B^3$ and color its edges by $b$. Equip $H\setminus (B^3\cup
\Gamma')$ with the cohomology class induced by $\coh$ on the boundary
(it exists and it is unique).
Then let $v_b:=[H\setminus \mathring{B}^3,\Gamma'_b]\in \V(\Su\sqcup \hS_d)=\VV_{-d}(\Su)$.

\begin{prop}\label{prop:genericbasis}
  Let $\Su$ be a connected decorated surface of genus $g$ with
  $\{p_i\}=\emptyset$ and suppose that the class $\coh\in H^1(\Su;\C/2\Z)$ is
  such that $2\coh$ is non-integral.  
  Then for each $d\in \Z$, the set $\{v_b:b\in \mathcal B_d\}$ described above
  forms a basis of $\VV_{-d}(\Su)$.  
  In particular, the set $\{v_b:b\in \mathcal B\}$ is a basis for $\VV(\Su)$.
\end{prop}
\begin{proof}
 Since $2\coh$ is non-integral then Proposition \ref{prop:genericspine}
 implies that there exists a handlebody $H$ containing an oriented spine $\Gamma$
which is 
 $\coh$-regular.  Then let $\Gamma'$ and  $\col$ be as above.  

    First, we show that $\card(\B)=r^{2g-2}{r'}^{g-1}=\dim\VV(\Su)$ where
  the last equality comes from Remark \ref{r:dim}.  Indeed, to
  construct an element of $\B$, we first fix the colors of
  $(\alpha_i)_{i=1\cdots g}$ of $(e_i)_{i=1\cdots g}$.  This choice is
  given by an element of $(\Hr^+)^g$.  Then we see $T\cup
  e_0\subset\Gamma'$ as a rooted tree with root $p_-^{-1}$ and we
  color its edges recursively (starting from the leafs).  Each edge of
  $T\cup e_0$ can be colored by the $r$ colors that appear in the
  decomposition when tensoring the two already colored adjacent edges.
  Finally, let's call $\col(e_0)=\alpha_0+h$ $(h\in\Hr)$ the $r$
  possible choices for $\col(e_0)$.  The compatibility with $\coh$
  ensures that all $\alpha_0+h$ are congruent to $\col(e_1)$ modulo
  $2$. 
  So we must have $\col(e_0)=\col(e_1)+2r'd$ for some $d\in \Z$.  
  In other words, we want $\alpha_0+h-\col(e_1)\in 2r'\Z$. 
  There is exactly one $h$ satisfying this if $r$ is odd and two if $r$ is even.
   Only these $\frac{r}{r'}$ ($=1$
  or $2$) colors for $e_0$ lead to coloring in $\B$.  
      As $T$ contains $2g-3$ edges, the cardinal of $\B$ is
  ${r'}^gr^{2g-3}\times \frac{r}{r'}$ as announced.

We now need to prove that the above vectors are linearly independent. To this purpose, we pair them with the elements $[H,\Gamma'_{b}]$ of $\V'(\Su\sqcup \hS_{d})$ obtained by reversing the orientation of $H$ and taking the colorings opposite to those of $\mathcal{B}_d$.   
We have:
$$\langle [H,\Gamma'_{b_1}],[H,\Gamma'_{b_2}]\rangle=\Zr(M)$$
where $M$ is a connected sum of copies of $S^2\times S^1$ (one per edge of $\Gamma\setminus T$ and one more for the handle whose essential sphere is $\hS_d$) and contains the ribbon graph $D\Gamma'$ obtained by gluing the two copies of $\Gamma'$ along $p^{\pm1,0}_+$.
 A surgery presentation for $M$ is obtained by considering two parallel copies of $\Gamma$ in $S^3$, encircling all the corresponding edges out of $T$ by $0$-framed meridians, joining the two copies of $e_1$ by three segments (representing the doubles of $e^{\pm1,0}_{-1}$) and finally by encircling these three segments (colored by $V_0,V_0$ and $\sigma^d$) by a $0$-framed meridian (see the left part of the Figure \ref{fig:pairing}).
 \begin{figure}
$ \epsh{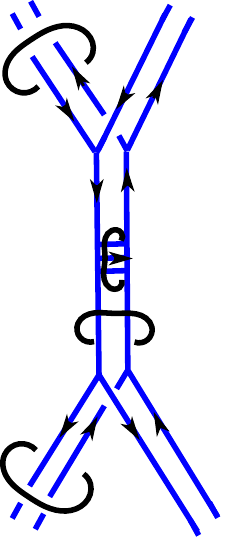}{30ex}\to \epsh{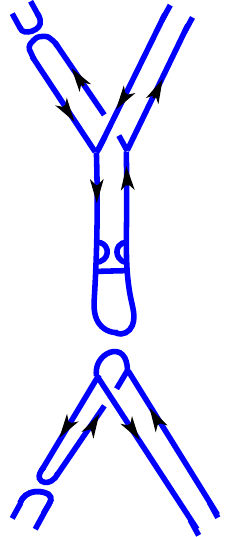}{30ex}\to \epsh{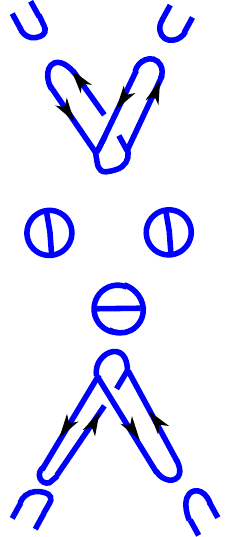}{30ex}$
\caption{Pairing $\Gamma'_{b_1}$ with $\Gamma'_{b_2}$. In the drawing the $0$-framed meridians on which surgery is to be performed are in black: they encircle the edges of $T$. The three horizontal segments in the middle of the drawing are colored by $V_0,V_0$ and $\sigma^{d}$. To pass from the leftmost to the central drawing we applied Lemma \ref{lem:old71} to all the zero framed meridians, thus reducing to a framed trivalent graph in $S^3$ without surgered meridians. In the last drawing we show how to split this graph as a connected sum of $\theta$-shaped graphs. }\label{fig:pairing}
\end{figure}
 Applying Lemma \ref{lem:old71} to all the pair of edges encircled by a $0$-framed meridian as exemplified in Figure \ref{fig:pairing}, and using the fact that $b_1\in \mathcal{B}_d$ and $-b_2\in \mathcal{B}_{d}$,  one sees that if $b_2\neq -b_1$ then $\Zr(M)=0$ and else it is non-zero because, besides some powers of $r$, 
 $\Zr(M)$ is equal to the invariant of a graph in $S^3$ which is a connected sum of various $\theta$-shaped graphs (one per each vertex of $\Gamma$ and two more for the graph resulting from the part of $\Gamma'$ surrounding $e_0$ and $e_1$). The invariant of such graph is easily checked to be non-zero.
 This proves that the vectors $v_b$ are independent as they may be distinguished by the co-vectors $v_{-b}\in \V'(\Su\sqcup \hS_d),\ b\in \mathcal B_d$.
\end{proof}
\begin{example}\label{ex:torusgeneric}
  Suppose $g=1$ and $m\subset \Su$ is an oriented simple closed curve 
  such that $\coh(m)=\wb\coh_0$ for some $\coh_0\in \C\setminus\Z$;
  let $\Gamma$ be an oriented circle representing the spine of a
  handlebody whose meridian is $m$. Then $T$ is the empty graph and
  the basis $\{v_b\}$ described in Proposition \ref{prop:genericbasis}
  is given by all the colorings of $\Gamma$ by complex numbers
  $\alpha\in\coh_0+\Hr^+$.
  In particular in this case all
  the colorings have degree $0$ and $\dim(\VV(\Su))=r'$.
\end{example}

\subsection{Finite dimensional representations of the Torelli groups}\label{sub:torelli}
Let $\Su$ be a surface of genus $g>1$.  An element $\coh\in
H^1(\Su;\C/2\Z)$ is called an {\em irrational class} if $\coh(x)\notin
\Z/2\Z,\,\forall x\in H_1(\Su;\Z)\setminus\{0\}$.  Elements of
$H^1(\Su;\C/2\Z)$ are generically irrational.  If $\coh$ is an
irrational class, then one easily check that the values of $\coh$ on
$H_1(\Su;\Z)\setminus\{0\}$ are contained in $\C/2\Z\setminus\Q/2\Z$.
Let us equip $\Su$ with an irrational cohomology class.

Let $\T_g$ be the Torelli group of $\Su$.
In this subsection we show that the action of $\T_g$ on $\VV(\Su)$ is
potentially much more interesting than the corresponding action in the
standard TQFT associated to Reshetikhin-Turaev invariants: indeed the
generators of $\T_g$ given by Dehn twists along separating curves and
along pairs of homologous simple closed curves act with infinite order
on $\VV(\Su)$ while those of the standard Reshetikhin-Turaev TQFTs have order $r$.

If $c\subset \Su$ is a simple closed curve, let $\DT_c$ be the
right-handed Dehn twist along $c$.  We recall that as proved by Birman
and Powell $\T_g$ is generated by the following infinite set of
diffeomorphisms of $\Su$:
\begin{itemize}
\item $\{\DT_c| c \ {\rm separating,\ non-trivial,\  simple\  closed\ curve}\}$;
\item $\{\DT_{c}^{-1}\circ \DT_{c'}|c,c' \ {\rm disjoint,\ non-separating,\ homologous\ simple\ oriented\ curves}\}$.
\end{itemize}
Furthermore, the above result was improved for $g\geq 3$ by Johnson
who proved \cite{Jo2} that actually a finite set of diffeomorphisms
of the second type suffice to generate the whole group. 
On contrast,
as proved by Mess \cite{Me}, $\T_2$ is an infinite rank free
group. Finally let us recall that Johnson \cite{Jo} also proved that there exists a finite subset of the
set: $$\{\DT_{c}^{-1}\circ \DT_{c'}|c,c' \ {\rm disjoint,\ non-separating,\
  -c\sqcup c' \ bound \ a\ twice-punctured\ torus\ in\ \Su}\}$$
which generates the whole $\T_g$ if $g\geq 3$ .  The main proposition of
this subsection is the following:
\begin{theo}\label{T:torelli} 
  Let $r>2$ and let $\Su$ be a genus $g\ge3$ surface equipped with an
  irrational cohomology class $\coh$.  Then for for every pair of
  disjoint, non-trivial,
  non-separating, 
  homologous simple closed oriented curves $c,c'$ the
  order of the action of $\DT_{c}^{-1}\circ \DT_{c'}$ on $\VV(\Su)$ 
  is infinite.
\end{theo}
\begin{proof}
  Let $\Gamma$ be a 
  $\coh$-regular spine for an oriented handlebody $H$
  bounding $\Su$ such that $c$ and $c'$ are meridians of edges $e$ and
  $e'$ of $\Gamma$ and $T\subset \Gamma$ be a maximal sub-tree. Let
  also 
  $\coh([c])=\coh([c'])=\overline{\alpha} \in (\C\setminus \Q)/2\Z$ 
  (the curve $c$ is not separating so $[c]\neq0$).
  The effect of the combination of Dehn-twists $\DT_{c'}^{-1}\circ \DT_c$ is to
  change $\Gamma\subset H$ into a new graph $\Gamma'$ which is
  identical to $\Gamma$ except for its framing which is changed by
  $+1$ on $e$ and by $-1$ on $e'$.  Let us color the edges of
  $\Gamma'$ and $\Gamma$ by a coloring corresponding to a vector $v_b$
  of the basis $\{v_b\}$ of $\VV(\Su)$ constructed in Proposition
  \ref{prop:genericbasis}, and let $\alpha+h$ and $-\alpha+k$ be the
  colors respectively of $e$ and $e'$. 
  We assume that $v_b$ is such that $h\neq k$ (this is possible as $r>2$). 
  Restoring the framing of
  $\Gamma'$ to that of $\Gamma$ changes the vector $[H,\Gamma']\in
  \V(\Su)$ by multiplication by the twists
  $q^{\frac{(\alpha+h)^2-(r-1)^2}{2}}$ (for $e$) and
  $q^{-\frac{(\alpha+k)^2-(r-1)^2}{2}}$ (for $e'$).  Then it
  holds: $$\DT_{c'}^{-1}\circ \DT_c
  (v_b)=q^{\frac{(\alpha+h)^2-(\alpha+k)^2}{2}}v_b=q^{(h-k)(\alpha+\frac{(h+k)}{2})}v_b.$$
  By hypothesis on $\coh$, we have $\alpha$ is irrational so the order
  of the action of $\DT_{c'}^{-1}\circ \DT_c$ on $v_b$ is
  infinite. (Remark also that the action of $\T_g$ is non-projective
  as the framing anomaly intervenes only through the Maslov index
  which is always $0$ for Torelli elements.)
\end{proof}

\subsection{Representation of the Kauffman skein algebra}\label{sub:kauffman}
\newcommand{\KS}{\mathcal{SK}} 
The Kauffman skein module $\KS^A(M)$ of an oriented manifold $M$ (see
\cite{Tu1,Pr}) is the quotient of the free $\Q[A^{\pm1}]$-module
generated by isotopy classes of unoriented framed links in $M$ modulo
the Kauffman skein relations~:
$$\epsh{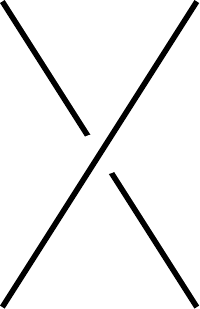}{6ex}=A\epsh{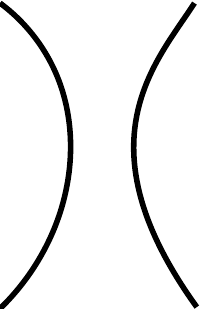}{6ex}+A^{-1}\epsh{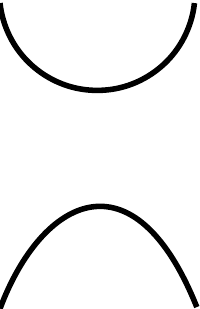}{6ex}\et 
L\sqcup\epsh{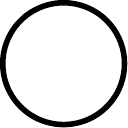}{4ex}=-(A^2+A^{-2})L.$$ Here $A$ is either a
formal variable or a non zero complex number.  If $\Su$ is an oriented
surface and $M=\Su\times[0,1]$, $\KS^A(\Su)=\KS^A(M)$ is an algebra
with product induced by the composition of the cobordism
$M=\Su\times[0,1]:\Su\to\Su$ with itself ($M\circ M=M$).

Consider the free group $$G=H^1(\Su;\Z/2\Z)$$ and let
$\pi:\Su\times[0,1]\to\Su$ be the projection.  Using the intersection
form $\iota:H_1(\Su;\Z/2\Z)\otimes H_1(\Su;\Z/2\Z)\to\Z/2\Z$, $G$ can
be identified with $H_1(\Su;\Z/2\Z)$.  Then there is a natural
$G$-grading on $\KS^A(\Su)=\oplus_\coh\KS^A_\coh(\Su)$: a framed link
$L\subset\Su\times[0,1]$ represent a homogeneous element of degree
$\coh_L=\iota(\pi_*([L]),.)\in G$ where $[L]$ is the class of $L$ in
$H_1(\Su\times[0,1],\Z/2\Z)$.  One easily check that the skein
relations are homogeneous and that
$\KS^A_\coh(\Su).\KS^A_{\coh'}(\Su)\subset\KS^A_{\coh+\coh'}(\Su)$.

Let $\ve=q^{r^2/2}$ which is a fourth root of unity ($\ve=-1$ if $r$
is even).  Then $\KS^{\ve}(\Su)$ is well-known, is related to the
character variety of $\Su$ and embeds naturally into the center of
$\KS^{q^{1/2}}(\Su)$ (see \cite{BW,Le}).
We now show that $\KS^{q^{1/2}}(\Su)$ has a natural
representation in some TQFT vector spaces associated to $\Su$.

We first fix a base point, a Lagrangian $\La$ and a cohomology class
$\coh_0$ for $\Su$ and consider
$$\VVV(\Su)=\bigoplus_{\coh\in G}\VV(\Su+\coh)\text{ where }\Su+\coh
=(\Su,\emptyset,\coh_0+\coh,\La)$$
which only depends of the class of $\coh_0$ modulo $\Z$.

Let $M=(\Su\times[0,1],T,\Id\times0\sqcup\Id\times1,\coh,0):
(\Su,\emptyset,\coh_0,\La) \to(\Su,\emptyset,\coh_1,\La)$ be a decorated
manifold where $T$ is an oriented framed link colored by the simple module
$S_1$.  Let $L_T$ be the underlying framed unoriented link in
$\Su\times[0,1]$.  Then we shall say that $M$ is ``a lift of $L_T$''.
\begin{prop}
  Let $\Su=(\Su,\emptyset,\coh_0,\La)$ be a decorated surface and $L$
  be an unoriented framed link in $\Su\times[0,1]$.  Then 
  \begin{enumerate}
  \item $L$ has a lift
    $M_L:(\Su,\emptyset,\coh_0,\La)\to(\Su,\emptyset,\coh_1,\La)$.
  \item for a lift $M_L$ of $L$ as above, $\coh_1=\coh_0+\coh_L$. 
  \end{enumerate}
\end{prop}
\begin{proof}
  Let denote by $m_1,\ldots,m_n$ the meridians of the $n$ components link $L$.
  Recall that $S_1\in\cat_{\wb1}$ thus a cohomology class for a lift of $L$ is
  compatible if and only if its values on each $m_j$ is $\wb1$.  Let
  $X=\Su\times[0,1]$.  It is an exercise to show that $H_1(X\setminus
  L)\stackrel{g_*}\to H_1(X)$ is surjective with two preferred sections
  ${s_i}_*\circ\pi_*$ where $s_i:\Su\to\Su\times\{i\}\subset X\setminus L$ and
  $\pi:X\to\Su$.  The long exact sequence can then be used to show that $\ker
  g_*\simeq\Z^n$ is generated by the meridians $[m_j]$.  Hence there is a
  unique cohomology class in $\coh_L^X\in H^1(X\setminus L;\Z/2\Z)$ which
  vanish on $(s_0)_*(H_1(\Su))$ and whose value is $1$ on each meridian
  $[m_j]$.  Let $\coh'_L=s_1^*(\coh_L^X)$, then the unique cohomology class in
  $H^1(X\setminus L;\Z/2\Z)$ which restricts to $\coh_0$ on $\Su\times\{0\}$ and
  which is $1$ on each $m_j$ is $\pi^*(\coh_0)+\coh_L^X$.  Its restriction on
  $\Su\times1$ is $s_1^*(\pi^*(\coh_0)+\coh_L^X)=\coh_0+\coh'_L$.  To see that
  $\coh'_L=\coh_L$, let $\gamma$ be a simple closed curve in $\Su$.  Consider
  the surface $C=\gamma\times I\setminus v(L)\subset X\setminus L$ where
  $v(L)$ is an open tubular neighborhood of $L$.  Then modulo $2$, the
  boundary of $C$ is given by $[\gamma\times0]+[\gamma\times1]+$some
  meridians.  The number of these meridians is given by the number of
  intersection points of $\gamma\times[0,1]\cap L$ which is also the number of
  intersection points of $\gamma\cap \pi(L)$ and is given modulo 2 by
  $\iota(\pi_*([L]),[\gamma])=\coh_L([\gamma])$.  Hence
  $\coh'_L([\gamma])=\coh_L^X([\gamma\times1])=\coh_L([\gamma])$.
\end{proof}
\begin{theo}
  Let $\Su=(\Su,\emptyset,\coh_0,\La)$ be a decorated surface. There
  is a representation of $\KS(\Su)$ in $\VVV(\Su)$ given by
  blocks
  $$[L]\mapsto K(L)\text{ where }
  K(L)_{|\V(\Su+\coh)}:\VV(\Su+\coh)\to\VV(\Su+\coh+\coh_L)$$ is given
  by sending $[L]$ to $\VV(M_L)$ for any lift
  $M_L:\Su+\coh\to\Su+\coh+\coh_L$ of $L$.
\end{theo}
\begin{proof}
  Let $M_L$ be a lift of $L$.\\
  There is a isomorphism in $\cat$, $w:S_1\to S_1^*$ sending
  $v_1\mapsto v_0^*$ and $v_0\mapsto-qv_1^*$.  It is compatible with
  the pivotal structure i.e. $(w^*)^{-1}\circ w:S_1\to S_1^{**}$ is
  given by the multiplication by the pivotal element $K^{1-r}$.
  Graphically, we have the following skein relations where all strands
  are colored by $S_1$ and all coupons by $w$ or $w^{-1}$~:
  $$F\bp{\epsh{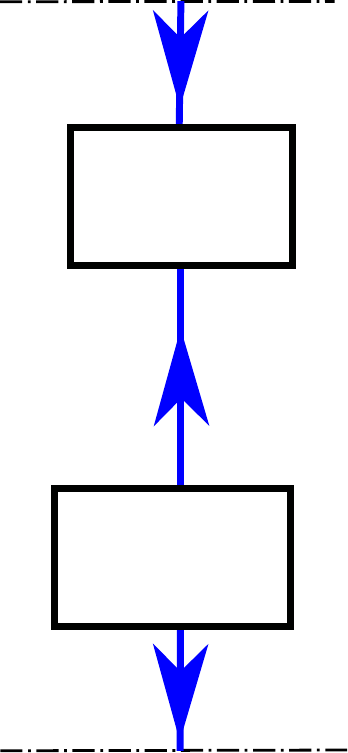}{8ex}}=F\bp{\epsh{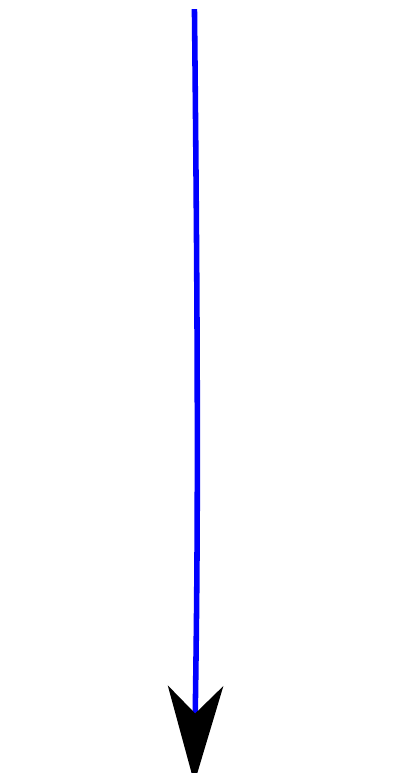}{6ex}}\et 
  F\bp{\epsh{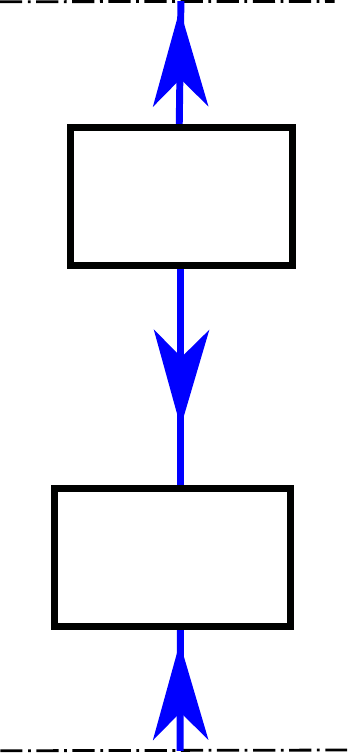}{8ex}}=F\bp{\epsh{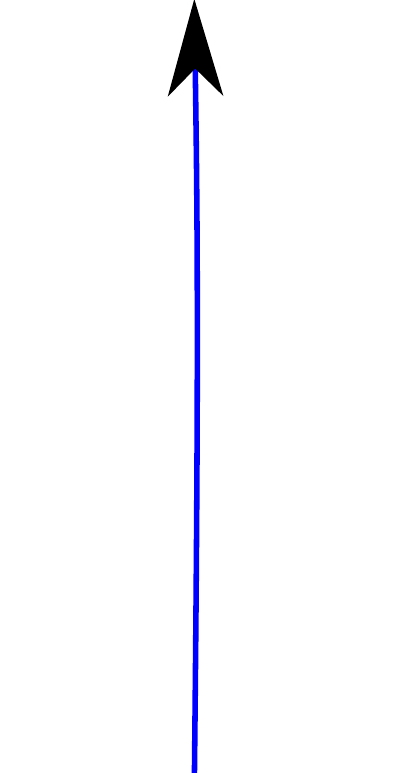}{6ex}}$$
  Using these two relations one easily shows that reversing the
  orientation of a component of $M_L$ leads to a skein equivalent
  cobordism.  Next, suppose that $L_X,L_0,L_\oo\subset\Su\times[0,1]$
  are related by a Kauffman skein relation. Let $M_X,M_0,M_\oo$ be
  lifts of these links and assume that only the components of their
  link involved in the skein relation differ.  We claim that
  $M_X-q^{\frac12}M_0-q^{-\frac12}M_\oo$ is skein equivalent to zero.
  To see this, we modify the components of $M_0$ and $M_\oo$, by
  inserting $w^{\pm1}$-colored coupons as above if necessary so that
  they are diffeomorphic to $M_X$ outside the ball where they differ
  as in the Kauffman skein relation.  According to the orientation of
  the link in $M_X$, we are left to check four skein relations as for
  example
  $$F\bp{\epsh{fig26}{6ex}}-q^{\frac12} F\bp{\epsh{fig32}{6ex}} 
  -q^{-\frac12}F\bp{\epsh{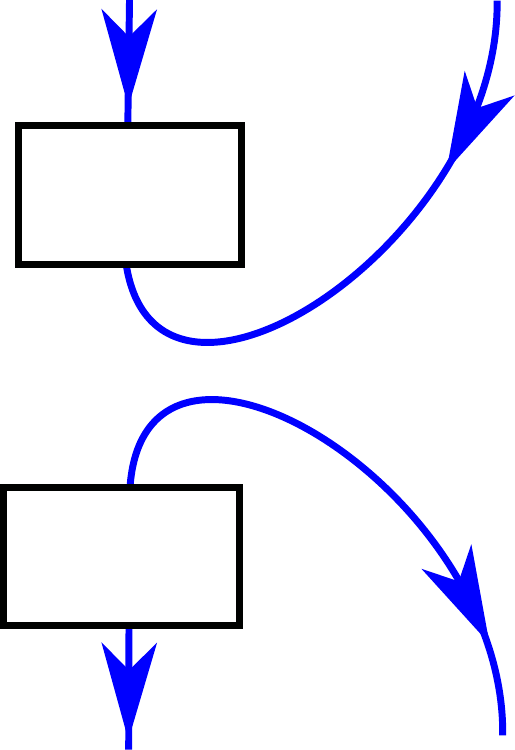}{7ex}}$$
  which is a direct computation.

  Hence the map $L\mapsto \VV(M_L)$ satisfies the Kauffman skein
  relation and induces a map on $\KS(\Su)$.
\end{proof}
\begin{rem}
  Similarly, one can define $\VVV'(\Su)$ and
  $$K':\KS(\Su)\to\End_\C(\VVV'(\Su)).$$
\end{rem}
\begin{rem}
  Following the ideas of \cite{FWW}, one should be able to prove the
  asymptotic faithfulness modulo center of the mapping class group
  representation into $\VV(\Su;\Z/2\Z)$ where
  $\Su=(\Su,\emptyset,0,\La)$.  Here the asymptotic is for the order
  $2r$ of $q$ going to $\oo$.  More precisely, the following should
  hold: for any non central $f$ in the mapping class group of $\Su$,
  and any simple curve $L\subset\Su$ with $f(L)$ non isotopic to $L$,
  there exists $r_0\in\N$ such that $$r>r_0\implies
  \VVV_r(M_L)\neq\VVV_r(M_f)\VVV_r(M_L)\VV_r(M_f)^{-1}=\VVV_r(M_{f(L)})$$
  where $M_L$ is a lift of
  $L\subset\Su\times\frac12\subset\Su\times[0,1]$ and $M_f$ is the
  mapping cylinder of $f$ (cf Section \ref{S:cobcat}).  One difference
  with \cite{FWW} would be that it is not clear for this
  representation that such a $r_0$ should be dependent of $f$; indeed
  there are no evident elements in the kernel of these representations
  as the powers of the Dehn twists are not necessarily
  trivial.
\end{rem}

\subsection{A TQFT for the Volume Conjecture.}\label{sub:kashaevTQFT}
In this subsection we restrict our analysis to the case when $\coh=0$ and
explain how this case is related to the volume conjecture. We compute
explicitly some of the graded dimensions of the vector spaces $\VV(\Su)$ and,
in the end of the section, we show that this case produces a new set of
representations of mapping class groups which could possibly be injective (we
could find no element in the kernel of these representations).

Let $\Cob_{\wb 0}$ be the subcategory of $\Cob$ formed by all the objects and
cobordisms whose cohomology classes are $0$.  If $\Su\in \Cob_{\wb 0}$ then
$\mcg(\Su)$ acts projectively on $\V(\Su)$ for every $r$. If $r$ is odd then
$V_{\wb 0}$ (also known as the Kashaev module) belongs to the subcategory $\cat_{\wb 0}$
of $\cat$ and thus the graphs decorating the morphisms of $\Cob_{\wb 0}$ may be
colored by it.  In particular if one considers a closed cobordism formed by a
closed $M$ containing a framed oriented link $T$ whose edges are colored by
$V_0$ the invariant one gets is the generalized Kashaev invariant defined in
\cite{CGP} (so called because if $M=S^3$ it coincides with the Kashaev
invariant of $T$ (see \cite{MM})).
Hence the restriction of the functor $\VV$ to $\Cob_{\wb 0}$ yields a
TQFT in particular for the Kashaev invariant.  In view of the Volume
Conjecture it is natural to wonder whether the action of $\mcg(\Su)$ on
$\VV(\Su)$ (for $\Su\in \Cob_{\wb 0}$) is asymptotically related to a
geometric action like, for instance, the action on the space of $L^2$
functions on the character variety of $\Su$: 
$$\chi_{\SL(2;\C)}(\Su)=\{\rho:\pi_1(\Su)\to
\SL(2;\C)\}//\SL(2;\C)$$ where the quotient is meant in the sense of
geometric invariant theory, and the measure is induced by Goldman's
symplectic form.  So given a diffeomorphism $\phi\in \mcg(\Su)$,
letting $N_\phi\in\End(\VV(\Su))$ be the morphism induced by the
mapping cylinder of $\phi$ (as defined in Subsection
\ref{rem:mcgrepres}), we ask the following:
\begin{question} 
If $\phi$ is pseudo-Anosov, is there a relation between asymptotic behavior (for large $r$) of the
  spectrum of $\VV(N_\phi)\in\End(\VV(\Su))$ and the dilatation factor of $\phi$? Can one recover the hyperbolic volume of the mapping torus
  associated to $\phi$?
\end{question}

Currently the above questions are only motivated by the generalization of
volume conjecture proposed in \cite{CGP}: indeed if one considers for instance
$\Su$ being a closed surface containing some $V_0$-colored points, then the
mapping cylinder of a pseudo Anosov diffeomorphism permuting these points is
hyperbolic and the image of the points gives a $V_0$-colored link in the
cylinder.

Another aspect of our theory is that it seems to be related to the logarithmic
conformal field theories and in particular to the representation theory for
the triplet algebras $W(r)$ (see \cite{FGST1} and \cite{FGST}). Indeed as
shown in \cite{FGST1} there is a $SL(2;\Z)$ equivariant correspondence between
the space of conformal blocks associated to the triplet algebra $W(r)$ and the
center of the small quantum group $\overline{\Uqg}_q(\mathfrak{sl}_2)$ (seen
as a representation s of $SL(2;\Z)$ as shown in \cite{LM}).  But we expect
that for $\Su=S^1\times S^1$ and $\coh=0$, there exists a $SL(2,\Z)$-covariant
map from the center $Z$ of the small quantum group
$\overline{\Uqg}_q(\mathfrak{sl}_2)$ to $\V(\Su)$.  The map can be described
as follows : 
\newcommand{\sP}{{\mathsf P}} 
For $r$ odd, let $\sP=\bigoplus_{i=0}^{r-1}P_i\otimes
\C^H_{\epsilon(i)r}$ where $\epsilon(i)=1$ if $i$ is odd and $0$ else.
If instead $r$ is even, let $\sP=\bigoplus_{i=0}^{\frac
  r2-1}P_{2i}$. For any $f\in\End_\cat(\sP)$, let
$v(f):\emptyset\to(S^1\times S^1,\emptyset,0,\La)$ be the cobordism
$(S^1\times B^2,S^1\times \{0\})$ where the color of $S^1\times \{0\}$
is $\sP$ and there is a coupon on it containing the endomorphism $f$.
There is a natural map from $Z$ to the center of $\UsltH_{/{K^r-1}}$
(indeed every element of $Z$ has weight zero and thus commutes also
with the action of $H$) which acts faithfully on $\PP$.  Using this
action $\rho$, we get the map $Z\to\V(S^1\times S^1)$ given by
$z\mapsto v(\rho(z))$.  As any indecomposable module of $\cat_{\wb 0}$
is the tensor product of a direct summand of $\sP$ with a power of
$\sigma$, one can show that vectors $v(f)$ generate the skein module
$\Skein(S^1\times B^2,S^1\times S^1)$ and so also $\V(S^1\times
S^1,\emptyset,0,\La)$.  Because of the skein relations in the torus,
this map factors through an isomorphism
$H\!H_0(\End_\cat(\sP))\to\Skein(S^1\times B^2,S^1\times S^1)$ from
the Hochschild homology of $\End_\cat(\sP)$.  
\begin{question}[Relations between $\V$ and the logarithmic CFTs]
  Is the above map equivariant with respect to the actions of $SL(2;\Z)$ on
  $Z$ described in \cite{LM} and the action of the mapping class group of
  $S^1\times S^1$?  What are its kernel and cokernel?  Can one extend these
  correspondence to more general spaces of conformal blocks?
\end{question}

Even if Theorem \ref{P:verlinde} does not apply to the case of a closed
surface containing no points and whose cohomology class is $0$, it does apply
to a fairly similar case: let $\Su^g_{2j}$ be a connected decorated surface of
genus $g\geq 1$ and containing exactly one point $p$ whose color is $P_{2j}$
for some $j\in \N$, $j<\frac{r-1}{2}$ or possibly $V_0$ in case $r$ is
odd.  The following is a direct corollary of Theorem \ref{P:verlinde} and Lemma \ref{lem:generalhopflinks}:
\begin{prop}[Graded dimensions of $0$-cohomology vector spaces]
  If the color of $p$ is $P_{2j}$ with $j<\frac{r-1}{2}$ then
  $$\dim_t \VV(\Su^g_{2j})=(-1)^{r-1}\frac{(r')^g}{r}\sum_{k\in \Hr} 
  \left(\frac{\{r\beta\}}{\{\beta+k\}} \right)^{2g-1}
  \left(q^{(r-1-2j)(\beta+k)}+q^{-(r-1-2j)(\beta+k)}\right).$$ If the color of
  $p$ is $V_0$ (and hence $r$ is odd) then :
$$\dim_t \VV(\Su^g_{{r-1}})=r^{g-1}\sum_{k\in \Hr} 
  \left(\frac{\{r\beta\}}{\{\beta+k\}} \right)^{2g-1}.$$ 
\end{prop}

If $\coh=0$ then $\mcg(\Su)$ acts on $\VV(\Su)$. The following proposition
shows that the order of the action of each Dehn twist along a non-separating
curve is infinite (while, this is not the case for the standard WRT theory):
\begin{theo}\label{T:Dehnnon-separating}
  Let $\Su$ be a closed decorated surface whose cohomology class is $0$
  (possibly containing some points), let $c\subset \Su$ a non-separating
  curve and $\DT_c$ the Dehn-twist along $c$. Then the order of the action of
  $\DT_c$ on $\VV(\Su)$ is infinite.
\end{theo}

\begin{proof}
  Let $c'$ be a curve in $\Su$ intersecting $c$ exactly once and let $H$ be a
  handlebody bounding $\Su$, and push $c'$ into $H$. Equip $c'$ with the
  framing parallel to $\Su$ and let $T$ and $xT$ be the $\cat$-ribbon graphs
  represented by $c'$ equipped respectively with the color $P_{0}$ and with
  the color $P_0$ together with a coupon containing a nilpotent endomorphism
  $x_0$.  Then 
$\DT_c$ acts on $[H,T]$ as the twist on $P_0$ and 
it holds $$\DT_c([H,T])=[H,T]-(r-1)\{1\}[H,xT], \ {\rm and}\
  \DT_c([H,xT])=[H,xT].$$ To prove our claim it is then sufficient to show that
  the vector $[H,xT]$ is non-zero in $\VV(\Su)$ (this automatically proves
  also that $[H,T]\neq 0$).  To this purpose let $H'$ be a handlebody bounding
  $\overline{\Su}$ so that $H'\circ H=S^3$ and $[H',\emptyset]\in \VV'(\Su)$
  the vector it represents.  Then it holds $\langle
  [H',\emptyset],[H,xT]\rangle=\qt(xP_0)\neq 0$.
\end{proof}

\subsection{$r=2$ and the Torsion}\label{sub:r=2torsion}
For $r=2$, the colored invariant of links is related to the
multivariable Alexander polynomial.  This fact first observed by Jun
Murakami 
\cite{jM2}
is exposed with many details in Viro's paper \cite{Vi}.
 Complete formulas for computing Reidemeister torsion from multivariable
Alexander polynomials on a surgery presentation are given by Turaev in
\cite{Tu2}. The Reidemeister torsion was defined in 1935 \cite{Re},
and shown to classify lens spaces.  We show that our invariant for
$r=2$, denoted by $Z_2$, is a canonical normalization of Reidemeister
torsion.  We give the exact relation with Reidemeister-Turaev torsion
involving an additional complex spin structure (Theorem \ref{teo:reidemeister}).  In particular we show directly that this
quantum invariant classifies lens spaces (Proposition \ref{prop:distinguishlens}). The construction in this
paper, when specialized at $r=2$, gives a TQFT extension of abelian
Reidemeister torsion. We could speculate that this picture extend to
the general case for a non commutative torsion to be understood.

If $r=2$ then $q=\e^{\frac{i\pi}2}=i$ and we have
$\qd(\alpha)=\dfrac{-2}{i^\alpha+i^{-\alpha}}=\dfrac{2i}{\qn{\alpha-1}}$.  
\begin{prop}\label{prop:FvsDelta2} For $r=2$, the  invariant $F'$ of $\C\setminus
  (2\Z+1)$-colored trivalent graphs in $S^3$ 
   is related to
  the invariant $\underline \Delta_2$ defined by Viro in \cite{Vi} by
  $$\forall T,\ \ \ 
  F'(T)=(-2i)^{1-v/2}\underline \Delta_2(\overleftarrow T)$$ where $v$ is
  the number of vertices of $T$ and $\overleftarrow T$ is the graph with reversed orientation.
\end{prop}
\begin{proof}
The equality is easily seen to hold for the unknot and for the theta graph (see \cite{Vi} Section 8.1 for the values of $\Delta_2$). For what concerns the tetrahedron, let us fix the notation of a coloring on $T$ as follows: $\epsh{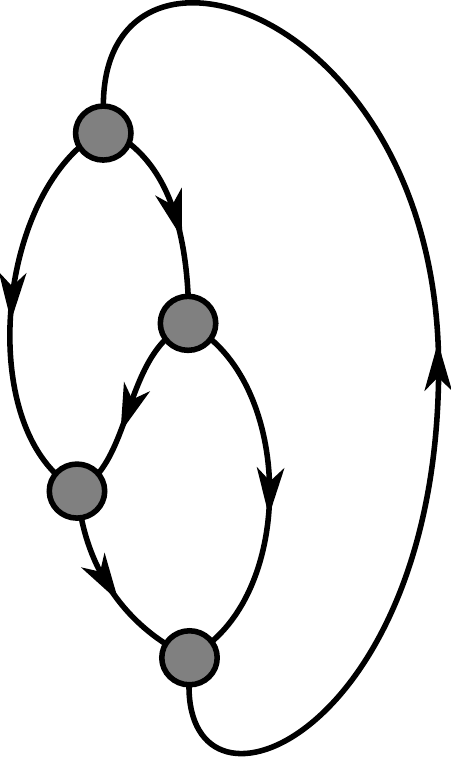}{12ex}
  \put(-36,10){\ms{j_1}}\put(-28,2){\ms{j_2}}\put(-28,-15){\ms{j_3}}
  \put(-12,-7){\ms{j_4}}\put(-18,14){\ms{j_6}}\put(-8,3){\ms{j_5}}$.
  
In \cite{Vi} only colorings such that the algebraic sum of the colors around each vertex is in $\{\pm 1\}$ were considered. For instance with our notation these conditions read as $j_6+j_1-j_5=\pm1, j_1+j_2-j_3=\pm 1$ and so on. Then there are exactly two vertices in the tetrahedron where the sum is $-1$ and exactly one  edge connecting them which we will call the ``special edge''. According to \cite{Vi}, Section 8.2, the value of $\Delta_2(T)$ is $i^{1+A}-i^{-1-A}$ if $A$ is the color of this edge.  
In our case, the value of the invariant $F'(T)$ is :
$\sjtop{j_1}{j_2}{j_3}{j_4}{j_5}{j_6}$ (an explicit formula was provided in Section 1.2 of \cite{CGP}). In particular it holds: 
$$\sjtop\alpha{\beta-\alpha+1}\beta
{\gamma-\beta+1}\gamma{\gamma-\alpha+1}=\dfrac{\qn{\gamma-\alpha}}{\qn1}=\frac{i^{\alpha-\gamma}-i^{\gamma-\alpha} }{-2i}.$$
If we reverse the orientations of all the edges, with our choice of coloring the special edge is $j_6$ and so its color is $A=-(\gamma-\alpha+1)$ (the minus coming from the arrow reversal). This proves the claim also for the tetrahedral graphs.

To conclude it is sufficient to prove that the invariants have the same characterizing properties, which allow to reduce their computation to products and sums of theta graphs, tetrahedra and unknots. These properties are those specified in Section 9 of \cite{Vi} for $\Delta_2$ and in Section 1.2 of \cite{CGP} for $F'$.  
\end{proof}

\begin{cor}\label{cor:conway} Let $L$ be a framed oriented link in
  $S^3$ with $n$ ordered components colored
  $\alpha_1,\ldots,\alpha_m\in\C\setminus(2\Z+1)$, then the 
  invariant $F'(L)$ for $r=2$ is related to
  the Alexander-Conway function $\nabla(L)$ by
  $$F'(L)=-2i\nabla(L)(i^{1-\alpha_1},\ldots,i^{1-\alpha_m})
  i^{\sum_{j,k=1}^m\frac{\alpha_j\alpha_k-1}{2}\lk_{jk}}$$
  Where $\lk$ is the linking matrix of $L$.
\end{cor}
\begin{proof}
This is a direct consequence of Proposition \ref{prop:FvsDelta2} and of Proposition 8.6.A in \cite{Vi}.
\end{proof}

Let us recall for $r$ even the set $\Hr=\{1-r,\ldots,-1,1,\ldots,r-3,r-1\}$ and
denote by $\Hr^+=\Hr\cap\N$.  Then we have:
\begin{prop}\label{evencase}
  Let $\Omega^+_\alpha=\sum_{k\in\Hr^+}\qd(\alpha+k)V_{\alpha+k}$ then
  $\Omega_\alpha=(\unit\oplus-\sigma^{-1})\otimes\Omega^+_\alpha$
  where the minus sign comes from $\qd(\beta-r)=-\qd(\beta)$.  Hence,
  if $L\cup T$ is a computable presentation of a compatible triple,
  $$F'((L,\lambda\Omega)\cup T)=F'((L,2\lambda\Omega^+)\cup T).$$ 
\end{prop}

In the case $r=2$, it is convenient to use 
$$\tilde\Omega_\alpha=\frac{1}{2}\Omega^+_\alpha=\frac{\qd(\alpha+1)}{2}V_{\alpha+1}=\frac{i}{i^\alpha-i^{-\alpha}}V_{\alpha+1}\ .$$
\begin{rem}
  An advantage  in the case $r=2$ is that $\cat_{\wb 1}$ is
  semi-simple with simple objects $\sigma^k\otimes V_0$.  Thus there
  exists a reduced Kirby color $\tilde\Omega_{\wb 1}$ and this allows  to extend the
  class of triples where $Z$ is defined: any triple
  $(M,T,\omega)$ with $\omega\neq0$ admits a computable presentation.
\end{rem}

 In the general formula for the normalized invariant $Z$ we have to
substitute $\lambda=\frac{1}{4}$, $\eta=\frac{1}{2}$, $\dep=i$.
 The normalized invariant of a connected $M=S^3(L)$
containing a colored graph $T$ and cohomology class $\omega$ represented by $\alpha$ and with canonical extended structure is then given by formula \eqref{eq:Zrconnected}:
$$Z_2(M,T,\omega)=\frac{1}{2}i^{-\sigma}F'((L_j,\tilde\Omega_{\alpha_j}),T) \ .
$$

In particular, if $T=\emptyset$ and $L$ has $m$ components, we can use Corollary \ref{cor:conway} and get:
$$ F'\left((L_j,\tilde\Omega_{\alpha_j})\right)=\left(\prod_{j=1}^m\frac{i}{i^{\alpha_j}-i^{-\alpha_j}}\right)
(-2i)\nabla(L)(i^{-\alpha_1},\dots,i^{-\alpha_m})i^{\sum_{j,k=1}^m\frac{(\alpha_j+1)(\alpha_k+1)-1}2\lk_{jk}} \ .
$$
Finally recalling that the Alexander-Conway function satisfies 
$$\nabla(L)(t_1^{-1},\ldots,t_m^{-1})=(-1)^m\nabla(L)(t_1,\ldots,t_m)$$ and observing that $\sum_{j,k} \frac{(\alpha_j+1)(\alpha_k+1)-1}{2}\lk_{jk}=\sum_{j,k} \frac{\alpha_j(\alpha_k+2)}{2}\lk_{jk}$ we get:

\begin{prop}
If $M$ is $3$-manifold with nonzero cohomology class $\omega\in H^1(M,\C/2\Z)$,
given by a computable surgery presentation $(S^3(L),\alpha)$ (each $\alpha_j\in\C/2\Z$ is nonzero)
  then 
$$ Z_2(M,\emptyset,\omega)= i^{-\sigma-m-1}
  \left(\prod_{j=1}^m\frac{1}{i^{\alpha_j}-i^{-\alpha_j}}\right)\nabla(L)(i^{\alpha_1},\dots,i^{\alpha_m})i^{\sum_{j,k=1}^m\frac{\alpha_j(\alpha_k+2)}2\lk_{jk}}\ .$$
\end{prop}

The Turaev refined abelian Reidemeister torsion  is defined for a closed $3$-manifold
$M$ endowed with an homomorphism $\phi: H_1(M)\rightarrow \C^*$ and a complex
spin structure \mbox{$\sigma\in Spin^c(M)$} (or Euler structure). To a cohomology class $\omega\in H^1(M,\C/2\Z)$
we associate the homomorphism  $\phi_\omega: H_1(M)\rightarrow \C^*$ defined by the formula
$$\phi_\omega(h)=e^{i\pi<\omega,h>}= i^{2<\omega,h>}\ .$$
 Turaev uses the spin$^c$ structure and the canonical homology orientation
to fix the indeterminacy in Reidemeister torsion. 
Our invariant $Z_2(M,\emptyset,\omega,0)$ is a version of the Reidemeister torsion for the homomorphism given by $\phi_\omega$
which fixes the indeterminacy without any extra choice.
 The theorem below clarifies the relationship between $Z_2(M,\emptyset,\omega,0)$ and the Turaev refined torsion $\tau^{\phi_\omega}(M,\sigma)$, 
$\sigma \in Spin^c(M)$. Here the canonical homology orientation is used.

\begin{theo}\label{teo:reidemeister}
  If $M$ is a closed $3$-manifold equipped with a spin$^{c}$-structure
  $\sigma$ and non trivial cohomology class $\omega\in H^1(M,\C/2\Z)$,
  and $\phi_\omega \in \mathrm{Hom}(H_1(M),\C^*)$ is the corresponding
  homomorphism, then
  $$\tau^{\phi_\omega}(M,\sigma)=
  Z_2(M,\emptyset,\coh)\times i^{1+b_1(M)+4\psi_\sigma (\omega)}\ .$$
\end{theo}
Here $b_1(M)$ is the first Betti number and
the map $\psi_\sigma: H^1(M,\C/2\Z)\rightarrow \C/\Z$ is a version of the  quadratic linking function $\phi_\sigma :H_2(M,\Q/\Z)\rightarrow \Q/\Z$
defined by Deloup-Massuyeau in \cite{DM2} (Definition 2.2). The definition extends to a function $\phi_\sigma :H_2(M,\C/\Z)\rightarrow \C/\Z$, and we define
$\psi_\sigma$ by the formula 
$$\psi_\sigma(\omega)=\phi_{\overline{\sigma}}\left(\frac{1}{2}D\omega\right)\ .$$
Here $\sigma\mapsto \overline \sigma$ is the natural involution on complex spin structures, $D:H^1(M;\C/2\Z)\to H_2(M;\C/2\Z)$ is Poincar\'e duality and $\times \frac{1}{2}$ is the map induced by the natural ``division by two'' isomorphism from $\C/2\Z$ to $\C/\Z$. 

 If we are given a surgery presentation of $M$ with corresponding $4$-manifold $W$, a class
$\tilde w\in  H^2(W,M;\C)$ which is a lift to $\C$-coefficients of $w=\delta \omega \in H^2(W,M;\C/2\Z)$ (where $\delta:H^1(M;\C/2\Z)\to H^2(W,M;\C/2\Z)$ is part of the exact sequence of the pair $(W,M)$), and we let
$c(\tilde \sigma)\in H^2(W,\Z)$ be the Chern class of a complex spin structure 
$\tilde \sigma $ on $W$ whose restriction to $M$ is $\sigma$,
 then 
$\psi_\sigma(\omega)$ is given by the following formula
\begin{equation}
\label{quadratic}
\psi_\sigma(\omega)=-\frac{1}{8}\tilde w\cdot \tilde w -\frac{1}{4}\tilde w\cdot  c(\tilde \sigma)\in \C/\Z
\end{equation}
Invariance of this formula follows from \cite[Lemma 2.6]{DM2}.

\begin{proof}
  We consider a computable presentation with $m$ components:
  $M=S^3(L)$, $\omega=\omega_\alpha$ with $\alpha_j\in\C/2\Z$ non zero
  for $1\leq j\leq m$. The Turaev surgery formula for the torsion is
  given by the formula \cite[VIII.2,2b]{Tu2}
  $$\tau^{\phi_\omega}(M,\sigma)=(-1)^{b_1(M)+m}\mathrm{det}_0(B_L)\prod_{j=1}^m\frac{t_j^{\frac{k_j-1}{2}} }
{t_j^{\frac{1}{2}}-t_j^{-\frac{1}{2}}}
\nabla(L)(t_1^{\frac{1}{2}},\dots,t_m^{\frac{1}{2}})\ .$$ Here
$B_L=\mathrm{lk}$ is the linking matrix, and letting $b_+$ ($b_-$) the
number of its positive (resp. negative) eigenvalues we have
$\mathrm{det}_0(B_L)=(-1)^{b_-}$; moreover we have
$t_j^{\frac{1}{2}}={i^{\alpha_j}}$, and the $k_j$, $1\leq j\leq m$,
are the charges (combinatorial data described in \cite{Tu2}, VI 2.2,
to identify $\sigma$). Complex spin structures are alternatively
described by Chern vectors $s_j$ \cite{DM}, and the correspondence is
given by the formula \cite[3.1]{DM} $$k_j-1=-s_j+\sum_k
\mathrm{lk}_{jk}\ , \ 1\leq j\leq m\ .$$
$$\tau^{\phi_\omega}(M,\sigma)=(-1)^{b_+}\prod_{j=1}^m
\frac{1 }{i^{\alpha_j}-i^{-\alpha_j}}
 \nabla(L)(i^{\alpha_1},\dots,i^{\alpha_m})\times i^{-\sum_j\alpha_js_j+\sum_{j,k}\alpha_j\mathrm{lk}_{jk}}\ .$$
$$\tau^{\phi_\omega}(M,\sigma)=Z_2(M,\emptyset,\alpha)
\times i^{1+m-b_+-b_--\sum_{j,k=1}^m\frac{\alpha_j\alpha_k}2\lk_{jk} -\sum_j\alpha_js_j}
\ .$$
We have $b_1(M)=m-b_+-b_-$. 
Identification of the last factor is obtained from formula \ref{quadratic}.
\end{proof}

\begin{prop}\label{prop:distinguishlens} Let $p>q>0$ be two coprime integers and consider the lens
  space $L(p,q)$ equipped with a non zero cohomology class $\coh$.  Then
  $$Z_2(L(p,q),\coh)=\frac{(-1)^{k}\e^{\frac{i\pi k^2q}{p}}}
  {4i\sin\frac{\pi kq}p\sin\frac{\pi k}p}$$ for some $k\in\Z\setminus 
  p\Z$ depending on $\coh$.  Hence the invariants $Z_2$ distinguish
  the lens spaces.
\end{prop}
\begin{proof}
  There exists integers $a_i\ge2$ such that $\displaystyle{\frac
    pq=\brk{a_1,\ldots,a_n}= a_1-\frac1{a_2-\cdots
      \frac1{a_{n-1}-\frac1{a_n}}}}$ ($a_i$ is the smallest integer
  greater that $\brk{a_i,\ldots,a_n}$).  Let $L= \ms{\begin{pmatrix}
      a_1&1&&0\\
      1&a_2&\ddots\\
      &\ddots&\ddots&1\\
      0&&1&a_n
    \end{pmatrix}}$ then $L(p,q)$ is presented by surgery on the chain
  link $C=C(a_1,\ldots,a_n)$ whose linking matrix is given by $L$ ($C$
  is a connected sum of framed Hopf links see \cite{PS}).

  For $k\in \Z\setminus p\Z$, let $\vec\coh=L^{-1} \ms{\begin{pmatrix}
      2k\\0\\\vdots\\0
    \end{pmatrix}}$ and $\vec u=\ms{\begin{pmatrix} 1\\\vdots\\1
    \end{pmatrix}}$.  Then the coordinates of $\coh$ satisfy
  $\coh_1=2k\frac qp$
  and $\coh_n=2k\frac{(-1)^{n-1}}p$ (the vector
  $\coh$ represents all cohomology class for the different values of
  $k$).  Indeed if one defines $c_{n+1}=0$, $c_n=1$ and recursively
  $c_i=-a_ic_{i+1}-c_{i+2}$ for $i=0\cdots n-1$, then $(-1)^ic_{n-i}$
  is an increasing sequence of integers, $\gcd(c_i,c_{i+1})=1$ and
  $\frac{c_{i-1}}{c_i}=-\brk{a_i,\ldots,a_n}$.  Thus $c_0=(-1)^np$,
  $c_1=(-1)^{n-1}q$ and finally $\coh_i=(-1)^{n-1}\frac{2k}p c_i$.
  $Z_2(L(p,q),\coh)=\frac{1}{2}i^{-\sigma}F'((C,\tilde\Omega_{\coh_j})).$
  $L$ is positive definite and so $\sigma=n$.  Now the link being a
  connected sum of Hopf links, its invariant is easy to compute (see
  \cite{CGP}) and thus
  \begin{align*}
    Z_2(L(p,q),\coh)&=\frac{1}{2^{n+1}}i^{-n}\qd(\coh_1+1)\qd(\coh_n+1)
  \prod_{j=1}^ni^{\frac{a_j}2((\coh_j+1)^2-1)}
  \prod_{j=1}^{n-1}-2i^{(\coh_j+1)(\coh_{j+1}+1)}\\
  &=\frac{-4i^{-n}}{2^{n+1}\qn{\coh_1}\qn{\coh_n}}(-2)^{n-1}
  \exp\bp{\frac {i\pi}4\bp{^t(\vec\coh+\vec u)L(\vec\coh+\vec u)-\sum_ia_i}}\\
  &=\frac{i^{n}}{\qn{\coh_1}\qn{\coh_n}}
  \exp\bp{\frac {i\pi}4\bp{^t(\vec\coh+\vec u)L(\vec\coh+\vec u)
    -{^t\vec u}L\vec u+2(n-1)}}\\
  &=\frac{i^{2n-1}}{\qn{\coh_1}\qn{\coh_n}}
  \exp\bp{\frac {i\pi}4\bp{^t\vec\coh L(\vec\coh+2\vec u)}}\\
  &=\frac{(-1)^{n-1}i}{\qn{2k\frac qp}\qn{2k\frac{(-1)^{n-1}}p}}
  \exp\bp{\frac {i\pi}4\bp{{2k}(2k\frac qp+2)}}\\
  &=\frac{(-1)^{k}i}{\qn{\frac {2kq}p}\qn{\frac{2k}p}}
  \exp\bp{\frac {i\pi k^2q}p}\\
  \end{align*}
  The Reidemeister-Franz proof of the classification of lens spaces using the
  torsion (see also \cite{Ma}) can then
  be applied.  Recall Franz's lemma:   
  \begin{lemma}[Franz, \cite{Franz}]\label{lem:franz}
    Let $p\geq 1$ be an integer and let $\Z^\times_p$ be the multiplicative
    group of invertible elements of $\Z/p\Z$. Let $a : \Z^\times_p \to \Z$ be
    a map such that
    \begin{enumerate}
    \item $\sum_{ j \in \Z^\times_p} a ( j ) = 0$,
    \item $\forall j \in \Z^\times_p , a(j) = a(-j)$,
    \item $\prod_{j\in \Z_p^\times}(\zeta^j -1)^{a(j)} = 1$ for any $p^{th}$
      root of unity $\zeta\neq 1$.
    \end{enumerate}
    Then, $a(j) = 0$ for all $j \in \Z_p^{^\times}$.
  \end{lemma}
    Now, suppose that for some $q'$ coprime with $p$ there exists 
    an invertible
    $c\in \Z/p\Z$ such that 
    \begin{equation}\label{eq:franzreidemeiseter}
    \frac{(-1)^{k}\e^{\frac{i\pi k^2q}{p}}}
  {4i\sin\frac{\pi kq}p\sin\frac{\pi k}p}=\frac{(-1)^{ck}\e^{\frac{i\pi c^2k^2q'}{p}}}
  {4i\sin\frac{\pi c kq'}p\sin\frac{\pi ck}p},\ \forall k\in \Z\setminus p\Z.
    \end{equation}
    Taking the absolute values and setting $\zeta=\exp(\frac{2i\pi k}{p})$ we get 
   $$(1-\zeta^q)(1-\zeta)(1-\zeta^{-q})(1-\zeta^{-1})=(1-\zeta^{cq'})(1-\zeta^{c})(1-\zeta^{-cq'})(1-\zeta^{-c})
   $$ 
   which holds for all $p^{th}$-roots of unity. 
   This, 
   implies that the sets $\{1,-1,q,-q\}$ and $\{c,-c,cq',-cq'\}$ must coincide, thus leaving $8$ possible cases. Indeed letting for each $j\in \Z_p^{\times}$ $m(j)$ and $m'(j)$ be the number of times $j$ appears in the sequence $(1,-1,q,-q)$ and $(c,-c,cq',-cq')$ respectively, it can be easily verified that $a(j)=m(j)-m'(j)$ satisfies the hypotheses of Lemma \ref{lem:franz}. Thus $m(j)=m'(j)$ and the two sequences coincide up to permutation.
   Now, analyzing the argument mod $\pi$ of Equation
   \eqref{eq:franzreidemeiseter}, we have $q=c^2q'$ mod $p$.   Thus modulo
   $p$, either
   $c^2=1$ and $q=q'$ or $c^2=q^2$ and $qq'=1$.
   Finally, the isotopy between the chain links $C(a_1,\ldots,a_n)$
   and $C(a_n,\ldots,a_1)$ implies that $L(p,q)\simeq L(p,q')$ through
   an orientation preserving diffeomorphism when $qq'=1$ mod $p$.
  \end{proof}
  
\subsection{The space $\V(S^1\times S^1)$ in the case
  $r=2$.}\label{sub:torus2}
In this subsection we give a basis for $\V(S^1\times S^1,\emptyset, \coh, \La)$ when $r=2$ and $\coh$ is any cohomology class and $\La$ any Lagrangian.  We use this basis in the next subsection when considering the action of the mapping class group on $\V(S^1\times S^1,)$.

As explained in the preceding subsection, the case $r=2$ is specific
for many reasons. One is that the projective module $P_{r-1}=P_1$
coincides with both $V_0$ and $S_1$; also the projective module  $P_0$ is isomorphic to $V_0\otimes V_0$.  The basis of
$P_0$ is given by vectors $v_0^L,v_2^U, v_0^R,v_{-2}^D$ where 
$Hv_i^X=iv_i^X$ for $X\in \{U,L,R,D\}$,  $E(v_0^L)=v_2^U, F(v_0^L)=v_{-2}^D, E(v_{-2}^D)=v_0^R$,
$F(v_2^U)=v^R_0$ and the action of $E$ and $F$ on all other vectors is zero (see the Appendix).  We leave the following as
an exercise: $$P_0\otimes V_0=(\sigma\otimes V_0)\oplus V_0\oplus
(\sigma^{-1}\otimes V_0),\ {\rm and}\ P_0\otimes P_0=(\sigma\otimes
P_0)\oplus (\sigma^{-1}\otimes P_0)\oplus P_0\oplus P_0.$$ Observe
that $P_0$ has a nilpotent endomorphism $x$ of order $2$ uniquely
defined by $$x(v_0^L)=v^R_0.$$ Finally, $P_0$ generates the category of
projective modules in $\cat_{\wb 0}$.  
For any module $M\in
\cat$ let $T_M$ be the oriented solid torus whose core is a framed
oriented circle colored by $M$; furthermore we shall denote $T_{xP_0}$
the solid torus whose core is colored by $P_0$ and contains a coupon
colored by the nilpotent endomorphism $x:P_0\to P_0$.
\begin{prop}\label{prop:torus0}
Let $r=2$.   
Let $\Su=(S^1\times S^1, \emptyset, \coh, \La)$ be a decorated surface where 
$\coh\in H^1(\Su;\C/2\Z)$ is any
  cohomology class.  If there exists a simple closed curve $m\subset
  \Su$ such that $\alpha:=\coh(m)\notin\Z/2\Z$ then $\V(\Su)$ is
  spanned by $[T_{V_\alpha}]$ where the meridian of $T$ is $m$.  If
  instead $\coh=0$ and $m\subset \Su$ is an oriented non-trivial
  simple closed curve then $\V(\Su)$ is spanned by $[T_{P_0}],
  [T_{xP_0}]$ where $T$ is the solid torus whose meridian is $m$.
  Moreover, $\V'(\Su)$ is spanned by $[T'_{\emptyset}],
  [ST_{\emptyset}]$ where $T'_{\emptyset}$ is an empty, negatively
  oriented copy of $T$ and $ST'_{\emptyset}$ is obtained from
  $T'_{\emptyset}$ by changing the identification of its boundary with
  $S^1\times S^1$ via the map $S(\theta,\theta')=(-\theta',\theta)$.
\end{prop}
\begin{proof}
  The first case is a special case of Proposition \ref{P:torusccase}.  
 In the
  second case 
  we claim that the
  vectors $[T_{P_0}]$ and $[T_{xP_0}]$ are linearly independent in
  $\V(\Su)$.  Indeed, it is sufficient to pair them with co-vectors in
  $\V'(\Su)$ which distinguish them.  If we pair $[T_{P_0}]$ with the
  empty torus $[ST'_\emptyset]$ whose meridian intersects $m$ once, we
  get an $S^3$ containing the core of $T$ as an unknot 
 colored by $P_0$ 
  so $$\Zr([ST'_{\emptyset}]\circ [T_{P_0}])=\eta\,\qt_{P_0}(\Id_{P_0})=0,\
  \Zr([ST'_{\emptyset}]\circ[T_{xP_0}])=\eta \, \qt_{P_0}(x) =-\eta=-\frac{1}{2}$$ 
  where $\qt_{P_0}$ is the modified trace, see Lemma \ref{lem:modifiedtraces}.  
   This shows that
  $[T_{xP_0}]\neq 0$ and that its dual vector in $\V'(\Su)$ is
  $-2[ST'_{\emptyset}]$.

  To show that $[T_{P_0}]\neq 0$ we couple it with the empty torus
  $T'_\emptyset$ whose meridian coincides with that of $T$, thus
  getting a $S^2\times S^1$ containing a knot of the form $\{p\}\times
  S^1$ whose component is colored by $P_0$.
  As $P_0\simeq V_0\otimes V_0^*$, the decorated manifold
  $T'_\emptyset\circ T_{P_0}$ is skein equivalent to two parallel
  knots with opposite orientation in $S^2\times S^1$ colored by $V_0$.
  Then we can apply 1-surgery (Proposition \ref{P:1-surg}) and we
  obtain 
   $\Zr([T'_\emptyset]\circ[T_{P_0}])= (\eta\qd(0))^{-1}\Zr(S^3,o)=1$
   where $o$ denotes the unknot
  colored by $V_0$.

  Similarly it holds $\Zr([T'_{\emptyset}]\circ [T_{xP_0}])=0$. 
  Indeed by Lemma \ref{lem:nilpotentendo} the decorated manifold $[T'_{\emptyset}]\circ [T_{xP_0}]$
  is skein equivalent to two parallel
  knots with opposite orientation in $S^2\times S^1$ colored by $V_0$ and encircled by another $V_0$-colored one.
  Applying as before a $1$ surgery we obtain that $\Zr([T'_{\emptyset}]\circ [T_{xP_0}])=(\eta \qd(0))^{-1}\Zr(S^3,o\sqcup o)=0$.
  Hence $[T_{P_0}]$ and $[T_{xP_0}]$
  are linearly independent in $\V(\Su)$ and their dual basis vectors are 
  $[T'_\emptyset]$ and $ -2[ST'_\emptyset]$, respectively.

  We finally claim that $[T_{P_0}]$ and $[T_{xP_0}]$ span the whole
  $\V(\Su)$. Indeed, by Theorem \ref{teo:skeinsurjects} $\V(\Su)$ is
  spanned by $\Skein(T,\Su)$ where $T$ is a solid torus bounding $\Su$.  
  Every element of $\Skein(T,\Su)$ is equivalent to an element whose graph  intersects  the meridian disc of $T$ transversally.  
  Furthermore, by Proposition
  \ref{P:trick} we can suppose that such elements contains an edge
  colored by a projective module.  Thus, up to skein
  equivalence, we may 
  consider elements which intersect 
   the meridian disc
  in a single edge colored by a projective module, hence by a direct sum of copies of $P_0$ as $P_0$ generates  additively the projective modules of  $\cat_{\wb 0}$. 
  Finally the
  endomorphisms $\Id$ and $x$ generate $\Hom(P_0,P_0)$ thus every skein
  is equivalent to a linear combination of $[T_{P_0}]$ and $[T_{xP_0}]$ in
  $\V(\Su)$.

\end{proof}
\begin{rem}
  The case $\coh=0$ is singular and is not dealt with 
  the Verlinde formula of Theorem \ref{P:verlinde}.
  In particular in the above case the dimension in this case is $2$
  instead that $1$ as in the generic case.
\end{rem}

\subsection{The mapping class group representation for $r=2$ and $\coh=0$.}\label{sub:mcgrep}
Recall the cylinders $L_{\La,\La'}$ and $M_f$ of Remark
\ref{rem:mcgrepres}.  Let $\Su$ be a decorated surface with Lagrangian
$\La$ and the $0$ cohomology class.  
The mapping class group ${\rm Mod}(\Su)$ is the isotopy classes of
orientation preserving diffeomorphisms of $\Su$.   
For $f$ an element of ${\rm Mod}(\Su)$, we can associate the endomorphism
$N_f:=\V(L_{f_*(\La),\La}\circ M_f)$ of $\V(\Su)$ where $f_*:H_1(\Su;\Z)\to H_1(\Su;\Z)$ is the push-forward of $f$.  The composition law
in $\Cob$ and the contribution of the signature-weight of extended cobordism to
$\V$ imply that for $f,g$ in the mapping class group, one has
$$N_f\circ N_g=\dep^{-\mu(f_*(\La),\La,g_*^{-1}(\La))}N_{fg}.$$
\begin{theo}\label{T:mcgr=2g=1}
 Let $\Su=S^1\times S^1$ and $r=2$. 
 Then the map
${\rm Mod}(\Su) \to \End_\C(\V(\Su))$ given by 
 $f\to N_f$ yields a projective representation 
  which is injective modulo the
  center.
\end{theo}
\begin{proof} 
Choose an oriented meridian $m$ and an oriented
  longitude $l$ of the oriented torus so that $m\cdot l=1$.
  Let  
 $T$ and $S$ 
  be the elements of ${\rm Mod}(\Su)$ given, respectively, by the Dehn twist along $m$ and the map sending $m$ to $l$ and $l$ to $-m$.   It is well known that 
 ${\rm Mod}(\Su)=\langle S,T| S^4=1, (TS)^3=S^2\rangle$ 
  which is isomorphic as a group to $SL(2;\Z)$.
 
  Observe that acting by  $S$ on 
  $\V(\Su)$ we need to take into account the cohomology class $\coh=0$ and the Lagrangian subspace
  $\La$ of the decorated surface $\Su$.  Indeed, if we consider the generators
  $[T_{P_0}],[T_{xP_0}]$ described in Proposition \ref{prop:torus0}
  for $\V(\Su)$ where we fix the lagrangian $\La$ to be the homology
  class $[m]$, then $S_*(\La)=\R.[l]\subset H_1(\Su;\R)$.   To correct
  this, in the definition of $N_S$ we post-compose the morphism $M_S$
  with  $L_{S_*(\La),\La}$  to get
  an endomorphisms of $\Su$. Thus, when  $S$ acts on either of the generators
  $[T_{P_0}]$ or $[T_{xP_0}]$ the overall Maslov
  correction factor is
$$-\mu(S_*(\La),\La,\Id_*^{-1}(\La))=-\mu([l],[m],[m])=0.$$
Similarly,  the correction factor for the action of $T$ is
$$ -\mu(T_*(\La),\La,\Id_*^{-1}(\La))=-\mu([m],[m],[m])=0.$$
However, in general, for the composition of combinations of $[T_{P_0}]$ and $[T_{xP_0}]$ the Maslov
  correction
  is non-zero. Thus, the
  representation of the mapping class group of the torus is
  only defined up to powers of $\dep$, and so it is projective.

Next we show that the image of $[T_{xP_0}]$ via the mapping class $S$
  is $-\frac{1}{2}[T_{P_0}]$. Indeed, as in Proposition
  \ref{prop:torus0}, let $T'_{\emptyset}$ be the empty solid torus
  negatively oriented whose meridian coincides with
  $m$. Pairing
  $N_S([T_{xP_0}])$ 
  with $[T'_{\emptyset}]\in
  \V'(\Su)$ 
   we
  get an $S^3$ containing the core of $T$ as an unknot 
 colored by $xP_0$.  As computed in the proof of Proposition \ref{prop:torus0} the  value of this manifold is $-\frac{1}{2}$.  It also follows as in the  proof of Proposition \ref{prop:torus0} that 
$\Zr([T'_{\emptyset}]\circ N_S([T_{xP_0}])=0$.  Finally, recall that the dual vector of  $ [T_{P_0}]$ is $[T'_\emptyset]$. 
Combining these results we have $N_S([T_{xP_0}])=-\frac{1}{2}[T_{P_0}]$.  Similarly, it can be shown that $ N_S([T_{P_0}])=-2[T_{xP_0}].$

  By Lemma \ref{lem:twistPj} we have the twist $\theta_{P_0}$ is equal to $\Id-2i
  x$, 
  and so $(\theta_{P_0})^n=\Id-2i nx$.  Thus, the Dehn
  twist $T$ has infinite order as 
  $$N_T^n([T_{P_0}])=[T_{(\theta_{P_0})^n}]=[T_{P_0}]-2i n[T_{xP_0}],\ {\rm  and}\ N_T^n([T_{xP_0}])=[T_{xP_0}].$$
where $T_{(\theta_{P_0})^n}$ is 
the solid torus whose core is colored by $P_0$ and contains a coupon
colored by $(\theta_{P_0})^n:P_0\to P_0$.

Hence we the obtained a projective representation of the mapping class group of the torus which in the basis $[T_{xP_0}], [T_{P_0}]$ of $\V(\Su)$ reads
$$T=\left(\begin{array}{cc}
1 & -2i\\
0 & 1
\end{array}\right), \ S=\left(\begin{array}{cc}
0 & -2\\
-\frac{1}{2} & 0
\end{array}\right).$$
This representation is only projective (indeed $(TS)^3=\sqrt{-1} S^2$) but may actually be corrected by multiplying $S$ by $-\sqrt{-1}$ to a true representation of $SL(2;\Z)$ which is conjugated (via the diagonal matrix with eigenvalues $-2\sqrt{-1},1$) to the natural representation of $SL(2;\Z)$ and is thus faithful.
Hence the projective representation of $SL(2;\Z)$ on $\V(\Su)$ is injective on $SL(2;\Z)/{\pm Id}$.
\end{proof}

\appendix

\section{Projective modules and algebraic facts}
Here we recall some properties of $\Ubar$-modules and there traces.  Except for  Lemma \ref{lem:old71} all of these results are proved in \cite{CGP3}.

\begin{prop}\label{P:proj-mod_New}
For $i=0,...,r-1$, there exists indecomposable projective modules $P_i$ of highest weight $2r-i-2$ such that any projective indecomposable
  weight module $P\in\cat_{\wb0}\cup\cat_{\wb1}$ is isomorphic to the unique
  module $\C^H_{kr}\otimes P_i$ that has the same highest weight $(k+2)r-i-2$ as $P$.  Moreover, the highest and lowest weights of $P_i$ satisfy ${\hw}(P_i)\le {\lw}(P_i)+4(r-1)$.
\end{prop}

\begin{cor} \label{C:factor V0^2_New}
  Let $\alpha\in\C\setminus\Z$.  Let $P\in\cat_{\wb 0}$ be a
  projective module, then there exist maps $f_i:P\to
  V_0\otimes\sigma^{n_i}\otimes V_0^*$,
  $g_i:V_0\otimes\sigma^{n_i}\otimes V_0^*\to P$, $f'_j:P\to
  V_{\alpha+2k_j}\otimes V_{-\alpha}$ and $g'_j:V_{\alpha+2k_j}\otimes
  V_{-\alpha}\to P$ such that
  $$\Id_P=\sum_ig_if_i+\sum_jg'_jf'_j$$
  where $n_i\in\Z$ and $k_j\in\Z\setminus r'\Z$.
\end{cor}

In \cite{GKP2}, a modified trace $\qt_P:\End_\cat(P)\to\C$ is defined for any projective module $P$.    
The invariant $F'$ of subsection \ref{ss:F'} is defined on a larger
class of graphs in $S^3$.  Namely, if a $\cat$-colored ribbon graph $T$
has at least one edge colored by a projective object $P$, then cutting
this edge yields a (1-1)-tangle whose image by the Reshetikhin-Turaev
functor $F$ is an endomorphism $f\in\End_\cat(P)$.  Then $F'$ is
defined by $F'(T)=\qt_P(f)$ where $\qt_P$ is the modified trace on
$P$.   The modified trace is non degenerate:
\begin{prop}\label{P:qtnondegen}
  Let $P\in\cat$ be a projective module and $V\in\cat$.  Then
  the pairing  
  $$
  \begin{array}{ccl}
    \Hom_\cat(V,P)\times\Hom_\cat(P,V)&\to&\C\\
    (h_1,h_2)&\mapsto&\qt_P(h_1h_2)
  \end{array}
  $$
  is non degenerate.
\end{prop}

This lemma is key to prove the $1$-surgery relation (Proposition \ref{P:1-surg}); because of its topological nature we give a proof directly here :
\begin{lemma}\label{lem:old71}
  Let $k\in\Z$ and $\alpha\in \Cp,\beta \in \C\setminus \Z$.  
  If $\alpha\in\Z$, we assume $k\in r'\Z$.  Let $\Ma$ be the
  endomorphism of $V_{\alpha+2k}\otimes V_{-\alpha}$ given by
  encircling two parallel lines with a component labeled by
  $\Omega_{\beta}=\sum_{i\in\Hr}\qd(\beta+i)V_{\beta+i}$.
  If $k=r'n\in r'\Z$ then
  $$\Ma=r^3 q^{2r'n \beta}\qd(\alpha)^{-1}\Id_{\sigma^{n}}
  \otimes(\ev_{V_\alpha} \circ\tcoev_{V_\alpha}),\text{ that is graphically}$$ 
  \[\epsh{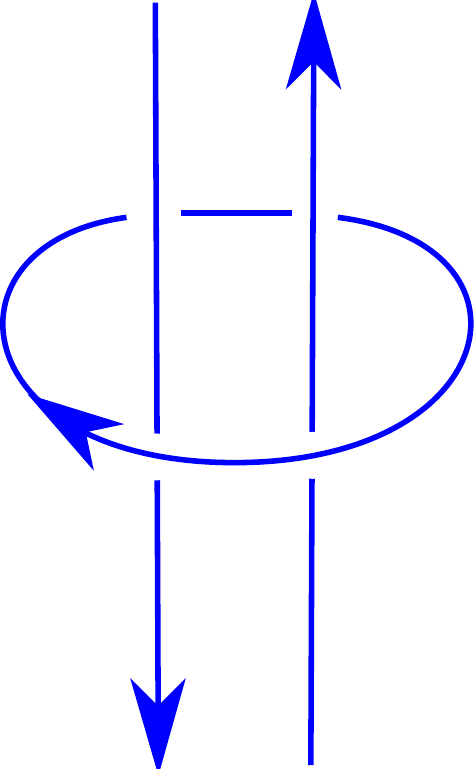}{10ex}\put(-6,18){\ms{\alpha}}
  \put(-46,-18){\ms{\alpha\!+\!2r'n}}\put(-35,2){\ms{\Omega_{\beta}}}
  \quad=\quad r^3 q^{2r'n \beta}\qd(\alpha)^{-1}\quad
  \epsh{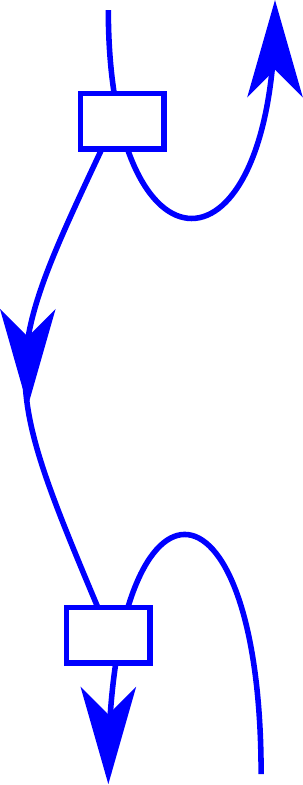}{10ex}\put(-13,2){\ms{\sigma^n}}\put(1,18){\ms{\alpha}}
  \put(-39,-18){\ms{\alpha\!+\!2r'n}}\quad\text{ with coupons colored
    by the identity,}
  \]
  and $\Ma=0$ otherwise. 
\end{lemma}
\begin{proof}
  Let $W=V_{\alpha+2k}\otimes V_{-\alpha}$.  Then $W\in
  \cat_{\underline{0}}$ and so all the weights of $W$ are in $2\Z$.
  First we show that $\Ma$ is zero unless $k=r'n$ for some $n\in \Z$.
  Let $V$ be a weight module in $\cat_{\underline 0}$.  By Lemma 4.9
  of \cite{CGP} we can do handle-slide moves on the circle component
  of the graph representing $\Ma \otimes \Id_{V}$ to obtain via
  Proposition 4.4 of \cite{CGP} the equality of morphisms:
  \begin{equation}
    \label{eq:transparent}
    c_{W,V}\circ (\Ma\otimes Id_{V})=c_{V,W}^{-1}\circ (\Ma\otimes Id_{V}).
  \end{equation}
  Note here that the Kirby color did not change after the handle
  slides because $V\in\cat_{\underline 0}$.  The braidings $c_{W,V},
  c_{V,W}^{-1}: W\otimes V\to V\otimes W$ are given by
  $$c_{W,V}=\tau\circ R =  \tau\circ q^{H\otimes H/2}(\Id \otimes \Id + 
  (q-q^{-1})E\otimes F + \cdots )$$ and
  $$c_{V,W}^{-1}=R^{-1}\circ \tau = (\Id \otimes \Id + (-q+q^{-1})E\otimes F 
  + \cdots )q^{-H\otimes H/2}\circ\tau. $$ 
  Here the dots ``$\cdots$'' are linear combination of power
  $(E\otimes F)^i$ with $i\ge2$.
  
  Let $x\in W$ be a weight vector of weight $\mu$ and set $y=\Ma(x)$.
  Let $v\in V$ be a weight vector of weight $\nu$.  Let $V'$ be the
  vector space span by $\{F^iv\}_{i \ge2}$ and $V''$ be the vector
  space span by $\{E^iv\}_{i \ge2}$.  Applying \eqref{eq:transparent}
  to $x\otimes v$ we have
  $$c_{W,V}(y\otimes v)=c_{V,W}^{-1}(y\otimes v).$$
  Now
  \begin{align}
    c_{W,V}(y\otimes v) & =  \tau\circ q^{H\otimes H/2}(y\otimes v 
    + (q-q^{-1})Ey\otimes Fv +...)\\
    & =q^{\mu\nu/2} v\otimes y + (q-q^{-1})q^{(\mu+2)(\nu-2)/2}Fv\otimes Ey+Y'
  \end{align}
  where $Y'\in V'\otimes W$. The last
  Similarly,
  \begin{align}
    c_{V,W}^{-1}(y\otimes v)&=q^{-\mu\nu/2}( v\otimes y-(q-q^{-1})Ev\otimes Fy + Y'')
  \end{align}
  where $Y''\in V''\otimes W$.  Setting the above equations equal we
  have $Fv\otimes Ey=Ev\otimes Fy=0$.  Since $V$ and $v$ are arbitrary
  we have that $Ey=Fy=0$.  So $(q-q^{-1})(EF-FE)y=(K-K^{-1})y=0$ and
  $K^2y=y$.  Thus, $\mu\in r\Z$.  
  But since we already knew $\mu\in 2\Z$ we have $\mu=2r'n$ for some
  $n\in \Z$.  Let $w$ be the generator of $\sigma^{-n}$ then $E,F,H$
  all act by zero on the vector $w\otimes y \in \sigma^{-n}\otimes W$.
  Therefore, $w\otimes y$ generates a trivial module in
  $\sigma^{-n}\otimes W$, but $\Hom(\unit, \sigma^{-n}\otimes W)\cong
  \Hom(\unit, V_{\alpha+ 2k-2r'n}\otimes V_{-\alpha})=0$ unless
  $2k-2r'n=0$.  Thus, $y=\Ma(x)=0$ unless $k=r'n$.

  As $\Hom_{U_q(sl_2)}(\sigma^{n},\sigma^{n}\otimes V_\alpha\otimes
  V_{-\alpha})\simeq\Hom_{U_q(sl_2)}(\sigma^{n}\otimes V_\alpha\otimes
  V_{-\alpha},\sigma^{n})\simeq\C$, then $f_{r'n}$ and
  $\Id_{\sigma^{n}}\otimes(\ev_{V_\alpha} \circ\tcoev_{V_\alpha})$ are
  proportional.  Let $\gamma \in\C$ be the coefficient of
  proportionality.  
  We compute $\gamma$ by comparing the modified trace i.e. the value
  of $\Nr$ on the braid closure of diagrams representing the two
  morphisms.   The modified trace of
  $\Id_{\sigma^{n}}\otimes(\ev_{V_\alpha} \circ\tcoev_{V_\alpha})$ is
  $\gamma$ times the value of $\Nr$ on the disjoint union of two
  unknot colored with $\sigma^{n}$ and $\alpha$ which is
  $(-1)^{n(r-1)}\qd(\alpha)$.  On the other hand, let $\bar{f}_{r'n}$
  be the closure of the diagram defining $\Ma$.  Since
  $V_{\alpha+2r'n}\cong V_\alpha \otimes \sigma^n$ we have that
  $\bar{f}_{r'n}$ is a connected sum of three Hopf links.  Using Relations 
  \eqref{eq:sigmabraid} 
  the strand of $\bar{f}_{r'n}$ labeled with $\sigma^n$ can be removed
  up to a coefficient.  In particular,
  $\Nr(\bar{f}_{r'n})=(q^{2r'\beta})^n(-1)^{n(r-1)}\Nr(\bar{f}_0).$
  Now, 
   $$\Nr(\bar{f}_0)=  \sum_{i\in \Hr} \qd(\beta+i)\qd(\beta+i)^{-1}(-1)^{r-1}
   rq^{\alpha(\beta +i)}(-1)^{r-1}rq^{-\alpha(\beta +i)}= \sum_{i\in \Hr}  r^2=  r^3$$
   Thus,
   $(-1)^{n(r-1)}\qd(\alpha)\gamma=(-1)^{n(r-1)}q^{2r'n\beta}
   r^3$ so $\gamma=\qd(\alpha)^{-1}q^{2r'n\beta} r^3$.
\end{proof}

Let $V,W\in \cat$ and let $S'(V,W)$ be the long Hopf tangle whose
vertical strand is colored by $W$ and whose closed component is
colored by $V$, oriented so that the linking number of the two
components is $1$.  Let 
\begin{equation}
  \label{eq:Phi}
  \Phi_{V,W}=F\bp{\epsh{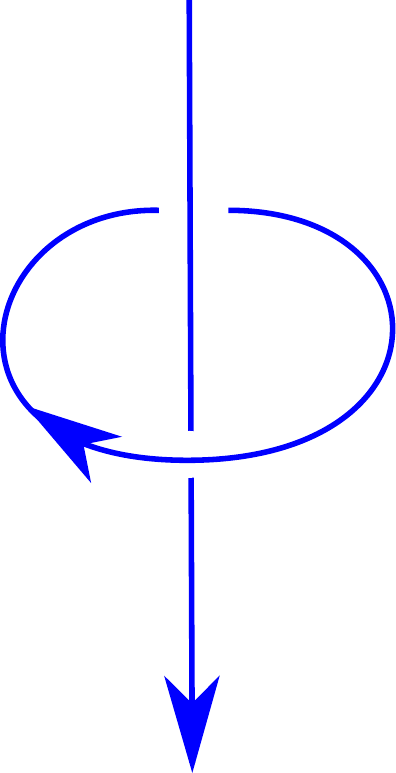}{8ex}
  \put(-8,-10){\ms{W}}\put(-22,-3){\ms{V}}}=F(S'(V,W))\in \End(W).
\end{equation}
For $j\in \{0,1,\ldots , r-2\}$ There is a nilpotent $x_j\in\End(P_j)$ determined by:
  \begin{align}
    \Phi_{V_0,P_j}&=(-1)^{r+j}2rx_j,&
    \Phi_{P_{i},P_j}&=(-1)^{i}2r(q^{(i+1)(j+1)}+q^{-(i+1)(j+1)})x_j
  \end{align}
  where $i\in \{0,1,\ldots , r-2\}$.
\begin{lemma}[General hopf links]\label{lem:generalhopflinks}
  $\forall i,j\in \{0,1,\ldots , r-2\}$ and $\forall
  \alpha,\beta\in\Cp$, one has 
  \begin{align}\label{eq:HopfSimple1}
    \Phi_{V_\beta,V_\alpha}&=\frac{(-1)^{r-1}r}{\qd(\alpha)}q^{\alpha\beta}Id_{V_\alpha}&
    \Phi_{S_i,S_j}&=(-1)^i\qx{(i+1)(j+1)}{j+1}Id_{V_j}
    \\\label{eq:HopfSimple2}
    \Phi_{S_i,V_\alpha}&=\qx{(i+1)\alpha}{\alpha}Id_{V_\alpha}&
    \Phi_{P_i,V_\alpha}
    &=(-1)^{r-1}r\frac{q^{(r-1-i)\alpha}+q^{-(r-1-i)\alpha}}{\qd(\alpha)}Id_{V_\alpha}
  \end{align}
  \begin{align}
    \Phi_{S_i,P_j}&=(-1)^i\qx{(i+1)(j+1)}{j+1}\Id_{P_j}+
    (-1)^i\frac{i\qn{(i+2)(j+1)}-(i+2)\qn{i(j+1)}}{\qn{j+1}}\,x_j.
  \end{align}
\end{lemma}

\begin{lemma}[The modified trace on $P_j$]\label{lem:modifiedtraces}
Let $\qt$ be the modified trace discussed above.  Then
$$\qd(P_i)=\qt_{P_i}(\Id_{P_i})=(-1)^{j+1}(q^{j+1}+q^{-j-1})\et\qt_{P_i}(x_j)=(-1)^{j+1}$$
\end{lemma}
\begin{lemma}[Twist on $P_j$]\label{lem:twistPj}
  The action of the twist on $P_j$ is given by 
  $$\theta_{P_j}=(-1)^jq^{\frac{j^2+2j}2}(1-(r-j-1)\qn{j+1} x_j)$$
  and in particular is not of finite order.
\end{lemma}

\begin{lemma}\label{lem:nilpotentendo}
Let $r\geq 2$ and let $V\in \cat_{\wb 0}$ be any module and for every $\alpha\in \C\setminus \Z$ let $\phi_\alpha:V\to V$ be the endomorphism represented by encircling a $V$-colored strand by a $\Omega_{\overline{\alpha}}$-colored $0$-framed meridian. Then $\phi_{\alpha}\circ\phi_{\beta}=0,\ \forall \alpha,\beta\in \C\setminus \Z$. Moreover if $V=P_{k}, k\in \{0,1,\ldots r-1\}$ then $\phi_{\alpha}=0$ unless $k=0$, in which case it is $(-1)^{k+1}x_k$.
\end{lemma}
\begin{proof}
First of all observe that one may fix a basis of $V$ such that the matrix expressing $\phi_\alpha$ depends analytically on $\alpha\in \C\setminus \Z$. Then observe that for $\beta\notin\alpha+\Z$ the morphism represented by $\phi_{\alpha}\circ \phi_\beta$ is equivalent (by sliding the $\Omega_{\overline{\beta}}$-colored over the $\Omega_{\overline{\alpha}}$-colored meridian) to that represented by a disjoint union of a $\Omega_{\overline{\beta}}$ colored $0$-framed unknot and $\phi_\alpha$ which is the zero morphism as $\Omega_\alpha$ is a projective color. By analyticity in $\alpha$ we may drop the condition on $\beta\notin\alpha+\Z$.

Let us now compute the invariant of $(S^2\times S^1,P_{k}\times S^1,\coh_{\overline{\alpha}}),\ k\in \{0,1,\ldots r-1\}$ where $\coh_{\overline{\alpha}}$ is the cohomology class whose value on the class of $\{pt\}\times S^1$ is $\overline{\alpha}$. We use the presentation by surgery on $0$-framed $\Omega_{\overline{\alpha}}$-colored unknot encircling a $P_k$ colored unknot.
It equals:
$$\Zr(S^2\times S^1,P_{k}\times S^1,\coh_{\overline{\alpha}})=\eta\lambda \sum_{l\in \Hr} \qd(\alpha+l)\qt(\Phi_{P_k,V_{\alpha+l}})$$
$$=(-1)^{r-1}\frac{r^2\eta\lambda}{\{r\alpha\}}  \sum_{l\in \Hr} \{\alpha+l\}(q^{(\alpha+l)(r-k-1)}+q^{-(\alpha+l)(r-k-1)})$$
$$=(-1)^{r-1}\frac{r^2\eta\lambda}{\{r\alpha\}}  \sum_{l\in \Hr} \{(\alpha+l)(r-k)\}-\{(\alpha+l)(r-k-2)\}$$
 $$=\frac{(-1)^{r-1}r^2\eta\lambda}{\{r\alpha\}}  \sum_{l\in \Hr} \{(\alpha+l)(r-k)\}
 =(-1)^{r-1}\frac{r^2\eta\lambda}{\{r\alpha\}} r\{r(\alpha+l)\}\delta_{k,0}=\delta_{k,0}$$
 Now, since we know that $\phi_{\alpha}$ is multiple of $x_k$ and
 the modified trace of $x_k$ is $(-1)^{k+1}$ we get the statement.
\end{proof}


\begin{thebibliography}{99}

\bibitem{AK} J. E. Andersen, and R. Kashaev - {\em A TQFT from quantum
    Teichm\"uller theory.} preprint, arXiv:1109.6295.

\bibitem{An} J. E. Andersen -{\em Asymptotic faithfulness of the
    quantum $SU(n)$ representations of the mapping class groups.}
  Ann. of Math. (2), 163(1):347–368, 2006.

\bibitem{ADO} Y. Akutsu, T. Deguchi, and T. Ohtsuki - {\em Invariants
    of colored links.} J. Knot Theory Ramifications \textbf{1} (1992),
  no. 2, 161-184.
  
\bibitem{At89} M.  Atiyah - {\em Topological quantum field theories. }
  {Inst. Hautes \'Etudes Sci. Publ. Math.} No. \textbf{68} (1988),
  175--186 (1989).

\bibitem{BB} S. Baseilhac, R. Benedetti - {\em Classical and quantum
    dilogarithmic invariants of flat $PSL(2,C)$-bundles over
    3--manifolds.}  Geom. Topol. 9 (2005), 493--569.

\bibitem{Bl} C. Blanchet - {\em Introduction to quantum invariants of
    $3$-manifolds topological quantum field theories and modular
    categories.}  Proceedings of the summer school on geometric and
  topological methods for quantum field theory, Villa de Leyva,
  Colombia, July 9-27, 2001. River Edge, NJ: World Scientific. 228-264
  (2003).

\bibitem{BEW} B.C. Berndt, R.J. Evans, K.S. Williams - {\em Gauss and
    Jacobi sums.} Canadian Mathematical Society Series of Monographs
  and Advanced Texts. A Wiley-Interscience Publication. John Wiley \&
  Sons, Inc., New York, 1998. xii+583 pp. ISBN: 0-471-12807-4

\bibitem{BHMV} C. Blanchet, N. Habegger, G. Masbaum and P. Vogel -
  {\em Topological quantum field theories derived from the Kauffman
    bracket.}  Topology 34 (1995), no. 4, 883--927.

\bibitem{BM} C. Blanchet, G. Masbaum, {\em Topological quantum field
    theories for surfaces with spin structure}, Duke Mathematical
  Journal 82 (1996), 229-267.

\bibitem{BW} F. Bonahon, H. Wong - {\em Representations of the
    Kauffman skein algebra I: invariants and miraculous
    cancellations}, preprint arXiv:1206.1638.


\bibitem{CP} V. Chari, A. Premet - {\em Indecomposable restricted
    representations of quantum $sl_2$.} Publ. Res. Inst. Math. Sci. 30
  (1994), no. 2, 335--352.

\bibitem{CGP} F. Costantino, N. Geer, B. Patureau-Mirand - {\em
    Quantum invariants of 3--manifolds via link surgery presentations
    and non-semi-simple categories.}, to appear on Journal of Topology.

\bibitem{CGP2} F. Costantino, N. Geer, B. Patureau-Mirand - {\em
   Relations between Witten-Reshetikhin-Turaev invariants and non-semi-simple ${\mathfrak sl}(2)$ $3$-manifold invariants},
   15 pages, arXiv:1310.2735 (2013) .

\bibitem{CGP3} F. Costantino, N. Geer, B. Patureau-Mirand - {\em Some remarks on a quantization of $\mathfrak{sl}_2$}, in preparation.
   
   
\bibitem{CM} F. Costantino, J. Murakami - {\em On $SL(2,\C)$ quantum
    $6j$-symbols and its relation to the hyperbolic volume.}
  arXiv:1005.4277, to appear in Quantum Topology (2012).

\bibitem{Co} F. Costantino - {\em $6j$-symbols, hyperbolic structures
    and the Volume Conjecture}, Geometry \& Topology 11 (2007),
  1831--1853.

\bibitem{CT} F. Costantino, D. Thurston - {\em 3--manifolds efficiently bound
    4-manifolds.}, J. Topol, 1 (2008), no. 3, 703--745.

\bibitem{DK} C. De Concini, V.G. Kac - {\em Representations of quantum
    groups at roots of $1$.} In {\em Operator algebras, unitary
    representations, enveloping algebras, and invariant theory.}
  (Paris, 1989), 471--506, Progr. Math., 92, Birkhauser Boston, 1990.

\bibitem{DM} F. Deloup, G. Massuyeau - {\em Reidemeister-Turaev
    torsion modulo one of rational homology three-spheres.}
  Geom. Topol. 7 (2003), 773--787.
  
\bibitem{DM2} F. Deloup, G. Massuyeau - {\emph{Quadratic functions and
      complex spin structures on three-manifolds}.}  Topology 44,
  No. 3 (2005), 509-555.

\bibitem{FN} C. Frohman, A. Nicas -{\em The Alexander Polynomial via
    topological quantum field theory.}  In: Differential Geometry,
  Global Analysis, and Topology, CMS Conf. Proc. 12,
  Amer. Math. Soc. Providence, RI, 1991, 27–40.

\bibitem{FGST1} B.L. Fe{\u\i}gin, A.M.Ga{\u\i}nutdinov,
  A.M. Semikhatov, I.Yu. Tipunin - {\em Modular group representations
    and fusion in logarithmic conformal field theories and in the
    quantum group center.}  Comm. Math. Phys. 06 (2006); 265 (1),
  47-93.

\bibitem{FGST} B.L. Fe{\u\i}gin, A.M.Ga{\u\i}nutdinov,
  A.M. Semikhatov, I.Yu. Tipunin - {\em The Kazhdan-Lusztig
    correspondence for the representation category of the triplet W
    -algebra in logarithmic conformal field theories.}
  Teoret. Mat. Fiz. 148 (2006), no. 3, 398--427; translation in
  Theoret. and Math. Phys. 148 (2006), no. 3, 1210--1235.

\bibitem{Franz} W. Franz - {\em \"Uber die Torsion einer \"Uberdeckung}. J. Reine Angew. Math., 173:245–254, 1935.

\bibitem{FWW} M.H. Freedman, K. Walker, Z. Wang -{\em Quantum SU(2)
    faithfully detects mapping class groups modulo center.} Geometry
  \& Topology Volume 6 (2002) 523--539.

\bibitem{GKT} N. Geer, R. Kashaev, V. Turaev - {\em Tetrahedral forms
    in monoidal categories and 3--manifold invariants.}  To appear in
  {Crelle's Journal}, arXiv:1008.3103.

\bibitem{GKP1} N. Geer, J. Kujawa, B. Patureau-Mirand - {\em
    Generalized trace and modified dimension functions on ribbon
    categories.} Selecta Math. (N.S.) 17 (2011), no. 2, 453--504.

\bibitem{GKP2} N. Geer, J. Kujawa, B. Patureau-Mirand - {\em
    Ambidextrous objects and trace functions for nonsemisimple
    categories.} Proc. Amer. Math. Soc., to appear, arXiv:1106.4477.

\bibitem{GP08} N. Geer, B. Patureau-Mirand - {\em Multivariable link
    invariants arising from Lie superalgebras of type I}, {J. of Knot Theory
    and its Ramifications} {19} (2010), no. 1, 93--115.

\bibitem{GP} N. Geer, B. Patureau-Mirand - {\em Polynomial 6j-Symbols and
    States Sums.}  Algebraic \& Geometric Topology 11 (2011) 1821--1860.

\bibitem{GP2} N. Geer, B. Patureau-Mirand - {\em Topological
    invariants from non-restricted quantum groups.}  Algebraic \&
  Geometric Topology 13 (2013) 3305--3363.

\bibitem{GPT} N. Geer, B. Patureau-Mirand, V. Turaev - {\em Modified quantum
    dimensions and re-normalized link invariants.} Compos. Math.  {145}
  (2009), no. 1, 196--212.

\bibitem{GPT2} N. Geer, B. Patureau-Mirand, V. Turaev - {\em Modified
    6j-Symbols and 3--manifold invariants.}  Advances in Mathematics
  Volume 228, Issue 2, (2011), 1163--1202.
  
\bibitem{GS} R. Gompf, A. Stipsicz - {\em $4$-manifolds and Kirby
    calculus.}  Amer. Math. Soc. Providence (RI), 1999, xv+558 pages.

\bibitem{Gu} S. Gukov - {\em Three-dimensional quantum gravity,
  Chern-Simons theory, and the A-polynomial.} { Comm. Math. Phys.}
  {255} (2005), no. 3, 577--627.

\bibitem{Hat} A. Hatcher - {\em Algebraic Topology.} Cambridge
  University Press (2002).

\bibitem{He} M. Hennings - {\em Invariants of links and 3--manifolds
    obtained from Hopf algebras.} J.  London Math. Soc. (2),
  54(3):594--624, 1996.

\bibitem{Je} L.C. Jeffrey - {\em Chern-Simons-Witten invariants of
    lens spaces and torus bundles, and the semiclassical
    approximation.}  Comm. Math. Phys. 147 (1992), no. 3, 563--604.

\bibitem{Jo} D. L. Johnson - {\em The structure of the Torelli group. I. A finite set of generators for $\T$} , 
	Ann. of Math. (2) 118 (1983), no. 3, 423--442.
	
\bibitem{Jo2} D. L. Johnson - {\em Homeomorphisms of a surface which act trivially on homology.}
 Proc. Amer. Math. Soc., 75(1):119–125, 1979.	
 
\bibitem{Ka} R. Kashaev - {\em Quantum hyperbolic invariants of
    knots.}  Discrete integrable geometry and physics (Vienna, 1996),
  343--359, Oxford Lecture Ser. Math. Appl., 16, Oxford Univ. Press,
  New York, 1999.

\bibitem{Kv} R. Kashaev - {\em A link invariant from quantum
  dilogarithm.} {Modern Phys. Lett. A} {10} (1995), no. 19,
  1409--1418.

\bibitem{KR} R. Kashaev, N. Reshetikhin - {\em Invariants of tangles
    with flat connections in their complements.}  Graphs and patterns
  in mathematics and theoretical physics, 151--172,
  Proc. Sympos. Pure Math., 73, Amer. Math. Soc., Providence, RI,
  2005.

\bibitem{Ke} T. Kerler - {\em Homology TQFT’s and the
    Alexander–Reidemeister Invariant of 3-Manifolds via Hopf Algebras
    and Skein Theory.}  Canad. J. Math. Vol. 55 (4), 2003 pp. 766–821.

\bibitem{KL} T. Kerler, V. Lyubashenko, {\em Non-semisimple
    Topological Quantum Field Theories for 3-Manifolds with Corners},
  LNM 1765, Springer Verlag, 2001.

\bibitem{Ki} R. Kirby - {\em A calculus for framed links.}
  Invent. Math. 45 (1978), 35--56.

\bibitem{Ku} G. Kuperberg - {\em Involutory Hopf algebras and
    3--manifold invariants.} Internat. J. Math. 2 (1991), 41--66.

\bibitem{LM} V. Lyubashenko, S. Majid - {\em Braided groups and quantum Fourier transform.}
J. Algebra 166 (1994) 506–528.

\bibitem{LT} T.Q. Le, A.T. Tran - {\em On the volume conjecture for
    cables of knots.}  Journal of Knot Theory and its Ramifications
  (2009), vol. 19, issue 12.

\bibitem{Le} T.Q. Le - {\em On Kauffman Bracket Skein Modules at Root
    of Unity.} preprint, arXiv:1312.3705.

\bibitem{Ma} G. Massuyeau - {\em An introduction to the abelian
    Reidemeister torsion of three-dimensional manifolds.} Annales
  mathématiques Blaise Pascal 18.1 (2011) 61--140.

\bibitem{Me} G.Mess - {\em The Torelli groups for genus 2 and 3 surfaces}, 
Topology 31 (1992), no.4, 775--790.

\bibitem{MM} H. Murakami, J. Murakami - {\em The colored Jones
    polynomials and the simplicial volume of a knot.}  Acta Math. 186
  (2001), no.1, 85--104.

\bibitem{jM2} J. Murakami - {\em The multi-variable Alexander
    polynomial and a one-parameter family of representations of
    $U_q({\mathfrak sl}(2,\C))$ at $q^2=-1$.} Quantum groups
  (Leningrad, 1990), 350--353, Lecture Notes in Math., 1510, Springer,
  Berlin, 1992.

\bibitem{jM} J. Murakami - {\em Colored Alexander Invariants and
    Cone-Manifolds.} Osaka J. Math. 45 no.2 (2008), 541--564.

\bibitem{MN} J. Murakami, K. Nagatomo - {\em Logarithmic knots
    invariants arising from restricted quantum groups.}
  Internat. J. Math. 19 (2008), no. 10, 1203--1213.

\bibitem{Oh} T. Ohtsuki - {\em Quantum invariants. A study of knots,
    3--manifolds, and their sets.} Series on Knots and Everything,
  \textbf{29}. World Scientific Publishing Co., Inc., River Edge, NJ, 2002.
    
\bibitem{Pi} R.S. Pierce - {\em Associative algebras.}  Graduate Texts
  in Mathematics, 88. Studies in the History of Modern Science,
  9. Springer-Verlag, New York-Berlin, 1982.

\bibitem{PS} V. V. Prasolov, A. B. Sossinsky - {\em Knots, Links,
    Braids and 3-Manifolds}, Translations of Mathematical Monographs,
  vol. 154, American Mathematical Society, Providence, RI, 1997.

\bibitem{Pr} J\'ozef H. Przytycki - {\em Skein modules of
    3-manifolds.} Bull. Polish Acad. Sci. Math., 39(1-2) 91--100, 1991.

\bibitem{Re} K. Reidemeister - {\em Homotopieringe und Linsenr\"aume} Hamburger Abhandl. 11 (1935), 102--109. 

\bibitem{RT0} N. Reshetikhin, V.G. Turaev - {\em Ribbon graphs and
    their invariants derived from quantum groups.}
  Comm. Math. Phys. 127 (1990), no. 1, 1--26.

\bibitem{RT} N. Reshetikhin, V.G. Turaev - {\em Invariants of
    3--manifolds via link polynomials and quantum groups.}
  Invent. Math. 103 (1991), no. 3, 547--597.
  
\bibitem{Ro} J. Roberts - {\em Kirby calculus in manifolds with
    boundary.}  Turkish J. Math. 21 (1997), 111--117.

\bibitem{Tu3} V.G. Turaev - {\em Reidemeister torsion in knot theory.}
  Uspekhi Mat. Nauk, 41(1(247)):97--147, 240, 1986.

\bibitem{Tu1} V.G. Turaev - {\em The Conway and Kauffman modules of a
    solid torus.}
  Zap. Nauchn. Sem. Leningrad. Otdel. Mat. Inst. Steklov. (LOMI),
  167(Issled. Topol. 6) 79--89, 190, 1988.
  
\bibitem{Tu} V.G. Turaev - {\em Quantum invariants of knots and
    3--manifolds.}  de Gruyter Studies in Mathematics, 18. Walter de
  Gruyter \& Co., Berlin, (1994).

\bibitem{Tu2} V.G. Turaev - {\em Torsions of 3-dimensional manifolds.}
  Progress in Mathematics, 208. Birkh\"auser Verlag, Basel, (2002).
  
\bibitem{TuHQFT} V.G. Turaev - {\em Homotopy quantum field theory.}  EMS Tracts in Mathematics 10, (2010).

\bibitem{Vz} A. Virelizier - {\em Kirby elements and quantum
    invariants.} Proc. London Math. Soc. 93 (2006), 474--514

\bibitem{Vi} O. Viro - {\em Quantum relatives of the Alexander
    polynomial.} Algebra i Analiz 18 (2006), no. 3, 63--157;
  translation in St. Petersburg Math. J. 18 (2007), no. 3, 391--457.

\bibitem{Wi} E. Witten - {\em Quantum field theory and the Jones polynomial},
 Comm. Math. Phys.  121  (1989),  no. 3, 351--399.

\end{thebibliography}
\end{document}